\documentclass[11pt]{amsart}

\usepackage[margin=1.2in]{geometry}

\usepackage{amssymb,amsmath,amscd,amsthm,xy,enumitem,pstricks,pst-node}
\usepackage{amsxtra}
\xyoption{all}
 
\newtheorem{thm}{Theorem}[section]
\newtheorem{prop}[thm]{Proposition}
\newtheorem{lem}[thm]{Lemma}
\newtheorem{cor}[thm]{Corollary}

\theoremstyle{definition}
\newtheorem{definition}[thm]{Definition}
\newtheorem{example}[thm]{Example}

\theoremstyle{remark}
\newtheorem{remark}[thm]{Remark}

\numberwithin{equation}{section}

\def\lra{\longrightarrow}
\def\Lra{\Longrightarrow}

\def\lhra{\lhook\joinrel\relbar\joinrel\rightarrow}

\def\BE#1{\begin{equation}\label{#1}}
\def\EE{\end{equation}}
\def\lr#1{\langle#1\rangle}
\def\blr#1{\big\langle#1\big\rangle}
\def\ti#1{\tilde{#1}}
\def\wt#1{\widetilde{#1}}
\def\ov#1{\overline{#1}}
\def\eref#1{(\ref{#1})}
\def\tn#1{\textnormal{#1}}

\def\De{\Delta}
\def\Ga{\Gamma}
\def\La{\Lambda}

\def\Si{\Sigma}

\def\al{\alpha}
\def\be{\beta}
\def\de{\delta}
\def\eps{\epsilon}
\def\gm{\gamma}
\def\io{\iota}
\def\ka{\kappa}
\def\na{\nabla}

\def\si{\sigma}
\def\th{\theta}
\def\ve{\varepsilon}
\def\vph{\varphi}
\def\vp{\varpi}
\def\vt{\vartheta}
\def\ze{\zeta}

\def\prt{\partial}

\def\fa{\mathfrak a}
\def\fB{\mathfrak B}
\def\bB{\mathbb B}
\def\C{\mathbb C}  
\def\cD{\mathcal{D}}
\def\bE{\mathbb E}
\def\cG{\mathcal G}
\def\bH{\mathbb H}
\def\cH{\mathcal H}
\def\cJ{\mathcal J}

\def\bL{\mathbb L}
\def\cL{\mathcal L}

\def\cM{\mathcal M}
\def\mf{\mathfrak M}

\def\cO{\mathcal O}
\def\P{\mathbb P}
\def\R{\mathbb{R}}

\def\cV{\mathcal V}
\def\bX{\mathbb X}
\def\Z{\mathbb{Z}}
\def\cZ{\mathcal Z}

\def\1{\mathbf 1}
\def\a{\mathbf a}
\def\b{\mathbf b}
\def\ff{\mathfrak f}
\def\k{\mathbf k}
\def\x{\mathbf x}

\def\oI{\mathring{\mathbb{I}}}

\def\bp{\bar{\partial}} 

\def\aut{\textnormal{Aut}}
\def\Cyl{\textnormal{Cyl}}
\def\cok{\textnormal{cok}}
\def\tnd{\textnormal{d}}

\def\Ext{\textnormal{Ext}}
\def\ev{\textnormal{ev}}
\def\GL{\textnormal{GL}}
\def\Hol{\textnormal{Hol}}
\def\Hom{\textnormal{Hom}}
\def\id{\textnormal{id}}
\def\Id{\textnormal{Id}}
\def\Im{\textnormal{Im}}
\def\ind{\textnormal{ind}}

\def\pt{\textnormal{pt}}
\def\rk{\textnormal{rk}}

\def\cP{\mathcal P}
\def\PGL{\textnormal{PGL}}
\def\Re{\textnormal{Re}}
\def\SO{\textnormal{SO}}
\def\Sym{\textnormal{Sym}}
\def\top{\textnormal{top}}

\def\bI{\mathbb I}
\def\fI{\mathfrak i}
\def\fJ{\mathfrak j}
\def\ne{\textnormal{e}}

\def\dbar{\bar\partial}
\def\eset{\emptyset}
\def\i{\infty}

\begin{document}

\title[Moduli space of maps with crosscaps]{The moduli space of maps with crosscaps:\\ 
Fredholm theory and orientability}
\author[Penka Georgieva and Aleksey Zinger]{Penka Georgieva and Aleksey Zinger$^*$} 
\thanks{$^*$Partially supported by the IAS Fund for Math and NSF grants DMS 0635607 and 0846978}
\address{Department of Mathematics, Princeton University, Princeton, NJ 08544}
\email{pgeorgie@math.princeton.edu}
\address{Department of Mathematics, SUNY Stony Brook, Stony Brook, NY 11790}
\email{azinger@math.sunysb.edu}

\date{\today}

\begin{abstract} 
Just as a symmetric surface with separating fixed locus halves into two 
oriented bordered surfaces, 
an arbitrary symmetric surface halves into two oriented symmetric half-surfaces,
i.e.~surfaces with crosscaps.
Motivated in part by the string theory view of real Gromov-Witten invariants,
we introduce moduli spaces of maps from surfaces with crosscaps, 
develop the relevant Fredholm theory, and resolve the orientability problem
in this setting.
In particular,
we give an explicit formula for the holonomy of the orientation bundle of a family of 
real Cauchy-Riemann operators over Riemann surfaces with crosscaps. 
Special cases of our formulas are closely related to the orientability question for the space of
real maps from symmetric Riemann surfaces to an almost complex manifold with 
an anti-complex involution
and in fact resolve this question in genus~0.
In particular, we show that the moduli space of real $J$-holomorphic maps from
the sphere with a fixed-point free involution to a simply connected almost 
complex manifold with an even canonical class is orientable.
In a sequel, we use the results of this paper to obtain 
a similar orientability statement for genus~1 real~maps.
\end{abstract}

\maketitle

\tableofcontents

\section{Introduction}
\label{intro_sec}

\noindent
The theory of $J$-holomorphic maps plays a prominent role in symplectic topology,
algebraic geometry, and string theory.
The foundational work of~\cite{Gr,McSa94,RT,FO,LT} has 
established the theory of (closed) Gromov-Witten invariants,
i.e.~counts of $J$-holomorphic maps from closed Riemann surfaces to symplectic manifolds.
In contrast, the theory of open and real Gromov-Witten invariants, i.e.~counts of
$J$-holomorphic maps from bordered Riemann surfaces with boundary mapping 
to a Lagrangian submanifold and of $J$-holomorphic maps from symmetric
Riemann surfaces commuting with the involutions on the domain and the target, 
has been under development over the past 10-15 years and still is today.\\

\noindent
The two main obstacles to defining the open invariants are the potential non-orientability 
of the moduli space
and the existence of real codimension-one boundary strata.
The orientability problem in open Gromov-Witten theory is addressed by
the first author in~\cite{Ge}.
Some approaches \cite{Melissa, psw, Teh} to dealing with the codimension one boundary
have raised the issue of orientability in
real Gromov-Witten theory.
Symmetric Riemann surfaces, however, have convoluted degenerations, making
the orientability of their moduli spaces difficult to study.
Physical considerations \cite{SV,AAHV,Wal}
suggest that oriented surfaces with crosscaps provide 
a suitable replacement for symmetric Riemann surfaces in real Gromov-Witten theory.
In this paper, we introduce moduli spaces of \hbox{$J$-holomorphic} maps from 
oriented surfaces with crosscaps, develop the necessary Fredholm theory,
and study the orientability of these moduli spaces.
In particular, we combine the principles of~\cite{Ge} 
with equivariant cohomology and give an explicit criterion specifying whether the determinant
line bundle of a loop of real Cauchy-Riemann operators over Riemann surfaces with crosscaps
is trivial.
As explained after Corollary~\ref{etaorient_cor} and in~\cite{GZ}, 
this last issue is related to the orientability problem in real Gromov-Witten theory
via the doubling constructions of~\eref{dblSi_e} and Section~\ref{double_sec}.
In a future paper, we will study compactifications of the moduli spaces of maps with crosscaps
and use them to define real Gromov-Witten invariants in the style of~\cite{Wal}.\\   

\noindent
A \textsf{symmetric surface} $(\hat\Si,\si)$ consists of a closed connected oriented smooth 
surface~$\hat\Si$ 
(manifold of real dimension~2) and an orientation-reversing involution 
$\si\!:\hat\Si\!\lra\!\hat\Si$.
Every anti-holomorphic involution~$\si$ on $\hat\Si\!=\!\P^1$ such that $\P^1/\si$ 
is not orientable is conjugate~to
\BE{etadfn_e}\eta\!:\P^1\lra\P^1,\qquad[u,v]\lra[-\bar{v},\bar{u}].\EE
An approach to orienting indices of real Cauchy-Riemann operators 
on real bundle pairs over~$(\P^1,\eta)$ is introduced in \cite[Section~2.1]{Teh}.
We reinterpret this construction in terms of real Cauchy-Riemann operators on
Riemann surfaces with orientation-preserving involutions on the boundary components
in a way that reduces to the orienting construction of \cite[Proposition~8.1.4]{FOOO}
whenever the boundary involutions are trivial.
This allows us to extend the principles used in~\cite{Ge}, 
which treats the case with trivial boundary involutions, to the case with any number 
of crosscaps, corresponding to the nontrivial boundary involutions.
Theorem \ref{main_thm} below specifies whether the index bundle of a loop of 
real Cauchy-Riemann operators over Riemann surfaces with 
orientation-preserving involutions on the boundary components is trivial.
As a corollary, we show that the moduli space of real $J$-holomorphic maps from $(\P^1,\eta)$
to a simply connected almost 
complex manifold with an even canonical bundle is orientable; see Corollary~\ref{etaorient_cor2}.
As another corollary, we conclude that the local system of orientations on the moduli space 
of $J$-holomorphic maps from bordered Riemann surfaces  
that commute with the involutions on the boundary and on the target
is  isomorphic to the pull-back of a local system defined on 
a product of the equivariant free loop space of the target, of
the fixed point locus of the involution on the target, 
and of its free loop space; see Corollary~\ref{ls_cor}. 
Along the way, we establish the necessary Fredholm theory for bordered surfaces
with crosscaps, discuss topological issues that crosscaps introduce,
and include examples illustrating a number of subtle points.
In~\cite{GZ}, we built on the results of this paper to orient moduli spaces of real genus~1 maps;
this is a step in our project to construct real Gromov-Witten invariants
in positive genera.\\

\noindent
An \textsf{involution} on a topological space (resp.~smooth manifold) $M$ 
is a homeomorphism (resp.~diffeomorphism) 
$\phi\!:M\!\lra\!M$ such that $\phi\!\circ\!\phi\!=\!\id_M$;
in particular, the identity map on~$M$ is an involution.
Let
$$M^{\phi}=\big\{x\!\in\!M\!:~\phi(x)\!=\!x\big\}$$
denote the fixed locus.
An involution~$\phi$ determines an action of~$\Z_2$ on~$M$;
we denote~by $H^*_{\phi}(M)$, $H_*^{\phi}(M)$, and $H_*^{\phi}(M;\Z)$
the corresponding $\Z_2$-equivariant cohomology and homology of~$M$
with $\Z_2$-coefficients and 
the $\Z_2$-equivariant homology of~$M$ with~$\Z$-coefficients, respectively;
see Section~\ref{equivcoh_sec}.
If the fixed-point locus of~$\phi$ is empty, there are canonical isomorphisms
$$H^*_{\phi}(M)\approx H^*(M/\Z_2), \qquad H_*^{\phi}(M)\approx H_*(M/\Z_2), 
\qquad H_*^{\phi}(M;\Z)\approx H_*(M/\Z_2;\Z).$$
If in addition $M$ is a compact manifold (and thus so is $M/\Z_2$),
we denote by $[M]^{\phi}\!\in\!H_*^{\phi}(M)$ the fundamental homology class of~$M/\Z_2$
with $\Z_2$-coefficients.
A \textsf{conjugation} on a complex vector bundle $V\!\lra\!M$ 
\textsf{lifting} an involution~$\phi$ is a vector bundle homomorphism 
$\ti\phi\!:V\!\lra\!V$ covering~$\phi$ (or equivalently 
a vector bundle homomorphism  $\ti\phi\!:V\!\lra\!\phi^*V$ covering~$\id_M$)
such that the restriction of~$\ti\phi$ to each fiber is anti-complex linear
and $\ti\phi\!\circ\!\ti\phi\!=\!\id_V$.
We will call the conjugation
$$\ti\phi_n\!: M\!\times\!\C^n\lra M\!\times\!\C^n, \qquad 
(x,v)\lra\big(\phi(x),\bar{v}\big)\quad\forall\,(x,v)\!\in\!M\!\times\!\C,$$
\textsf{the trivial lift of~$\phi$}.
For any conjugation~$\ti\phi$ in $V\!\lra\!M$ lifting~$\phi$, 
$V^{\ti\phi}\!\lra\!M^{\phi}$ is a maximal totally real subbundle of~$V|_{M^{\phi}}$.
We denote~by
$$\La_{\C}^{\top}(V,\ti\phi)=(\La_{\C}^{\top}V,\La_{\C}^{\top}\ti\phi)$$
the top exterior power of $V$ over $\C$ with the induced conjugation and~by
$$w_i^{\ti\phi}(V)\in H^i_{\phi}(M)$$
the $i$-th $\Z_2$-equivariant Stiefel-Whitney class of~$V$.
Moreover,
if $M$ is a manifold, possibly with boundary, or a (possibly nodal) surface, 
and $\phi$ is an involution on a submanifold $M'\!\subseteq\!M$,
a \textsf{real bundle pair} $(V,\ti\phi)\!\lra\!(M,\phi)$   
consists of a complex vector bundle $V\!\lra\!M$ and a conjugation~$\ti\phi$ on $V|_{M'}$
lifting~$\phi$.\\

\noindent
A \textsf{boundary involution} on a surface~$\Si$ with boundary $\prt\Si$ is 
an orientation-preserving involution~$c$ preserving each component of~$\prt\Si$.
The restriction of such an involution to a boundary component 
is either the identity or given~by 
\BE{antip_e}\fa\!: S^1\lra S^1, \qquad z\lra -z \quad\forall~z\!\in\!S^1\subset\C,\EE
for a suitable identification $(\prt\Si)_i\!\approx\!S^1$;
the latter type of boundary structure is called \textsf{crosscap} 
in the string theory literature.
We define
$$c_i=c|_{(\prt\Si)_i}, \qquad 
|c_i|= \begin{cases} 0,&\hbox{if}~c_i=\id;\\ 1,&\hbox{otherwise};\end{cases}
\qquad
|c|_k=\big|\{(\prt\Si)_i\!\subset\!\Si\!:\,|c_i|\!=\!k\}\big|\quad k=0,1.$$
Thus, $|c|_0$ is the number of standard boundary components of $(\Si,\prt\Si)$  
and $|c|_1$ is the number of crosscaps.\\

\noindent
An \textsf{oriented symmetric half-surface} (or simply \textsf{oriented sh-surface}) 
is a pair $(\Si,c)$ consisting of
an oriented bordered smooth surface~$\Si$ and a boundary involution~$c$ on~$\Si$. 
Such a pair doubles to a symmetric surface~$(\hat\Si,\hat{c})$; see~\eref{dblSi_e}.
We denote by $\cJ_{\Si}$ the space of all complex structures on~$\Si$ compatible
with the orientation 
and by $\cJ_c$ the subspace of $\cJ_{\Si}$ consisting of the complex structures~$\fJ$
so that $c$ is real-analytic with respect to~$\fJ$; see Section~\ref{double_sec}.
In the standard case of open Gromov-Witten theory, $c\!=\!\id_{\prt\Si}$ and 
$\cJ_c\!=\!\cJ_{\Si}$.\\

\noindent
A \textsf{real Cauchy-Riemann operator on a real bundle pair $(V,\ti{c})\!\lra\!(\Si,c)$},  
where $(\Si,c)$ is an oriented sh-surface, is a linear map of the~form
\BE{CRdfn_e}\begin{split}
D=\bp\!+\!A\!: \Ga(\Si;V)^{\ti{c}}
\equiv&\big\{\xi\!\in\!\Ga(\Si;V)\!:\,\xi\!\circ\!c\!=\!\ti{c}\!\circ\!\xi|_{\prt\Si}\big\}\\
&\hspace{.5in}\lra
\Ga_{\fJ}^{0,1}(\Si;V)\equiv\Ga\big(\Si;(T^*\Si,\fJ)^{0,1}\!\otimes_{\C}\!V\big),
\end{split}\EE
where $\bp$ is the holomorphic $\bp$-operator for some $\fJ\!\in\!\cJ_{\Si}$
and a holomorphic structure in~$V$ and  
$$A\in\Ga\big(\Si;\Hom_{\R}(V,(T^*\Si,\fJ)^{0,1}\!\otimes_{\C}\!V) \big)$$ 
is a zeroth-order deformation term. 
A real Cauchy-Riemann operator on a real bundle pair
need not be Fredholm in the appropriate completions;
see Remark~\ref{notFred_rem}.
However, it is Fredholm if the boundary involution~$c$ is real-analytic with respect 
to~$\fJ$; see Proposition~\ref{etaFred_prp}.\\

\noindent
Let $\bI=[0,1]$. 
Given an orientation-preserving diffeomorphism $\psi\!:\Si\!\lra\!\Si$, let
$$(M_{\psi},\prt M_{\psi})=\big(\bI\!\times\!(\Si,\prt\Si_{\psi})\big)/(1,x)\sim(0,\psi(x))$$
be the mapping torus of $\psi$ and $\pi\!: M_{\psi}\!\lra\!S^1$ be the projection map.
For each $t\!\in\!S^1$, let $\Si_t=\pi^{-1}(t)$ be the fiber over~$t$.
An involution~$c$ on~$\prt\Si$ commuting with~$\psi$ induces a fiber-preserving involution
on $\prt M_{\psi}$, which we continue to denote by~$c$.
In such a case, a continuous family of real Cauchy-Riemann operators on 
a real bundle pair $(V,\ti{c})$ over~$(M_{\psi},c)$ is a collection of real Cauchy-Riemann operators
$$D_t\!: \Ga(\Si_t;V|_{\Si_t})^{\ti{c}}\lra  \Ga_{\fJ_t}^{0,1}(\Si_t;V|_{\Si_t})$$
which varies continuously with $t\!\in\!S^1$. 
If $\fJ_t\!\in\!\cJ_c$, so that $D_t$ is Fredholm,
we denote by $\det D\lra\!S^1$ the determinant line bundle corresponding to this family; 
see \cite[Section~A.2]{MS} and \cite{detLB} for a construction.

\begin{thm}\label{main_thm}
Let $(\Si,c)$ be an oriented sh-surface, 
$\psi\!:\Si\!\lra\!\Si$ be a diffeomorphism preserving the orientation and each boundary component
and commuting with~$c$ when restricted to~$\prt\Si$,
and $(V,\ti{c})$ be a real bundle pair over~$(M_{\psi},c)$.
For each boundary component  $(\prt\Si)_i$ of~$\Si$,
choose a section $\al_i$ of 
$$(\prt M_{\psi})_i\equiv M_{\psi}|_{(\prt\Si)_i}\lra S^1.$$ 
If $D$ is a continuous family of real Cauchy-Riemann operators on $(V,\ti{c})$
such that each $D_t$ is compatible with some $\fJ_t\!\in\!\cJ_c$, then
\BE{mainthm_e}\begin{split}
\blr{w_1(\det D), S^1}&= \sum_{|c_i|=0}\Big(\!\big(\blr{w_1(V^{\ti{c}}),(\prt\Si)_i}+1\big)
\lr{w_1(V^{\ti{c}}),[\al_i]}+ \blr{w_2(V^{\ti{c}}),(\partial M_\psi)_i}\!\Big)\\
&\quad+\sum_{|c_i|=1}  \blr{w_2^{\La_{\C}^{\top}\ti{c}}(\La_{\C}^{\top}V),[(\prt M_{\psi})_i]^c},
\end{split}\EE
where the sums are taken over 
the connected components $(\prt M_{\psi})_i$ of~$\prt M_{\psi}$.
\end{thm}

\noindent
This theorem extends \cite[Theorem~1.1]{Ge} to bordered Riemann surfaces with crosscaps
and is the key step to the remaining results in this paper,
analogously to \cite[Theorem~1.1]{Ge} being the key step to the remaining results in~\cite{Ge}.
By Lemma~\ref{Latop_lmm}, $w_2^{\La_{\C}^{\top}\ti{c}}(\La_{\C}^{\top}V)$ 
in~\eref{mainthm_e} can be replaced by the simpler looking~$w_2^{\ti{c}}(V)$.
However, as \cite[Section~2.1]{Teh} suggests, $\La_{\C}^{\top}(V,\ti{c})$
is in fact the simpler object to work with.
By Proposition~\ref{tensor_prp},
\BE{w2sum_e}
w_2^{\La_{\C}^{\top}(\ti\phi_1\oplus\ti\phi_2)}\big(\La_{\C}^{\top}(V_1\!\oplus\!V_2)\big)
=w_2^{\La_{\C}^{\top}\ti\phi_1}(\La_{\C}^{\top}V_1)
+w_2^{\La_{\C}^{\top}\ti\phi_2}(\La_{\C}^{\top}V_2)\EE
for all real bundle pairs $(V_1,\ti\phi_1),(V_2,\ti\phi_2)\!\lra\!(M,\phi)$, but 
in general
$$w_2^{\ti\phi_1\oplus\ti\phi_2}(V_1\!\oplus\!V_2)
\neq w_2^{\ti\phi_1}(V_1) +w_2^{\ti\phi_2}(V_2).$$
The last equality fails even for the trivial rank~1 real bundle pairs 
$(V_i,\ti\phi_i)\!\lra\!(M,\phi)$ over the Klein bottle
with a natural fixed-point-free involution.
We show in~\cite{GZ} that the equivariant~$w_2$ of $\La_{\C}^{\top}(V,\ti\phi)$,
and not of~$(V,\ti\phi)$ itself, enters into orientability considerations in real Gromov-Witten
theory (i.e.~when interchanges of halves are considered).\\
 
\noindent
We prove Theorem~\ref{main_thm} in Section~\ref{mainpf_sec} by separating off the contributions
of the individual crosscaps, following one of the principles in the proof of \cite[Theorem~1.1]{Ge};
the contribution from the remainder of~$\Si$ is then given by \cite[Theorem~1.1]{Ge}.
We determine the contributions from the crosscaps in Section~\ref{mainpf_sec} by
combining some of the ideas in the proof of \cite[Theorem~1.1]{Ge} 
with equivariant cohomology.

\begin{remark}\label{fam_rem}
Families of real Cauchy-Riemann operators often arise by pulling back data from
a target manifold by smooth maps as follows. 
Suppose $(X,J)$ is an almost complex manifold with an anti-complex  involution
$\phi\!:X\!\lra\!X$ and $(V,\ti\phi)\!\lra\!(X,\phi)$ is a real bundle pair.
Let $\na$ be a connection in $V$ and 
$$A\in\Ga\big(X;\Hom_{\R}(V,(T^*X,J)^{0,1}\otimes_{\C}\!V)\big).$$ 
For any map $u:\Si\!\lra\!X$ and $\fJ\!\in\!\cJ_{\Si}$, 
let $\na^u$ denote the induced connection in $u^*V$ and
$$ A_{\fJ;u}=A\circ \prt_{\fJ} u\in\Ga(\Si;
\Hom_{\R}(u^*V,(T^*\Si,\fJ)^{0,1}\otimes_{\C}u^*V)\big).$$
If $c$ is a boundary involution on $\Si$ and $u\!\circ\!c\!=\!\phi\!\circ\!u$ on $\prt\Si$, 
the homomorphisms
$$\bp_u^\na =\frac{1}{2}(\na^u+\fI\circ\na^u\circ\fJ), \,\,
D_u\equiv \bp_u^\na\!+\!A_{\fJ;u}\!: \Ga(\Si;u^*V)^{u^*\ti\phi}\lra
\Ga^{0,1}_{\fJ}(\Si;u^*V)$$
are real Cauchy-Riemann operators on $(u^*V,u^*\ti\phi)\!\lra\!(\Si,c)$
that form families of real Cauchy-Riemann operators over families of maps.
\end{remark}

\noindent
The \textsf{double} of an oriented sh-surface $(\Si,c)$ is the closed oriented topological surface
\BE{dblSi_e}\hat\Si\equiv\big(\Si^+\sqcup\Si^-\big)\big/\!\sim~
\equiv\{+,-\}\!\times\!\Si\big/\!\sim,\qquad
(+,z)\sim\big(-,c(z)\big)~~~\forall\,z\!\in\!\prt\Si.\EE
The involution~$c$ on $\prt\Si$ naturally extends to the involution
$$\hat{c}\!:\hat\Si\lra\hat\Si, \qquad [\pm,z]\lra[\mp,z]\quad\forall\,z\!\in\!\Si.$$
Similarly, if $(X,\phi)$ is a manifold with an involution, 
a map $u\!:\Si\!\lra\!X$ such that $u\!\circ\!c\!=\!\phi\!\circ\!u$ on~$\prt\Si$
\textsf{doubles} to a map $\hat{u}\!:\hat\Si\!\lra\!X$ such that 
$\hat{u}\!\circ\!\hat{c}\!=\!\phi\!\circ\!\hat{u}$.
A complex structure~$\fJ$ on $\Si^+\!=\!\Si$ extends to a complex structure~$\hat\fJ$ 
on~$\hat\Si$ so that $\hat{c}^*\hat\fJ\!=\!-\hat\fJ$ if and only~if
$c\!:\prt\Si\!\lra\!\prt\Si$ is real-analytic with respect to~$\fJ$;
see Corollary~\ref{dblJ_cor}.
If $c$ is real-analytic with respect to~$\fJ$, $J$ is an almost complex structure  on~$X$ 
such that $\phi^*J\!=\!-J$,  and $u$ as above is $(J,\fJ)$-holomorphic, 
then $\hat{u}$ is $(J,\hat\fJ)$-holomorphic.

\begin{remark}\label{dbldfn_e}
We note that \eref{dblSi_e} does not specify a smooth structure on~$\hat\Si$
across the boundary of~$\Si$.
Whenever $c$ is real-analytic with respect to~$\fJ$, there is a natural doubled 
smooth complex structure~$\hat\fJ$ so that the image of $\prt\Si$ in~$\Si$ 
is a real-analytic curve; see~(1) in the proof of Corollary~\ref{dblJ_cor}.
However, there can be other smooth and complex structures on~$\hat\Si$ that
are compatible with~$\hat{c}$ and restrict to~$\fJ$ on~$\Si$; see Remark~\ref{dbl_rem}.
\end{remark}

\noindent
Let $(\Si,c)$ be a genus~$g$ oriented sh-surface with orderings 
$$(\prt\Si)_1,\ldots,(\prt\Si)_{|c|_0} \qquad\hbox{and}\qquad
(\prt\Si)_{|c|_0+1},\ldots,(\prt\Si)_{|c|_0+|c|_1}$$
of the boundary components with $|c_i|\!=\!0$ and with $|c_i|\!=\!1$, respectively.
Denote by $\cD_c$ the group of diffeomorphisms of~$\Si$ preserving the orientation and each
boundary component and commuting with the involution~$c$ on~$\prt\Si$.
If $(X,\phi)$ is a smooth manifold with an involution~and 
\BE{btuple_eq}
\b=(B,b_1,\ldots,b_{|c|_0+|c|_1})
\in H_2(X;\Z)\oplus H_1(X^{\phi};\Z)^{\oplus |c|_0}\oplus 
H_1^{\phi}(X;\Z)^{\oplus |c|_1}, \EE
let $\fB_g(X,\b)^{\phi,c}$ denote the space of maps 
$u\!:\Si\!\lra\!X$ such that 
\begin{enumerate}[label=$\bullet$,leftmargin=*]
\item $u\!\circ\!c\!=\!\phi\!\circ\!u$ on $\prt\Si$,
\item $\hat{u}_*[\hat\Si]=B$, $u_*[(\prt\Si)_i]=b_i$ for $i=1,\ldots,|c|_0$, and
\item $[u|_{(\prt\Si)_i}]^{c_i}=b_i$ for $i=|c|_0\!+\!1,\ldots,|c|_0\!+\!|c|_1$,
where $[u|_{(\prt\Si)_i}]^{c_i}$ is the equivariant pushforward 
of $[(\prt\Si)_i]^{c_i}$  by~$u|_{(\prt\Si)_i}$.
\end{enumerate}
We define
$$\cH_g(X,\b)^{\phi,c}=
\big(\fB_g(X,\b)^{\phi,c}\times\cJ_c\big)/\cD_c.$$
By Lemma~\ref{cDcJ_lem}, the action of $\cD_c$ on $\cJ_{\Si}$ given by $h\cdot\fJ=h^*\fJ$
preserves~$\cJ_c$; so, the above quotient is well-defined.

\begin{remark}\label{equivH1_rem} 
As discussed in detail at the beginning of Section~\ref{equivappl_subs}, 
$H_1^{\phi}(X;\Z)$ provides a finer
invariant than $H_1(X;\Z)$.
\end{remark}

\begin{remark}\label{mfld_rem} 
For simplicity, we will assume that the action of $\cD_c$ has no fixed points on the relevant
subspaces of $\fB_g(X,\b)^{\phi,c}\!\times\!\cJ_c$.
This happens for example if sufficiently many marked points are added to~$\Si$.
In applications to more general cases,  this issue can be avoided by working with 
Prym structures on Riemann surfaces; see~\cite{Loo}. 
\end{remark}

\noindent
The determinant line bundle of a family of real Cauchy-Riemann operators
$D_{(V,\ti\phi)}$ on 
$$\fB_g(X,\b)^{\phi,c}\times\cJ_c$$ 
induced by a real bundle pair as in Remark~\ref{fam_rem} descends 
to a line bundle over $\cH_g(X,\b)^{\phi,c}$, which we still denote by $\det D_{(V,\ti\phi)}$.
As a direct corollary of Theorem~\ref{main_thm}, we obtain the following result
on its orientability.

\begin{cor}\label{orient_cor} 
Let $\gm$ be a loop in $\cH_g(X,\b)^{\phi,c}$ and $\wt\gm$ be a path in
$\fB_g(X,\b)^{\phi,c}\!\times\!\cJ_c$ lifting $\gm$ such that
$\wt\gm_1=\psi\cdot\wt\gm_0$ for some $\psi\!\in\!\cD_c$ with $\psi|_{\prt\Si}\!=\!\id$.
For each boundary component $(\prt\Si)_i$ of~$\Si$, denote by 
$\al_i\!:S^1\!\lra\!X$ and $\be_i\!:S^1\!\times\!(\prt\Si)_i\!\lra\!X$ 
the paths traced by a fixed point on $(\prt\Si)_i$ and
by the entire boundary component~$(\prt\Si)_i$ along~$\wt\gm$. 
Then,
\BE{orient_eq}\begin{split}
\lr{w_1(\det D_{(V,\ti\phi)}),\gm}&= 
\sum_{i=1}^{|c|_0}\Big(\big(\blr{w_1(V^{\ti\phi}),b_i}+1\big)
\lr{w_1(V^{\ti\phi}),[\al_i]}+\blr{w_2(V^{\ti\phi}),[\be_i]}\Big)\\
&\quad+\sum_{i=|c|_0+1}^{|c|_0+|c|_1} \!\!
\blr{w_2^{\La_{\C}^{\top}{\ti\phi}}(\La_{\C}^{\top}V),[\be_i]^{\id\times c_i}},
\end{split}\EE
where $[\be_i]^{\id\times c_i}\!\in\!H_2^{\phi}(X)$ is the equivariant push-forward of
$[S^1\!\times\!(\prt\Si)_i]^{\id\times c_i}$ by~$\be_i$.
\end{cor}

\noindent
The first assumption in this corollary imposes no restriction on~$\gm$;
see Lemma~\ref{lift_cor}.
Lemma~\ref{w2equiv_lmm} simplifies the computation of the second line 
in~\eref{orient_eq} in some cases.
In particular, combining Corollary~\ref{orient_cor} and Lemma~\ref{w2equiv_lmm},
we obtain Corollary~\ref{etaorient_cor} below, which concerns the orientability
problem for families of real Cauchy-Riemann operators over
surfaces without standard boundary components (i.e.~with crosscaps only).\\

\noindent
If $(V,\ti\phi)\!\lra\!(X,\phi)$ is a rank~1 real  bundle pair,
a \textsf{real square root for $(V,\ti\phi)$} is 
a rank~1 real bundle pair $(L,\ti\phi')\!\lra\!(X,\phi)$ and an isomorphism
$$(V,\ti\phi)\approx (L,\ti\phi')^{\otimes2}$$
of real bundle pairs.
As shown in \cite[Section 2.1]{Teh}, real square roots canonically 
induce orientations on the determinant lines of real Cauchy-Riemann operators
over disks with crosscaps.
Thus, Corollary~\ref{etaorient_cor} explains and extends this key observation
in~\cite{Teh}.

\begin{cor}\label{etaorient_cor}
Let $(X,\phi)$ be a manifold with an involution, $(V,\ti\phi)\!\lra\!(X,\phi)$
be a real bundle pair,
and $(\Si,c)$ be an oriented sh-surface with $|c|_0\!=\!0$.
If $\pi_1(X)\!=\!0$ and $c_1(V)$ is an even class or 
$\La^{\top}_{\C}(V,\ti\phi)$ admits a real square root, then
the determinant line bundle 
$$\det D_{V,\ti\phi} \lra \cH_g(X,\b)^{\phi,c}$$
is orientable.
\end{cor}

\noindent
By Corollary~\ref{SQrt_crl}, 
the two sets of completely different conditions in Corollary~\ref{etaorient_cor} 
are in fact two different  specializations of the natural vanishing
condition on the $w_2$-terms in~\eref{mainthm_e} and~\eref{orient_eq} that first 
came~up in the case $|c|_1\!=\!0$ in~\cite{Ge}: 
$w_2^{\La^{\top}_{\C}\ti\phi}(\La^{\top}_{\C}V)$ is a square of a class
in~$H^1_{\phi}(X)$.
Since there are no non-trivial square top cohomology classes on
a closed orientable surface, each summand on the second lines 
in~\eref{mainthm_e} and~\eref{orient_eq} vanishes if 
$w_2^{\La^{\top}_{\C}\ti\phi}(\La^{\top}_{\C}V)$ is a square class.
Thus, the two sets of conditions in Corollaries~\ref{etaorient_cor} 
and~\ref{etaorient_cor2} can be replaced by the single requirement 
that the equivariant $w_2$ is a square class.\\

\noindent
Let $\si$ be an orientation-reversing involution on a compact closed oriented genus~$g$ 
surface~$\hat\Si$.
Denote by $\cJ_{\si}$ and $\cD_{\si}$ the space of complex structures~$\hat\fJ$ on~$\hat\Si$
such that $\si^*\hat\fJ\!=\!-\hat\fJ$ and the group of orientation-preserving diffeomorphisms~$\psi$ 
of~$\hat\Si$  such that $\si\!\circ\!\psi\!=\!\psi\!\circ\!\si$, respectively.
Let $(X,\phi)$ be a smooth manifold with an involution and $J$ 
be an almost complex structure on $X$ such that $\phi^*J\!=\!-J$.
Of a particular interest in real Gromov-Witten theory is the problem of 
the orientability of the moduli spaces 
$$\mf_g(X,J,B)^{\phi,\si}\equiv
\big\{(\hat{u},\hat\fJ)\in C^{\i}(\hat\Si,X)\!\times\!\cJ_{\si}\!:~\hat{u}_*[\hat\Si]\!=\!B,~
\hat{u}\!\circ\!\si\!=\!\phi\!\circ\!\hat{u},~
\dbar_{J,\hat\fJ}\hat{u}\!=\!0\big\} \big/\cD_{\si},$$
where
\BE{dbar_e} \dbar_{J,\hat\fJ}\hat{u}=\frac12\big(\tnd\hat{u}+J\circ\tnd\hat{u}\circ\hat\fJ\big).\EE
This problem is closely related to the problem of orienting the index bundle for 
a family of real Cauchy-Riemann operators over $\mf_g(X,J,B)^{\phi,\si}$
induced by the bundle $(TX,\tnd\phi)\!\lra\!(X,\phi)$ as in Remark~\ref{fam_rem}.
Since $\hat\Si$ can be decomposed into two conjugate oriented sh-surfaces,
$(\Si,c)$ and~$(\bar\Si,c)$, 
the latter problem is in turn closely related to the orientability of 
$\det D_{(TX,\tnd\phi)}$ over the space of holomorphic maps from~$\Si$ to~$X$
that commute with the involutions on the boundary.\\

\noindent
Let $\eta\!:\P^1\!\lra\!\P^1$ be as in~\eref{etadfn_e}.
As shown in \cite[Section~2.3]{Teh}, the orientability problem for 
$\mf_0(X,J,B)^{\phi,\eta}$
is precisely equivalent to the orientability problem for $D_{(TX,\tnd\phi)}$ over
\begin{equation*}\begin{split}
\cP(X,J,B)&\equiv\big\{\hat{u}\!\in\!C^{\i}(\P^1,X)\!:~\hat{u}_*[\P^1]\!=\!B,~
\hat{u}\!\circ\!\eta\!=\!\phi\!\circ\!\hat{u},~\dbar_{J,\hat\fJ}\hat{u}\!=\!0\big\}\\
&\equiv\big\{u\!\in\!C^{\i}(D^2,X)\!:~\hat{u}_*[\P^1]\!=\!B,~
u\!\circ\!\fa|_{S^1}\!=\!\phi\!\circ\!u|_{S^1},
~\dbar_{J,\fJ}u\!=\!0\big\},
\end{split}\end{equation*}
where $\hat\fJ$ and $\fJ$ are the standard complex structures on $\P^1$ and 
the unit disk $D^2\!\subset\!\P^1$, respectively.
The reason is that 
$$\mf_0(X,J,B)^{\phi,\eta}=\cP(X,J,B)\big/\aut(\P^1,\eta)
\qquad\hbox{and}\quad \aut(\P^1,\eta)\approx \R\P^3.$$
Thus, the orientability of $\mf_0(X,J,B)^{\phi,\eta}$ along a loop $\gm$ 
in this moduli space
is described by the last term in~\eref{orient_eq} with $(V,\ti\phi)\!=\!(TX,\tnd\phi)$.
This allows us to immediately deduce the following conclusion about the orientability
of $\mf_0(X,J,B)^{\phi,\eta}$ from Corollary~\ref{etaorient_cor}.

\begin{cor}\label{etaorient_cor2}
Let $(X,\phi,J)$ be a manifold with an involution and an almost complex structure~$J$
such that $\phi^*J\!=\!-J$.
If $\pi_1(X)\!=\!0$ and $c_1(TX,J)$ is an even class or 
$\La^{\top}_{\C}(TX,\ti\phi)$ admits a real square root, 
the moduli space $\mf_0(X,J,B)^{\phi,\eta}$ is orientable
for every $B\!\in\!H_2(X;\Z)$.
\end{cor}

\noindent
The orientability of $\mf_0(X,J,B)^{\phi,\eta}$ under the last assumption is shown
directly in \cite[Section~2.1]{Teh}.
This in particular implies that the moduli spaces $\mf_0(\P^{4m-1},B)^{\phi,\eta}$
are orientable.
All moduli spaces $\mf_0(\P^{2m-1},B)^{\phi,\eta}$ with the two standard involutions~$\phi$
are shown to be orientable in \cite[Section~A.1]{Teh},
as implied by the first case of our assumptions.
By Example~\ref{evenP_eg} below, the moduli spaces $\mf_0(\P^n,B)^{\tau_n,\eta}$,
where 
\BE{tau2mdfn_e}
\tau_n\!:\P^n\lra\P^n,\qquad
[Z_0,Z_1,\ldots,Z_n]\lra[\bar{Z}_0,\bar{Z}_1,\ldots,\bar{Z}_n],\EE
is not orientable if $n$ and $B$ are even, 
i.e.~the condition $\pi_1(X)\!=\!0$ alone does 
not suffice for the orientability of $\mf_0(X,J,B)^{\phi,\eta}$.
By Example~\ref{prodorient_eg}, which builds on \cite[Example~2.5]{Teh}, 
no divisibility condition on $c_1(TX)$ can suffice by itself either.\\

\noindent
This paper is organized as follows.
Section~\ref{equivcoh_sec} reviews $\Z_2$-equivariant cohomology and homology and
obtains a number of related results that are used in the proof and applications
of Theorem~\ref{main_thm}.
The examples of Section~\ref{examples_sec} indicate the delicate nature
of the equivariant $w_2$-terms in~\eref{orient_eq} and
show that Corollary~\ref{etaorient_cor2} is sharp in a sense.
Section~\ref{double_sec} describes doubling constructions for oriented surfaces
with boundary involutions, develops the necessary Fredholm theory, and 
obtains a Riemann-Roch theorem for 
Cauchy-Riemann operators over such surfaces.
The somewhat technical Sections~\ref{equivcoh_sec} and~\ref{double_sec} 
enable us to extend the principles from~\cite{Ge} to
surfaces with crosscaps.
The proof of Theorem~\ref{main_thm} is the subject of Section~\ref{mainpf_sec}.
Section~\ref{sec:ls} reinterprets Corollary~\ref{orient_cor} in terms of 
local systems.
In Section~\ref{sec:ls2}, we define moduli spaces of $J$-holomorphic maps from
oriented surfaces with crosscaps and  apply the reinterpretation of Section~\ref{sec:ls}
to describe their local systems of orientations.
In Appendix~\ref{realanal_app}, we show that the notion of almost complex structure
on a bordered Riemann surface used in this paper is equivalent to 
the notion of analytic structure used in~\cite{AG,KatzLiu,Melissa}.\\ 

\noindent
We would like to thank M.~Liu for detailed discussions on topics covered in this paper, 
and W.~Browder, E.~Brugell\'e, E.~Ionel, S.~Galatius, J.~Solomon, M.~Tehrani, and G.~Tian 
for related conversations, and
the referee for the quick and detailed feedback.
The second author is also grateful to the IAS School of Mathematics for its hospitality 
during the period when the results in this paper were obtained.

\section{Equivariant cohomology}
\label{equivcoh_sec}

\noindent
We begin this section by recalling basic notions in equivariant cohomology,
in the case the group is~$\Z_2$,
and establishing some key properties of the equivariant~$w_2$ 
of real vector bundle pairs; see Proposition~\ref{tensor_prp} and 
Corollary~\ref{SQrt_crl}.
We then make a key observation, Lemma~\ref{T2equivL_lmm}, concerning 
the $\Z_2$-equivariant second Stiefel-Whitney class of real bundle pairs over 
the torus
$(S^1\!\times\!S^1,\id\!\times\!\fa)$; it is used in the proof of Theorem~\ref{main_thm}.
Lemma~\ref{w2equiv_lmm} simplifies the computation of the second line in~\eref{orient_eq}
in some cases and immediately leads to Corollary~\ref{etaorient_cor} from 
Corollary~\ref{orient_cor}.
We conclude with three examples intended to give the flavor of 
the equivariant $w_2$-term which plays a central role in the orientability problems
studied in this paper.

\subsection{Basic notions}
\label{equivdfn_e}

\noindent
The group $\Z_2$ acts freely on the contractible space $\bE\Z_2\!\equiv\!S^{\i}$
with the quotient $\bB\Z_2\!\equiv\!\R\P^{\i}$.
An involution $\phi\!:M\!\lra\!M$ corresponds to a $\Z_2$-action on~$M$.
We denote by
$$\bB_{\phi}M=\bE\Z_2\!\times_{\Z_2}\!M$$
the corresponding Borel construction and by
$$H^*_{\phi}(M)\equiv H^*(\bB_{\phi}M;\Z_2), \quad
 H_*^{\phi}(M)\equiv H_*(\bB_{\phi}M;\Z_2), \quad 
 H_*^{\phi}(M;\Z)\equiv H_*(\bB_{\phi}M;\Z)$$
the corresponding \textsf{$\Z_2$-equivariant cohomology} and \textsf{homology} of~$M$. 
The projection map \hbox{$p_1\!:\bE\Z_2\!\times\!M\!\lra\!\bE\Z_2$} descends to a fibration
\BE{BXfib_e} M\lra \bB_{\phi}M\lra \bB\Z_2=\R\P^{\i}\,.\EE
If $(V,\ti\phi)\!\lra\!(M,\phi)$ is a real bundle pair,
$$\bB_{\ti\phi} V\equiv \bE\Z_2\!\times_{\Z_2}\!V\lra \bB_{\phi}M$$
is a real vector bundle;
this is the quotient of the vector bundle $p_2^*V\!\lra\!\bE\Z_2\!\times\!M$
by the natural lift of the free $\Z_2$-action on the base.
Let 
$$w_i^{\ti\phi}(V)\equiv w_i(\bB_{\ti\phi}V)\in H_{\phi}^i(M)$$
be the \textsf{$\Z_2$-equivariant Stiefel-Whitney classes of $V\!\lra\!M$}. 
For example, if $M$ is a point and $V\!=\!\C\!=\!\R\!\oplus\!\fI\R$, 
$$\bB_{\ti\phi}V=\R\P^{\i}\!\times\!\R \oplus\cO_{\R\P^{\i}}(-1)\lra\R\P^{\i}\,,$$
where $\cO_{\R\P^{\i}}(-1)$ is the tautological line bundle;
thus, $w_1^{\ti\phi}(V)$ is the generator of $H^1_{\phi}(M)$
in this case.
The non-equivariant Stiefel-Whitney classes of~$V$ are recovered from 
the equivariant Stiefel-Whitney classes of~$V$ by restricting to the fiber
of the fibration~\eref{BXfib_e}.
If $f\!:\Si\!\lra\!M$ is a continuous map commuting with involutions~$c$ on~$\Si$
and~$\phi$ on~$M$, 
the involution~$\ti\phi$ on~$V$ induces an involution $f^*\ti\phi$ on~$f^*V$ 
lifting~$c$ and 
\BE{equivback_e} w_i^{f^*\ti\phi}(f^*V)=\{\bB_{\phi,c}f\}^*w_i^{\ti\phi}(V)\in H^i_c(\Si),\EE
where 
$$\bB_{\phi,c}f\!:\bB_c\Si\!\lra\!\bB_{\phi}M, \qquad
\{\bB_{\phi,c}f\}\big([e,z]\big)=\big[e,f(z)\big],$$
is the map induced by~$f$.\\

\noindent
If an involution $c\!:\Si\!\lra\!\Si$ has no fixed points,
the projection $p_2\!:\bE\Z_2\!\times\!\Si\!\lra\!\Si$ 
descends to a fibration
\BE{Z2fib_e2}\bE\Z_2\lra \bB_c\Si\stackrel{q}{\lra} \Si/\Z_2\,.\EE
Since $\bE\Z_2$ is contractible, this fibration is a homotopy equivalence,
with a homotopy inverse provided by any section of~$q$.
In particular, $q$~induces isomorphisms
\BE{Z2isom_e}q^*\!:H^*(\Si/\Z_2)\lra H^*_c(\Si), \qquad
q_*\!: H_*^c(\Si;\Z)\lra H_*(\Si/\Z_2;\Z).\EE
Any section of~$q$ embeds $\Si/\Z_2$ as a homotopy retract, and 
every two such sections are homotopic.
Thus, if $f\!:\Si\!\lra\!M$ is a continuous map commuting  
with the involutions~$c$ on~$\Si$ and~$\phi$ on~$M$, we also denote~by
$$\bB_{\phi,c}f\!:\Si/\Z_2\lra\bB_{\phi}M$$
the composition of $\bB_{\phi,c}f\!:\bB_c\Si\!\lra\!\bB_{\phi}M$ with
any section of~$q$; this is well-defined and unambiguous up~to homotopy.
If $(V,\ti{c})\!\lra\!(\Si,c)$ is a real bundle pair, 
$V/\Z_2\!\lra\!\Si/\Z_2$ is a real vector bundle and
\begin{gather*}
\bB_{\ti\phi}V\lra q^*(V/\Z_2)\equiv
\big\{\big([e,x],[v]\big)\!\in\!\bB_c\Si\!\times\!(V/\Z_2)\!:\,
[x]\!=\![p(v)]\big\}, \\ 
[e,v]\lra \big([e,p(v)],[v]),
\end{gather*}
is a vector bundle isomorphism covering the identity on $\bB_c\Si$.
Thus,
\BE{Vquot_e} w_i^{\ti{c}}(V)=w_i\big(q^*(V/\Z_2)\big)=q^*w_i(V/\Z_2)
\in H^i_c(\Si;\Z_2).\EE 

\subsection{Tensor products of real line bundle pairs}
\label{tensorprod_subs} 

\noindent
We now establish some important properties of the equivariant~$w_2$ 
of real vector bundle pairs.

\begin{prop}\label{tensor_prp}
Let $(X,\phi)$  be a paracompact topological space with an involution.
\begin{enumerate}[label=(\arabic*),leftmargin=*]
\item If $(L_1,\ti\phi_1),(L_2,\ti\phi_2)\!\lra\!(X,\phi)$ are rank~1 real bundle pairs, then
\BE{tensor_e}w_2^{\ti\phi_1\otimes_{\C}\ti\phi_2}(L_1\!\otimes_{\C}\!L_2)
 =w_2^{\ti\phi_1}(L_1) +w_2^{\ti\phi_2}(L_2).\EE
\item If $(V_1,\ti\phi_1),(V_2,\ti\phi_2)\!\lra\!(X,\phi)$ are any real bundle pairs, then
$$w_2^{\La_{\C}^{\top}(\ti\phi_1\oplus\ti\phi_2)}\big(\La_{\C}^{\top}(V_1\!\oplus\!V_2)\big)
=w_2^{\La_{\C}^{\top}\ti\phi_1}(\La_{\C}^{\top}V_1)
+w_2^{\La_{\C}^{\top}\ti\phi_2}(\La_{\C}^{\top}V_2).$$
\end{enumerate}
\end{prop}

\begin{proof}
The first statement of this proposition implies the second, since
$$\La_{\C}^{\top}\big((V_1,\ti\phi_1)\!\oplus\!(V_2,\ti\phi_2)\big)
=\big(\La_{\C}^{\top}(V_1,\ti\phi_1)\big)
\otimes \big(\La_{\C}^{\top}(V_2,\ti\phi_2)\big).$$
Below we establish~\eref{tensor_e}.\\

\noindent
(1) Let $\P^{\i}$ denote the infinite-dimensional complex projective space with the 
standard involution,
$$\tau_{\i}\!: \P^{\i}\lra\P^{\i}, \qquad [z_1,z_2,\ldots]\lra[\bar{z}_1,\bar{z}_2,\ldots],$$
$p_1,p_2\!:\P^{\i}\!\times\!\P^{\i}\!\lra\!\P^{\i}$ be the projection maps, and
$$(\bX,\Phi)=\big(\P^{\i}\!\times\!\P^{\i},\tau_{\i}\times\tau_{\i}\big).$$
The homotopy exact sequence for the fibration~\eref{BXfib_e}
with~$M$ replaced by~$\bX$  gives an exact sequence
$$\pi_2(\bX)\lra \pi_2\big(\bB_{\Phi}\bX)\lra 0 \lra 0\lra\pi_1(\bB_{\Phi}\bX)
\lra \pi_1(\R\P^{\i})\lra 0.$$
In particular, $H_1(\bB_{\Phi}\bX;\Z)\!=\!\Z_2$.
By \cite[Satz~II]{Hopf}, $\pi_2(\bB_{\Phi}\bX)$ surjects onto 
$H_2^{\Phi}(\bX;\Z)$, since
$$H_2\big(\pi_1(\bB_{\Phi}\bX);\Z\big) = H_2\big(\Z_2;\Z\big)
\equiv H_2(\bB\Z_2;\Z)=H_2(\R\P^{\i};\Z)=0.$$ 
Thus, by the Universal Coefficient Theorem for Homology \cite[Theorem~55.1]{Mu2},
$H_2^{\Phi}(\bX)$  is generated~by
$$\P^1_1\equiv \big\{[1,z,[1]]\!:\,z\!\in\!\P^1\big\}, \quad
\P^1_2\equiv \big\{[1,[1],z]\!:\,z\!\in\!\P^1\big\},\quad
\R\P^2\equiv \big\{[z,[1],[1]]\!:\,z\!\in\!S^2\big\},$$
where $1\!\equiv\![1,0,\ldots]\!\in\!S^{\i}$, $[1]\!\in\!\P^{\i}$ denotes 
the $S^1$-equivalence class of~$1$ in~$S^{\i}$, and $[z_1,z_2,z_3]$ denotes
the $\Z_2$-equivalence class of $(z_1,z_2,z_3)$ in 
$S^{\i}\!\times\!\P^{\i}\!\times\!\P^{\i}$;
the spheres $\P^1_1$ and $\P^1_2$ are generators of $H_2^{\Phi}(\bX;\Z)$
coming from $\pi_2(\bX)$ via $\pi_2(\bB_{\Phi}\bX)$, while $\R\P^2$ generates 
$$\tn{Tor}\big(H_1(\bB_{\Phi}\bX;\Z),\Z_2\big)\approx\Z_2$$
in the split short exact sequence for $H_2^{\Phi}(\bX)$
provided by  \cite[Theorem~55.1]{Mu2}.\\

\noindent
(2) The involution $\tau_{\i}$ naturally lifts to a conjugation~$\ti\tau_{\i}$ 
on the tautological line bundle
$$\cO_{\P^{\i}}(-1)\lra\P^{\i}.$$
We first verify~\eref{tensor_e} for the rank~1 real bundle pairs 
$$(\bL_1,\ti\Phi_1)\equiv p_1^*\big(\cO_{\P^{\i}}(-1),\ti\tau_{\i}\big),~
(\bL_2,\ti\Phi_2)\equiv p_2^*\big(\cO_{\P^{\i}}(-1),\ti\tau_{\i}\big)
~\lra~(\bX,\Phi).$$
Since the homology classes of $\P_1^1$ and $\P_2^1$ are the images of classes 
in a fiber of~\eref{BXfib_e} with~$M$ replaced by~$\bX$, 
\begin{equation*}\begin{split}
&\blr{w_2^{\ti\Phi_1\otimes\ti\Phi_2}(\bL_1\!\otimes\!\bL_2),[\P_i^1]_{\Z_2}} 
=\blr{w_2(\bL_1\!\otimes\!\bL_2),[\P_i^1]_{\Z_2}} 
=\blr{c_1(\bL_1\!\otimes\!\bL_2),[\P_i^1]_{\Z}}+2\Z\\
&\hspace{.5in}=\blr{c_1(\bL_1),[\P_i^1]_{\Z}}+
\blr{c_1(\bL_2),[\P_i^1]_{\Z}}+2\Z\\
&\hspace{.5in}=\blr{w_2(\bL_1),[\P_i^1]_{\Z_2}}+
\blr{w_2(\bL_2),[\P_i^1]_{\Z_2}}
=\blr{w_2^{\ti\Phi_1}(\bL_1),[\P_i^1]_{\Z_2}}+
\blr{w_2^{\ti\Phi_2}(\bL_2),[\P_i^1]_{\Z_2}}.
\end{split}\end{equation*}
The restrictions of $\bB_{\ti\Phi_1}\bL_1$, $\bB_{\ti\Phi_2}\bL_2$,
and $\bB_{\ti\Phi_1\otimes_{\C}\ti\Phi_2}(\bL_1\!\otimes_{\C}\!\bL_2)$ to 
$\R\P^{\i}\equiv \R\P^{\i}\!\times\![1]\!\times\![1]\subset\bB_{\Phi}\bX$ are
$$S^{\i}\!\times\!_{\Z_2}\C
\approx \R\P^{\i}\!\times\!\R\oplus\cO_{\R\P^{\i}}(-1)\lra\R\P^{\i}.$$
Thus, 
$$\blr{w_2^{\ti\Phi_1\otimes_{\C}\ti\Phi_2}(\bL_1\!\otimes_{\C}\!\bL_2),[\R\P^2]_{\Z_2}} 
=0=0+0
=\blr{w_2^{\ti\Phi_1}(\bL_1),[\R\P^2]_{\Z_2}} 
+\blr{w_2^{\ti\Phi_2}(\bL_2),[\R\P^2]_{\Z_2}}.$$
Since $\P_1^1$, $\P_2^1$, and $\R\P^2$ generate $H_2^{\Phi}(\bX)$,
this establishes~\eref{tensor_e} for $(L_i,\ti\phi_i)\!=\!(\bL_i,\ti\Phi_i)$.\\

\noindent
(3) Let $(L_i,\ti\phi_i)\!\lra\!(X,\phi)$ be as in the statement of the proposition.
By the proof of \cite[Lemma~5.6]{MiSa}, there exist continuous maps 
$$f_1,f_2\!:(X,\phi)\!\lra\!(\P^{\i},\tau_{\i})
\qquad\hbox{s.t.}\quad
(L_i,\ti\phi_i)=f_i^*(\cO_{\P^{\i}}(-1),\ti\tau_{\i}).$$
Thus,
\begin{equation*}\begin{split}
&w_2^{\ti\phi_1\otimes_{\C}\ti\phi_2}(L_1\!\otimes_{\C}\!L_2)
=\{f_1\!\times\!f_2\}^*
w_2^{\ti\Phi_1\otimes_{\C}\ti\Phi_2}(\bL_1\!\otimes_{\C}\!\bL_2)
=\{f_1\!\times\!f_2\}^*\big( w_2^{\ti\Phi_1}(\bL_1)+w_2^{\ti\Phi_2}(\bL_2)\big)\\
&\hspace{.8in}
=\{f_1\!\times\!f_2\}^*\big( p_1^* w_2^{\ti\tau_{\i}}(\cO_{\P^{\i}}(-1))
+p_2^*w_2^{\ti\tau_{\i}}(\cO_{\P^{\i}}(-1)) \big)
=w_2^{\ti\phi_1}(L_1)+w_2^{\ti\phi_2}(L_2);
\end{split}\end{equation*}
the second equality above follows from~(2).
\end{proof}

\begin{lem}\label{KleinSurface_lem}
Let $\Si$ be a compact connected unorientable surface and $b_{\Si}\!\in\!H_1(\Si;\Z)$ be 
the nontrivial torsion class.
If $\ka\!\in\!H^1(\Si;\Z_2)$,
\BE{unorpair_e}\blr{\ka^2,[\Si]_{\Z_2}}=\lr{\ka,b_{\Si}}\,,\EE
where $[\Si]_{\Z_2}\!\in\!H_2(\Si;\Z_2)$ is the fundamental class with $\Z_2$-coefficients.
\end{lem}

\begin{proof}
By \cite[Theorem~77.5]{Mu}, $\Si$ is the connected sum of $m$ copies
of~$\R\P^2$ and
$$H_2(\Si;\Z)\approx\Z^{m-1}\oplus\Z_2$$
for some $m\!\in\!\Z$.
By \cite[Theorem~77.5]{Mu}, $\Si$ 
can be represented by the labeling scheme $a_1a_1a_2a_2\!\ldots\!a_ma_m$; see Figure~\ref{Tm_fig}.
From the labeling scheme, it is immediate that 
the torsion element is given~by 
$$b_{\Si}=a_1\!+\!a_2\!+\!\ldots\!+\!a_m.$$
From the diagram, we see~that the $\Z_2$-homology intersection product is given~by 
$a_i\!\cdot\!a_i\!=\!1$ and $a_i\!\cdot\!a_j\!=\!0$ if $i\!\neq\!j$.
By the unoriented Poincare duality \cite[Theorem~67.1]{Mu2},
$\ka$ is thus Poincar\'e dual to the sum of some number of $a_1,a_2,\ldots,a_m$
and $\ka^2$ is Poincar\'e dual to the sum of the same number of $a_1^2,a_2^2,\ldots,a_m^2$. 
Thus, each side of~\eref{unorpair_e} vanishes if and only if 
$\ka$ is Poincar\'e dual to an even number of $a_1,a_2,\ldots,a_m$.
\end{proof}

\begin{figure}
\begin{pspicture}(-7,-2.3)(10,1)
\psset{unit=.4cm}
\psline[linewidth=.04](6,0)(3,1.73)\psline[linewidth=.04]{->}(6,0)(4.5,.86)
\psline[linewidth=.04](3,1.73)(0,0)\psline[linewidth=.04]{->}(3,1.73)(1.5,.86)
\psline[linewidth=.04](0,0)(0,-3.46)\psline[linewidth=.04]{->}(0,0)(0,-1.73)
\psline[linewidth=.04](6,0)(6,-3.46)\psline[linewidth=.04]{->}(6,-3.46)(6,-1.73)
\psline[linewidth=.04](0,-3.46)(3,-5.19)\psline[linewidth=.04]{->}(0,-3.46)(1.5,-4.32)
\psline[linewidth=.04](3,-5.19)(6,-3.46)\psline[linewidth=.04]{->}(3,-5.19)(4.5,-4.32)
\pscircle*(6,0){.15}\pscircle*(3,1.73){.15}\pscircle*(0,0){.15}
\pscircle*(6,-3.46){.15}\pscircle*(3,-5.19){.15}\pscircle*(0,-3.46){.15}
\rput(5.3,1.3){$a_1$}\rput(.8,1.3){$a_1$}\rput(3,-.2){$a_1'$}
\rput(5.3,-4.76){$a_3$}\rput(.8,-4.56){$a_2$}
\rput(-.6,-1.73){$a_2$}\rput(6.7,-1.73){$a_3$}
\pscircle*(.75,.43){.12}\pscircle*(3.75,1.3){.12}
\pnode(.75,.43){A}\pnode(3.75,1.3){B}
\ncarc[nodesep=0,arcangleA=-76.1,arcangleB=-53.9,ncurv=.6,linestyle=dashed,linewidth=.03]{-}{A}{B}
\end{pspicture}
\caption{Labeling scheme for $\Si=\R\P^2\!\#\R\P^2\!\#\R\P^2$ and a deformation of 
the loop $a_1$ used to compute the intersection product on $H_1(\Si;\Z)$}
\label{Tm_fig}
\end{figure}
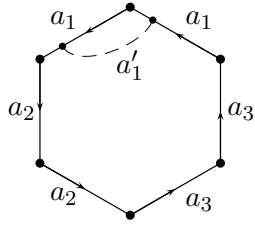

\begin{cor}\label{sqvan_cor}
For any topological space~$M$,
\BE{sqvancor_e} \big\{w\!\in\!H^2(M;\Z_2)\!:\,w(B)\!=\!0~\forall\,B\!\in\!H_2(M;\Z)\big\}
\supset\big\{\ka^2\!:\,\ka\!\in\!H^1(M;\Z_2)\big\}.\EE
If $H_1(M;\Z)$ is finitely generated, the reverse inclusion holds
if and only if $H_1(M;\Z)$ has no \hbox{4-torsion}.
\end{cor}

\begin{proof}
\noindent
(1) If $B\!\in\!H_2(M;\Z)$, there exists a continuous map $f\!:\Si\!\lra\!M$
from a closed oriented surface~$\Si$ such~that 
$$f_*[\Si]_{\Z}=B\in H_2(M;\Z).$$
Since every square class in $H^2(\Si;\Z_2)$ is trivial,
$$\blr{\ka^2,B}=\blr{\ka^2,f_*[\Si]_{\Z_2}}=\blr{(f^*\ka)^2,[\Si]_{\Z_2}}=0
\qquad\forall\,\ka\!\in\!H^1(M;\Z_2).$$
This establishes~\eref{sqvancor_e}.\\

\noindent
(2) If $H_1(M;\Z)$ is finitely generated,
\BE{H1dcomp_e}
H_1(M;\Z)\approx\Z^{r_0}\oplus\Z_{m_1}^{r_1}\oplus\ldots\oplus\Z_{m_k}^{r_k}\EE
for some $m_1,\ldots,m_k\!\ge\!2$ and $r_0,r_1,\ldots,r_m\!\ge\!0$.
In this case,
$$\Ext\big(H_1(M;\Z),\Z_2\big)\approx \bigoplus_{2|m_i}\Z_2^{r_i}\,;$$
see \cite[p331]{Mu2}.
We denote by $b_{i,j}$, with $i\!=\!1,\ldots,k$ and $j\!=\!1,\ldots,r_i$,
loops representing the generators of the torsion part in~\eref{H1dcomp_e}.
For each $i$ with $2|m_i$ and $j\!=\!1,\ldots,r_i$, there exist
a compact oriented surface $\Si_{i,j}$ with two boundary components 
$(\prt\Si_{i,j})_1$ and $(\prt\Si_{i,j})_2$,
a continuous map $F_{i,j}\!:\Si_{i,j}\!\lra\!M$, and 
an orientation-preserving diffeomorphism 
\hbox{$\vph_{i,j}\!: (\prt\Si_{i,j})_2\!\lra\!(\prt\Si_{i,j})_1$}
such~that
$$[F|_{(\prt\Si_{i,j})_1}]=(m_i/2)[b_{i,j}] \qquad\hbox{and}\qquad 
F_{i,j}|_{(\prt\Si_{i,j})_2}\!=\!F_{i,j}|_{(\prt\Si_{i,j})_1}\!\circ\!\vph_{i,j}.$$
The map $F_{i,j}$ descends to a continuous map~$\hat{F}_{i,j}$ from the unorientable surface
$$\hat\Si_{i,j}\equiv  \Si_{i,j}\big/\!\!\sim, \qquad
z\sim\vph_{i,j}(z)~~~\forall\,z\!\in\!(\prt\Si_{i,j})_2.$$
By Lemma~\ref{KleinSurface_lem},
\BE{si2iden_e2} \blr{\ka^2,\{\hat{F}_{i,j}\}_*[\hat\Si_{i,j}]_{\Z_2}}
=\blr{F_{i,j}^*\ka,[(\prt\Si_{i,j})_1]_{\Z_2}}
=\frac{m_i}{2}\lr{\ka,b_{i,j}} \qquad\forall\,\ka\in H^1(M;\Z_2).\EE
By the Universal Coefficient Theorem for Cohomology \cite[Theorem~53.1]{Mu2},
the natural homomorphism
$$H^1(M;\Z_2)\lra \Hom\big(H_1(M;\Z),\Z_2\big)
\approx \Z_2^{r_0}\oplus\bigoplus_{2|m_i}\Z_2^{r_i}$$
is an isomorphism.
By~\eref{si2iden_e2}, the homomorphism
\BE{SqHom_e}\bigoplus_{2|m_i,4{\not\,|}m_i}\!\!\!\!\Z_2^{r_i}\lra
\Ext\big(H_1(M;\Z),\Z_2\big)\subset H^2(M;\Z_2), \qquad \ka\lra\ka^2,\EE
is injective.
On the other hand, this homomorphism vanishes on the factors corresponding
to $\Z_{m_i}$ with $4|m_i$, since
their image is contained in the image of the homomorphism
$$H^1(M;\Z_{m_i})\lra H^2(M;\Z_{m_i})\lra H^2(M;\Z_2)$$
and $2\ka^2=0$ in $H^2(M;\Z_{m_i})$ for all $\ka\!\in\!H^1(M;\Z_{m_i})$.
Thus, the cokernel of the homomorphism~\eref{SqHom_e} is isomorphic
to $\bigoplus_{4|m_i}\Z_2^{r_i}$.
In particular, every element of $H^2(M;\Z_2)$ vanishing on 
the image of $H_2(M;\Z)$ in $H_2(M;\Z_2)$
is a square class if and only if $H_1(M;\Z_2)$ has no 4-torsion.
\end{proof}

\begin{cor}\label{SQrt_crl}
Let $(X,\phi)$ be a topological space with an involution and
$(L,\ti\phi)\!\lra\!(X,\phi)$ be a rank~1 real bundle pair.
\begin{enumerate}[label=(\arabic*),leftmargin=*]
\item\label{SQrtSC_it} If $X$ is simply connected and $w_2(L)\!=\!0$,  
$w_2^{\ti\phi}(L)$ is a square class.
\item\label{SQrt_it} If $X$ is paracompact and $(L,\ti\phi)$ admits a real square root,
$w_2^{\ti\phi}(L)=0$.
\end{enumerate}
\end{cor}

\begin{proof}
\ref{SQrtSC_it} The homotopy exact sequence for the fibration~\eref{BXfib_e}
with $M$ replaced by~$X$  gives an exact sequence
$$\pi_2(X)\lra \pi_2(\bB_{\phi}X)\lra 
0\lra0\lra \pi_1(\bB_{\phi}X)\lra \pi_1(\R\P^{\i})\lra 0.$$
In particular, $H_1(\bB_{\phi}X;\Z)\!=\!\Z_2$.
Thus, by Corollary~\ref{sqvan_cor},
\BE{spincond_e5}\big\{w\!\in\!H^2_{\phi}(X)\!:\,w(b)\!=\!0~\forall\,
b\!\in\!H_2^{\phi}(X;\Z)\big\}
=\big\{\ka^2\!:\,\ka\!\in\!H^1_{\phi}(X)\big\}.\EE
By \cite[Satz~II]{Hopf}, $\pi_2(\bB_{\phi}X)$ surjects onto 
$H_2^{\phi}(X;\Z)$, since
$$H_2\big(\pi_1(\bB_{\phi}X);\Z\big)
= H_2\big(\Z_2;\Z\big)
\equiv H_2(\bB\Z_2;\Z)=H_2(\R\P^{\i};\Z)=0.$$ 
Since the diagram
$$\xymatrix{\pi_2(X)\ar@{->>}[r]\ar[d]^{\approx}& \pi_2(\bB_{\phi}X)\ar@{->>}[d]\\   
H_2(X;\Z)\ar[r]& H_2(\bB_{\phi}X;\Z)\,}$$
commutes, $H_2(X;\Z)$ surjects onto $H_2^{\phi}(X;\Z)$.
Thus, we can replace $H_2^{\phi}(X;\Z)$ in~\eref{spincond_e5} by $H_2(X;\Z)$, i.e.
$$\big\{w\!\in\!H^2_{\phi}(X)\!:\,w(b)\!=\!0~\forall\,
b\!\in\!H_2(X;\Z)\big\}
=\big\{\ka^2\!:\,\ka\!\in\!H^1_{\phi}(X)\big\}.$$
Since the restriction of $w_2^{\ti\phi}(L)$ to the fiber $X\!\subset\!\bB_{\phi}X$
is $w_2(L)\!=\!0$, it follows that $w_2^{\ti\phi}(L)$ is a square class.\\

\noindent
\ref{SQrt_it} This follows immediately from the first statement 
of Proposition~\ref{tensor_prp}.
\end{proof}

\subsection{Applications to real bundle pairs}
\label{equivappl_subs} 
 
\noindent
The antipodal involution $\fa\!:S^1\!\lra\!S^1$ given by~\eref{antip_e} has no fixed points.
A~section of the projection~$q$ in~\eref{Z2fib_e2} in this case is given~by
$$S^1\lra S^{\i}\!\times_{\Z_2}\!S^1\subset(\C^{\i}\!-\!0)\!\times_{\Z_2}\!S^1, \qquad 
\ne^{\fI\th}\lra \big[(\ne^{\fI\th/2},0,0,\ldots),\ne^{\fI\th/2}\big].$$
Thus, if $(X,\phi)$ is a topological space with an involution and 
$\al\!:S^1\!\lra\!X$ is any map such that $\al\!\circ\!\fa\!=\!\phi\!\circ\!\al$, then
$$\bB_{\phi,\fa}\al\!: S^1\lra \bB_{\phi}X, \qquad 
\ne^{\fI\th}\lra \big[(\ne^{\fI\th/2},0,0,\ldots),\al(\ne^{\fI\th/2})\big].$$
The composition of $\bB_{\phi,\fa}\al$ with the projection in~\eref{BXfib_e}
is a generator of $\pi_1(\bB\Z_2)\!\approx\!\Z_2$, and so $[\al]^{\fa}\neq0\in H_1^{\phi}(X)$.
Furthermore, every loop in $B_{\phi}X$ which projects to a  generator of $\pi_1(\bB\Z_2)$
is homotopic to $\bB_{\phi,\fa}\al$ for some $\Z_2$-equivariant map $\al\!:S^1\!\lra\!X$.
If $\al,\be\!:S^1\!\lra\!X$ are two maps commuting with the involutions~$\fa$ on $S^1$
and~$\phi$ on~$X$, 
a homotopy $h\!:\bI\!\times\!S^1\lra\!\bB_{\phi}X$ between $\bB_{\phi,\fa}\al$ 
and $\bB_{\phi,\fa}\be$ lifts to a homotopy 
$$\ti{h}\!:\bI\!\times\!S^1\lra\bE\Z_2\!\times\!X$$
commuting with the $\Z_2$-actions.
Thus, $\bB_{\phi,\fa}\al$ and $\bB_{\phi,\fa}\be$ are homotopic if and only if
$\al$ and $\be$ are homotopic through maps $S^1\!\lra\!X$ intertwining~$\fa$ and~$\phi$.\\

\noindent
The composition of $\bB_{\phi,\fa}\al$ with the projection $\bB_{\phi}X\!\lra\!X/\Z_2$
is the loop $\ne^{\fI\th}\lra [\al(\ne^{\fI\th/2})]$, i.e.~the composition of 
the projection $q_X\!:X\!\lra\!X/\Z_2$ with the restriction of~$\al$ to the upper half 
$S_+^1$ of~$S^1$;
since $\al$ is~$\Z_2$-invariant, $\al(1)\!=\!\phi(\al(-1))$ and so  
the endpoints of this semi-circle map to the same point in~$X/\Z_2$.
Thus, if $\phi$ acts on~$X$ without fixed point, the loop $\bB_{\phi,\fa}\al$ in 
$\bB_{\phi}X$
corresponds to the loop $q_X\!\circ\!\al|_{S^1_+}$ in $X/\Z_2$
under the isomorphism $q_*$ 
in~\eref{Z2isom_e}, with $\Si$ replaced by~$X$.
For example, let $n\!\ge\!2$,
$$X=(S^n\!\times\!S^n)/\sim,~~(x_1,x_2)\sim(-x_1,-x_2), \qquad
\phi\big([x_1,x_2]\big)=[-x_1,x_2]=[x_1,-x_2].$$
If $\al\!:S^1\!\lra\!S^n$ is any map intertwining the antipodal involutions,
the homotopically trivial loops
$$\al_1,\al_2\!:S^1\lra X, \qquad \al_1=[\al,x^*],~~\al_2=[x^*,\al],$$
where $x^*\!\in\!S^n$ is any point, are $\Z_2$-equivariant.
Since $\bB_{\phi,\fa}\al_1$ and $\bB_{\phi,\fa}\al_2$ correspond to the two standard generators
of $\pi_1(\R\P^n\!\times\!\R\P^n)$ by the above, 
$$[\al_1]^{\fa}\neq[\al_2]^{\fa}\in H_1^{\phi}(X;\Z),$$
illustrating the statement made in Remark~\ref{equivH1_rem}.\\

\noindent
If $c\!:S^1\!\lra\!S^1$ is an orientation-preserving involution, denote by 
$$(V_{\pm},\ti{c}_{\pm})\lra (S^1\!\times\!S^1,\id\!\times\!c)$$
the real bundle pairs with 
$$V_{\pm}=\big(\bI\!\times\!S^1\!\times\!\C\big)\big/\sim, \quad
(0,z,v)\sim(1,z,\pm v)~~~\forall~z\!\in\!S^1,\,v\!\in\!\C,$$
and with the involutions induced by the standard conjugation on~$\C$.

\begin{lem}\label{T2equivL_lmm}
Let $c\!:S^1\!\lra\!S^1$ be an orientation-preserving involution different 
from the identity.
For every $n\!\ge\!1$,
$$\blr{w_2^{n\ti{c}_+}(nV_+),[S^1\!\times\!S^1]^{\id\times c}}=0,\quad
\blr{w_2^{\ti{c}_-\oplus (n-1)\ti{c}_+}(V_-\!\oplus\!(n\!-\!1)V_+),
[S^1\!\times\!S^1]^{\id\times c}}\neq0\in\Z_2.$$
\end{lem}

\begin{proof} We can assume that $c(z)\!=\!-z$ on $S^1\!\subset\!\C$.
The real bundle pairs 
$$(V,\ti{c})=(nV_+,n\ti{c}_+), (V_-\!\oplus\!(n\!-\!1)V_+,\ti{c}_-\oplus (n\!-\!1)\ti{c}_+)$$ 
canonically decompose into two $\Z_2$-equivariant real vector bundles,  
induced by the real and imaginary axes in~$\C$.
By \eref{Vquot_e},
$$ \blr{w_2^{\ti{c}}(V),[S^1\!\times\!S^1]^{\id\times c}}
= \blr{w_2(V/\Z_2),[S^1\!\times\!\R\P^1]}
= \blr{w_2(V_{\R}\!\oplus\!V_{\fI\R}),[S^1\!\times\!\R\P^1]},$$
where $\R\P^1=S^1/\Z_2\equiv\bI/0\!\sim\!1$ and 
$$V_{\R}=n(V_+)_{\R},(V_-)_{\R}\!\oplus\!(n\!-\!1)(V_+)_{\R},\qquad
V_{\fI\R}=n(V_+)_{\fI\R},(V_-)_{\fI\R}\!\oplus\!(n\!-\!1)(V_+)_{\fI\R}$$
are the $\Z_2$-quotients of the real and imaginary parts of~$V$. 
Since 
\begin{alignat*}{3}
(V_{\pm})_{\R}&=\big(\bI\!\times\!\bI\!\times\!\R\big)\big/\sim,&\qquad
(0,t,v)&\sim(1,t,\pm v),~~(s,0,v)\sim(s,1,v)&\quad &\forall~s,t\!\in\!\bI,\,v\!\in\!\R,\\
(V_{\pm})_{\fI\R}&=\big(\bI\!\times\!\bI\!\times\!\R\big)\big/\sim,&\qquad
(0,t,v)&\sim(1,t,\pm v),~~(s,0,v)\sim(s,1,-v)&\quad &\forall~s,t\!\in\!\bI,\,v\!\in\!\R, 
\end{alignat*} 
we find that 
$$(V_+)_{\R}=\tau,\quad (V_-)_{\R}=\gm_1,\quad
(V_+)_{\fI\R}=\gm_2,\quad (V_-)_{\fI\R}=\gm_1\!\otimes\!\gm_2,$$
where $\tau\!\lra\!S^1\!\times\!\R\P^1$ is the trivial real line bundle
and $\gm_1,\gm_2\!\lra\!S^1\!\times\!\R\P^1$ are the pull-backs of 
the Mobius/tautological line bundle by the projection maps.
Thus,
\begin{equation*}\begin{split}
w_2\big(n(V_+/\Z_2)\big)&=w_2\big(n(\tau\!\oplus\!\gm_2)\big)
=\binom{n}{2}w_1(\gm_2)^2=0\in H^2(S^1\!\times\!\R\P^1,\Z_2),\\
w_2\big((V_-\!\oplus\!(n\!-\!1)V_+)/\Z_2\big)&=
w_2\big(\gm_1\!\oplus\!\gm_1\!\otimes\!\gm_2\big)
+w_1\big(\gm_1\!\oplus\!\gm_1\!\otimes\!\gm_2\big)w_1\big((n\!-\!1)\gm_2\big)
+w_2\big((n\!-\!1)\gm_2\big)\\
&=w_1(\gm_1)\big(w_1(\gm_1)\!+\!w_1(\gm_2)\big)+w_1(\gm_2)\cdot(n\!-\!1)w_1(\gm_2)+0\\
&=w_1(\gm_1)w_1(\gm_2)\neq0\in H^2(S^1\!\times\!\R\P^1,\Z_2).
\end{split}\end{equation*}
This implies the claim.
\end{proof}

\begin{lem}\label{Latop_lmm}
Let $c\!:S^1\!\lra\!S^1$ be an orientation-preserving involution different 
from the identity.
If $(V,\ti{c})\!\lra\!(S^1\!\times\!S^1,\id\!\times\!c)$ is a real bundle pair,
\BE{Latoplmm_e}\blr{w_2^{\La^{\top}_{\C}\ti{c}}(\La^{\top}_{\C}V),[S^1\!\times\!S^1]^{\id\times c}}
=\blr{w_2^{\ti{c}}(V),[S^1\!\times\!S^1]^{\id\times c}}.\EE
\end{lem}

\begin{proof}
We continue with the notation of Lemma~\ref{T2equivL_lmm} and its proof.
By the proof of \cite[Lemma~2.2]{Teh}, every real bundle pair $(\ti{V},\ti{c})$ 
over $\bI\!\times\!S^1$ admits a \textsf{trivialization}, i.e.~a fiber-preserving bundle
isomorphism
$$\Psi\!: \ti{V}\lra \bI\times S^1\times\C^n \quad\hbox{s.t.}\quad
\Psi\big(\ti{c}(\Psi^{-1}(t,z,w))\big)=\big(t,c(z),\ov{w})~~
\quad\forall~(t,z,w)\in\bI\!\times\!S^1\!\times\!\C^n,$$
where $\ov{w}$ denotes the standard complex conjugate of~$w$. 
Thus, the real bundle pairs~$(V,\ti{c})$ over
$$(S^1\!\times\!S^1,\id\!\times\!c)=
\big((\bI\!\times\!S^1)/(0,z)\!\sim\!(1,z),\id\!\times\!c\big)$$ 
are classified by the homotopy classes of the clutching maps $S^1\!\lra\!\GL_n\C$ 
satisfying the condition $A(c(z))\!=\!\ov{A(z)}$ for all~$z\!\in\!S^1$.
By \cite[Lemma~2.2]{Teh}, there are two homotopy classes of such maps;
they are represented by 
the constant maps with values in the diagonal matrices,
with at most one diagonal entry~-1 and the remaining diagonal entries~1.
Thus,
$$(V,\ti{c})\approx n(V_+,\ti{c}_+)\qquad\hbox{or}\qquad
(V,\ti{c})\approx(V_1,\ti{c}_-)\oplus (n\!-\!1)(V_+,\ti{c}_+).$$
Since
\begin{alignat*}{2}
(V,\ti{c})&=n(V_+,\ti{c}_+)
\quad&\Lra\quad \La^{\top}_{\C}(V,\ti{c})&=(V_+,\ti{c}_+), \\
(V,\ti{c})&=(V_-,\ti{c}_-)\oplus (n\!-\!1)(V_+,\ti{c}_+)
\quad&\Lra\quad \La^{\top}_{\C}(V,\ti{c})&=(V_-,\ti{c}_-),
\end{alignat*}
the claim thus follows from Lemma~\ref{T2equivL_lmm}.
\end{proof}

\begin{lem}\label{w2equiv_lmm}
Let $(X,\phi)$ be a manifold with an involution,
$(V,\ti\phi)\!\lra\!(X,\phi)$ be a real bundle pair,
and $\be\!:S^1\!\times\!S^1\!\lra\!X$
be a continuous map commuting with $\id\!\times\!c$ and~$\phi$,
where \hbox{$c\!:S^1\!\lra\!S^1$} is an orientation-preserving involution different 
from the identity.
The real bundle pair $\be^*\La^{\top}_{\C}(V,\ti\phi)$ 
over $(S^1\!\times\!S^1,\id\!\times\!c)$ admits a real square root if and only~if
\BE{w2SQRT_e}\blr{w_2^{\ti\phi}(V),[\be]^{\id\times c}}
=0.\EE
If $\pi_1(X)\!=\!0$, then $[\be]^{\id\times c}$ is the image of the homology class
of a map $\be'\!:S^2\!\lra\!X$ under the inclusion of $X\!\lra\!\bB_{\phi}X$ and   
$$\blr{w_2^{\ti\phi}(V),[\be]^{\id\times c}}= \blr{w_2(V),[\be']}.$$
\end{lem}

\begin{proof}
By Lemma~\ref{Latop_lmm}, we can assume that $\rk_{\C}V\!=\!1$.
Suppose $(L,\ti{c})\!\lra\!(S^1\!\times\!S^1,\id\!\times\!c)$ is a real bundle pair
such~that
$$\be^*(V,\ti\phi)\approx(L,\ti{c})^{\otimes2}.$$
By the proof of Lemma~\ref{Latop_lmm},
$$L=\big(\bI\!\times\!S^1\!\times\!\C\big)\big/\sim, \qquad
(0,z,v)\sim(1,z,\pm v)~~~\forall\,z\!\in\!S^1,\,v\!\in\!\C,$$
with the sign $\pm$ fixed and the conjugation~$\ti{c}$ induced by 
the standard conjugation in~$\C$.
Thus,
$$\be^*(V,\ti\phi)\approx
\big(\bI\!\times\!S^1\!\times\!\C\big)\big/\sim, \qquad
(0,z,v)\sim(1,z,v)~~~\forall\,z\!\in\!S^1,\,v\!\in\!\C,$$
i.e.~$\be^*(V,\ti\phi)\approx(V_+,\ti{c}_+)$.
Along with Lemma~\ref{T2equivL_lmm}, this implies~\eref{w2SQRT_e}.
On the other hand, by the proof of Lemma~\ref{Latop_lmm} and~\eref{equivback_e},
$$\blr{w_2^{\ti\phi}(V),[\be]^{\id\times c}}=0
\qquad\Lra\qquad
\be^*(V,\ti\phi)\approx(V_+,\ti{c}_+).$$
Thus, $\be^*(V,\ti\phi)$ is isomorphic to the square of $(V_+,\ti{c}_+)$
if $\blr{w_2^{\ti\phi}(V),[\be]^{\id\times c}}=0$.\\

\noindent
If $\pi_1(X)\!=\!0$, the fibration \eref{BXfib_e} gives rise to an exact sequence
$$\ldots\lra\pi_2(X)\lra\pi_2(\bB_{\phi}X)\lra0
\lra0\lra\pi_1(\bB_{\phi}X)\lra\Z_2\lra0.$$
Thus, the restriction of $\bB_{\phi,\id\times c}\be\!:S^1\!\times\!\R\P^1\lra\bB_{\phi}X$
to at least one simple circle is homotopically trivial (specifically, the circle $S^1\!\times\!x$). 
Therefore, $[\be]^{\id\times c}$ is a spherical class and thus equals $\io_*[\be']$
for some $\be'\!\in\!H_2(X)$, with $\Z_2$ or~$\Z$ coefficients,
where $\io\!:X\!\lra\!\bB_{\phi}X$ is the inclusion of a fiber in~\eref{BXfib_e}.
Thus,
$$\blr{w_2^{\ti{c}}(V),[\be]^{\id\times c}}
=\blr{\io^* w_2^{\ti{c}}(V),[\be']} =\blr{w_2(V),[\be']}.$$
This establishes the last claim.
\end{proof}

\subsection{Some examples}
\label{examples_sec}

\noindent
We now give three concrete examples. 
By Example~\ref{twistV_eg},  real bundle pairs $(V,\ti\phi)\!\lra\!(X,\phi)$ 
which induce non-orientable determinant bundle are quite common over non-simply connected
spaces.
Examples~\ref{prodorient_eg} and~\ref{evenP_eg} 
illustrate the significance of the vanishing requirements on~$\pi_1(X)$ and $w_2(X)$ 
in Corollary~\ref{etaorient_cor2}, showing that neither requirement by itself suffices
for the orientability of the moduli space $\mf_0(X,J,b)^{\phi,\eta}$.
The orientability in the former example in fact fails due to the twisting phenomenon
of Example~\ref{twistV_eg}.
 The $m,n\!=\!1$ case of Example~\ref{prodorient_eg} is \cite[Example~2.5]{Teh}.

\begin{example}\label{twistV_eg}
Let $(V,\ti\phi)\!\lra\!(X,\phi)$ be the trivial bundle pair of rank~1, i.e.
$$V=X\times\C, \qquad
\ti\phi\!:V\lra V, \quad \ti\phi(x,v)=\big(\phi(x),\bar{v}\big)~~
\forall\,(x,v)\!\in\!X\!\times\!\C\,,$$
and $L\!\lra\!Y$ be any real line bundle.
If $\pi_X,\pi_Y\!:X\!\times\!Y\!\lra\!X,Y$ are the projection maps,
$$(V_L,\ti\phi_L)\equiv\big(\pi_X^*V\!\otimes_{\R}\!\pi_Y^*L,
\pi_X^*\ti\phi\!\otimes_{\R}\!\pi_Y^*\id_L\big)
\lra (X\!\times\!Y,\phi\!\times\!\id_Y)$$
is a real bundle pair.
If 
$$\pi_X^{\phi},\pi_Y\!:\bB_{\phi\times\id_Y}(X\!\times\!Y)=(\bB_{\phi}X)\times\!Y\lra
\bB_{\phi}X,Y$$
are the projection maps,
\begin{gather*}
\bB_{\ti\phi_L}(V_L)=\pi_X^{\phi\,*}\bB_{\ti\phi}V\otimes_{\R}\pi_Y^*L
\qquad\Lra\\
w_2^{\ti\phi_L}(V_L)=\pi_X^{\phi\,*}w_2^{\ti\phi}(V)+
\pi_X^{\phi\,*}w_1^{\ti\phi}(V)\cdot\pi_Y^*w_1(L)+\pi_Y^*w_1(L)^2\,.
\end{gather*}
In particular, if $X\!=\!\{\pt\}$ and $L\!\lra\!Y\!=\!S^1$ is the Mobius band line bundle,
$$u\!:S^1\!\times\!S^1\lra X\!\times\!Y, \qquad (s,t)\lra (\pt,s),$$
is a continuous map intertwining the involutions $\phi\!\times\!\id_Y$
and $\id\!\times\!\fa$ on $S^1\!\times\!S^1$, where $\fa$ is the antipodal map, 
such that 
$$\blr{w_2^{\ti\phi_L}(V_L),[u]^{\id\times\fa}}\neq0.$$
\end{example}

\begin{example}\label{prodorient_eg}
Let $m,n\!\in\!\Z^+$ and $\tau_n$ be as in~\eref{tau2mdfn_e}. We define
\begin{equation*}\begin{split}
\eta_{2m-1}\!:\P^{2m-1}&\lra \P^{2m-1} \qquad\hbox{by}\\
\big[W_1,W_2,\ldots,W_{2m-1},W_{2m}\big]
&\lra\big[-\ov{W}_2,\ov{W}_1,\ldots,-\ov{W}_{2m},\ov{W}_{2m-1}\big].
\end{split}\end{equation*}
With $S^1\!\subset\!\C$ denoting the unit circle as before, we define
\begin{gather*}
\cdot:S^1\!\times\!\P^n\lra\P^n, \quad 
v\!\cdot\![Z_0,Z_1,\ldots,Z_{n-1},Z_n]=[Z_0,Z_1,\ldots,Z_{n-1},vZ_n],\\
X=S^1\!\times\!S^1\times\P^n\times\P^{2m-1}, \qquad
\phi\!:X\lra X,\quad \phi(u,v,z,w)=\big(\bar{u},v,\tau_n(v\!\cdot\!z),\eta_{2m-1}(w)\big),\\
Y=\big\{(v,z)\!\in\!S^1\!\times\!\P^n\!:\,\tau_n(v\!\cdot\!z)\!=\!z\big\}.
\end{gather*}
Since the non-trivial deck transformation of the double cover 
$$S^1\!\times\!\R\P^n\lra Y, \qquad
(v,z)\lra \big(v^2,v^{-1}\!\cdot\!z\big),$$
is orientation-reversing if $n$ is odd, $Y$ is not orientable for every~$n$
(if $n$ is even, the covering space is not orientable).
Let $B\!\in\!H_2(X;\Z)$ denote the homology class of a line in the last factor.
Since the projections 
$$\pi_1\!\times\!\pi_2\!\times\!\pi_3,\pi_4\!:X\lra S^1\!\times\!S^1\times\P^n,\P^{2m-1}$$
induce an isomorphism
$$\mf_0(X,B)^{\phi,\eta}\approx\{\pm1\}\!\times\!Y\times
\mf_0(\P^{2m-1},B)^{\eta_{2m-1},\eta}\,,$$
the moduli space $\mf_0(X,B)^{\phi,\eta}$ is not orientable.
Thus, the condition $\pi_1(X)\!=\!0$ in Corollary~\ref{etaorient_cor2}
cannot be dropped, even at the cost of requiring $c_1(TX)$ to be divisible by 
an arbitrarily high integer.
\end{example}

\begin{example}\label{evenP_eg}
If $n,d\!\in\!\Z^+$ and $\tau_n$ is as in~\eref{tau2mdfn_e}, $\cP(\P^n,d)^{\tau_n,\eta}$ 
 consists of maps of the~form
$$u\!:\P^1\lra\P^n, \qquad 
[x,y]\lra\big[p_0(x,y),p_1(x,y),\ldots,p_n(x,y)\big],$$
where $p_0,p_1,\ldots,p_n$ are degree~$d$ homogeneous  polynomials 
in two variables without a common factor.
The commutativity condition on~$u$ implies that this space is empty if~$d$ is~odd.
For $d$ even, this commutativity condition is equivalent~to
$$p_i(x,y)=
A_i\prod_{r=1}^{d/2}\big((a_{i;r}x\!-\!b_{i;r}y)(\bar{b}_{i;r}x\!+\!\bar{a}_{i;r}y)\big),$$
for some $A_i\in\C$ and $[a_{i;r},b_{i;r}]\in\P^1$ such that 
$$[A_0,A_1,\ldots,A_n]=[\bar{A}_0,\bar{A}_1,\ldots,\bar{A}_n]\in\P^n.$$
Let $\De_{n+1;d}^{\eta}\subset(\Sym^{d/2}\C)^{n+1}$ denote the image of the
set 
$$\Big\{((b_{0;1},\ldots,b_{0;\frac{d}{2}}),\ldots,(b_{n;1},\ldots,b_{n;\frac{d}{2}}))
\!\in\!(\C^{d/2})^{n+1}\!:\,\bigcap_{i=0}^n
\{b_{i;1},-\bar{b}_{i;1}^{-1},\ldots,b_{i;\frac{d}{2}},-\bar{b}_{i;\frac{d}{2}}^{-1}\}\neq\eset\Big\}$$
under the quotient map $(\C^{d/2})^{n+1}\lra(\Sym^{d/2}\C)^{n+1}$.
The map
\begin{gather*}
\R\P^n\times \big((\Sym^{d/2}\C)^{n+1}-\De_{n+1;d}^{\eta}\big) \lra
\cP(\P^n,d)^{\tau_n,\eta}\,,\\
\begin{split}
&\big([A_0,\ldots,A_n],[b_{0;1},\ldots,b_{0;d/2}],\ldots,[b_{n;1},\ldots,b_{n;d/2}]\big)\\
&\qquad\qquad\lra 
\bigg[A_0\prod_{r=1}^{d/2}\big((x\!-\!b_{0;r}y)(\bar{b}_{0;r}x\!+\!y)\big),
\ldots,
A_n\prod_{r=1}^{d/2}\big((x\!-\!b_{n;r}y)(\bar{b}_{n;r}x\!+\!y)\big)\bigg],
\end{split}\end{gather*}
is an isomorphism over the open subset
of $\cP(\P^n,d)^{\tau_n,\eta}$ consisting of maps~$u$ such that 
$u([1,0])$ does not lie in any of the coordinate subspaces of~$\P^n$;
the subset $\De_{n+1;d}^{\eta}$ corresponds to polynomials with common factors and
thus does not correspond to maps to~$\P^n$.
Since $\R\P^n$ is not orientable if $n$ is even, it follows that $\cP(\P^n,d)^{\tau_n,\eta}$ 
is not orientable and neither is $\mf_0(\P^n,d)^{\tau_n,\eta}$.
Thus, the condition $w_2(TX)\!=\!0$ in Corollary~\ref{etaorient_cor2}
cannot simply be dropped.
\end{example}

\section{Fredholm theory}
\label{double_sec}

\noindent
This section describes doubling constructions for oriented sh-surfaces and 
shows that real Cauchy-Riemann operators over such surfaces are
Fredholm if the complex structure on the domain is compatible with the involution.
In the case the boundary involution is trivial, the results in this section specialize
to results in \cite[Section~3]{KatzLiu}.
However, in contrast to the situation in \cite[Section~3]{KatzLiu}, not
every complex structure on an oriented sh-surface can be doubled
and not every real Cauchy-Riemann operator is Fredholm.
Corollary~\ref{dblJ_cor} below describes a necessary and sufficient
condition for doubling a complex structure;
it can be viewed as directly capturing the bianalytic nature of 
the doubling construction for Klein surfaces in  \cite[Section~1.6]{AG}.
Remark~\ref{notFred_rem} provides examples of real Cauchy-Riemann operators
that are not Fredholm.\\

\noindent
Let $\fJ$ be an almost complex structure on a bordered Riemann surface $\Si$.
We call a smooth~chart 
$$\psi\!:(U,U\!\cap\!\prt\Si)\lra (\bH,\R),
\qquad\hbox{where}\quad \bH=\{z\!\in\!\C\!:~\Im\,z\ge0\},$$
\textsf{$\fJ$-holomorphic} if $\fJ\!=\!\psi^*\fJ_0$, where $\fJ_0$
is the standard complex structure on~$\C$.
By Corollary~\ref{sh_cor}, $\Si$ can be covered by such charts, and
so $(\Si,\fJ)$ is a Riemann surface in the sense of 
\cite[Definition~3.1.4]{KatzLiu} and \cite[Definition~2.5]{Melissa}.
We thus call $\fJ$ simply a \textsf{complex structure} on~$\Si$.\\

\noindent
Let $(\Si,c)$ be an oriented sh-surface.
If $\fJ$ is a complex structure on~$\Si$, we call $c$ 
\textsf{real-analytic with respect to~$\fJ$} if for every $z\!\in\!\prt\Si$
there exist $\fJ$-holomorphic charts
\BE{charts_e}
\psi_z\!:U_z\lra U_z' \qquad\hbox{and}\qquad \psi_{c(z)}\!:U_{c(z)}\lra U_{c(z)}',\EE
where $U_z$ and $U_{c(z)}$ are open subsets of $\Si$ containing $z$ and $c(z)$,
respectively, and  $U_z'$ and $U_{c(z)}'$ are open subsets of~$\bH$,
such~that 
\BE{Rover_e}\psi_{c(z)}\circ c\circ\psi_z^{-1}\!:~
  \psi_z\big(U_z\cap c(U_{c(z)}\!\cap\!\prt\Si)\big)\lra\R\EE
is a real-analytic function on an open subset of $\R\!\subset\!\C$.
In particular, $\id_{\prt\Si}$ is a real-analytic involution with respect to 
any complex structure~$\fJ$ on $\Si$ and so $\cJ_{\id_{\Si}}\!=\!\cJ_{\Si}$.
If $c$ is real-analytic with respect to~$\fJ$,
then \eref{Rover_e} is real-analytic for any choice of $\fJ$-holomorphic
charts as in~\eref{charts_e}.
The following lemma describes an important property of the collection~$\cJ_c$ 
of complex structures on~$\Si$ for which $c$ is real-analytic with respect to~$\fJ$. 

\begin{lem}\label{cDcJ_lem}
Let $(\Si,c)$ be an oriented sh-surface.
If $\fJ\!\in\!\cJ_c$ and $h\!\in\!\cD_c$, then $h^*\fJ\!\in\!\cJ_c$.
\end{lem}

\begin{proof}
If $\{\psi_z\!:U_z\!\lra\!\bH\}$ are the analytic charts for $(\Si,\fJ)$,  then 
$$h^*\psi_z\equiv\psi_z\!\circ\!h:h^{-1}(U_z)\lra \bH$$
are the analytic charts for $(\Si,h^*\fJ)$.
Since
$$h^*\psi_{c(z)}\circ c\circ \{h^*\psi_z\}^{-1}
=\psi_{c(z)}\circ c\circ \psi_z^{-1}: 
\psi_z\big(U_z\cap c(U_{c(z)}\!\cap\!\prt\Si)\big)\lra\R$$
is real-analytic (because $c$ is real-analytic with respect to~$\fJ$),
it follows that $c$ is real-analytic with respect to~$h^*\fJ$.
\end{proof}

\begin{lem}\label{Rover_lem}
Let $(\Si,c)$ be an oriented sh-surface.
If $\fJ$ is a complex structure on~$\Si$ such that $c$ is real-analytic with respect
to~$\fJ$, then for every $z\!\in\!\prt\Si$ there exist  $\fJ$-holomorphic
charts as in~\eref{charts_e} such that 
\BE{overlap_e}
c\big(U_z\!\cap\!\prt\Si\big)=U_{c(z)}\!\cap\!\prt\Si \qquad \hbox{and}\qquad
\psi_z|_{U_z\cap\prt\Si}=\psi_{c(z)}\circ c.\EE
\end{lem}

\begin{proof} The first condition in~\eref{overlap_e} can be achieved by shrinking 
the charts.
If $\psi_c$ and~$\psi_{c(z)}$ are $\fJ$-holomorphic charts as in~\eref{charts_e}
which satisfy the first equation in~\eref{overlap_e},
$$g\equiv\psi_{c(z)}\circ c\circ\psi_z^{-1}\!:~
  \psi_z\big(U_z\!\cap\!\prt\Si\big)\lra
  \psi_{c(z)}\big(U_{c(z)}\!\cap\!\prt\Si\big)$$
is a real-analytic orientation-preserving diffeomorphism between open subsets of~$\R$.
Let $\ti{g}\!:W'\!\lra\!\C$ be an extension of $g$ to a holomorphic map on 
a neighborhood~$W'$ of $\psi_z(U_z\!\cap\!\prt\Si)\!\times\!0$ in 
$\psi_z(U_z\!\cap\!\prt\Si)\!\times\!\R$.
Since $g$ is an orientation-preserving diffeomorphism between open subsets of~$\R$
(because $c$ is orientation-preserving), we can assume that $\ti{g}$ is a biholomorphic 
map taking $W'$ onto a neighborhood~$W''$ of $\psi_{c(z)}(U_{c(z)}\!\cap\!\prt\Si)\!\times\!0$
in $\psi_{c(z)}(U_{c(z)}\!\cap\!\prt\Si)\!\times\!\R$
and $W'\!\cap\!\bH$ onto $W''\!\cap\!\bH$.
Replacing $\psi_z$ with $\ti{g}\!\circ\!\psi_z|_{\psi_z^{-1}(W')}$,
we obtain a $\fJ$-holomorphic chart around~$z$ on~$\Si$ satisfying~\eref{overlap_e}.
\end{proof}

\noindent
Let $(\hat\Si,\hat{c})$ be the double of $(\Si,c)$ as in~\eref{dblSi_e}.

\begin{cor}\label{dblJ_cor}
Let $(\Si,c)$ be an oriented sh-surface and $\fJ\!\in\!\cJ_{\Si}$.
There exists a complex structure $\hat\fJ$ on~$\hat\Si$ so that 
$\hat\fJ|_{\Si}\!=\!\fJ$ and $\hat{c}^*\hat\fJ\!=\!-\hat\fJ$ if and only if 
$c$ is real-analytic with respect to~$\fJ$.
\end{cor}

\begin{proof}
(1) Suppose $c$ is real-analytic with respect to~$\fJ$.
Given $z\!\in\!\prt\Si$, choose $\fJ$-holomorphic charts on~$\Si$ as in Lemma~\ref{Rover_lem}.
The first condition in~\eref{overlap_e} implies that 
\BE{Wz_e}W_{[z]}\equiv \big(\{+\}\!\times\! U_z\sqcup \{-\}\!\times\!U_{c(z)}\big)\big/\!\!\sim
~\subset \hat\Si\EE
is an open subset.
The second condition in~\eref{overlap_e} implies that the map
$$\Psi_{[z]}\!:W_{[z]}\lra\C, \qquad 
x\lra\begin{cases}
\psi_z(z'),&\hbox{if}~x\!=\![+,z'],~z'\!\in\!U_z;\\
\ov{\psi_{c(z)}(z')},&\hbox{if}~x\!=\![-,z'],~z'\!\in\!U_{c(z)};
\end{cases}$$
is well-defined (agrees on the overlap of the two cases, which is when 
$z'\!\in\!U_z\!\cap\!\prt\Si$).
This map is a homeomorphism onto the open subset $\psi_z(U_z)\!\cup\!\ov{\psi_{c(z)}(U_{c(z)})}$
of~$\C$.
If $(\psi_z',\psi_{c(z)}')$ is another pair of charts as above with the same domains, 
the overlap map is given~by
$$\Psi_{[z]}'\!\circ\!\Psi_{[z]}^{-1}\!:  \Psi_{[z]}(W_{[z]})\lra \Psi_{[z]}'(W_{[z]}'),\quad
x\lra\begin{cases}
\psi_z'(\psi_z^{-1}(x)),&\hbox{if}~x\in\Psi_{[z]}(W_{[z]})\cap\bH;\\
\ov{\psi_{c(z)}'(\psi_{c(z)}^{-1}(\bar{x}))},&\hbox{if}~x\in\Psi_{[z]}(W_{[z]})\cap\ov\bH.
\end{cases}$$
Thus, the collection of our charts induces a complex structure on $\hat\Si$ that agrees
with $\fJ$ on $\Si^+$ and $-\fJ$ on $\Si^-$, as required.\\
 
\noindent
(2) Suppose there exists a complex structure $\hat\fJ$ on~$\hat\Si$ 
so that $\hat\fJ|_{\Si}\!=\!\fJ$ and $\hat{c}^*\hat\fJ\!=\!-\hat\fJ$.
By deforming~$\fJ$ away from~$\prt\Si$ and collapsing circles close to~$\prt\Si$,
we can assume that $\Si\!=\!D^2$ and so $\hat\Si\!=\!\P^1$.
Since $\fJ\!=\!\hat\fJ|_{\Si}$, the standard $\dbar$-operator $\dbar_0$ on
the trivial real bundle pair  $(D^2\!\times\!\C,\ti{c}_1)\!\lra\!(D^2,c)$ 
is surjective and Fredholm, by the commutativity property
used in the proof of Proposition~\ref{etaFred_prp}.
Remark~\ref{notFred_rem} then implies that~$c$ is real-analytic with respect to~$\fJ$.
\end{proof}

\begin{remark}\label{dbl_rem}
The image of $\prt\Si$ in $\hat\Si$ is an analytic curve with 
respect to the doubled complex structure~$\hat\fJ$ constructed in~(1)
of the proof of Corollary~\ref{dblJ_cor}: there are charts on~$\hat\Si$
taking this curve to~$\R\!\subset\!\C$.
There can be other complex structures~$\hat\fJ$ satisfying the requirements of
Corollary~\ref{dblJ_cor} for which the image of $\prt\Si$ is not analytic;
they induce different smooth structures on~$\hat\Si$ across~$\prt\Si$.
For example, let $\eta\!:\P^1\!\lra\!\P^1$ be as in~\eref{etadfn_e}.
Choose any simple curve in the upper-hemisphere of~$\P^1$ with ends at 
a pair of antipodal points on the equator.
Using~$\eta$ to double the curve, we obtain a simple closed curve which splits~$\P^1$
into two halves interchanged by~$\eta$; each half is a disk with boundary involution
induced by~$\eta$.
For a generic choice of the arc, the closed curve is not real-analytic with respect 
to the standard complex structure on~$\P^1$.
One can ensure that this curve is smooth at the junction by requiring it to run along
the equator near its ends.
\end{remark}

\noindent
A real bundle pair $(V,\ti{c})\!\lra\!(\Si,c)$ doubles to a complex bundle 
$$\hat{V}
\equiv\big(\{+\}\!\times\!V\sqcup \{-\}\!\times\!\bar{V}\big)\big/\!\sim,\qquad
(+,v)\sim\big(-,\ti{c}(v)\big)~~~\forall\,v\!\in\!V|_{\prt\Si},$$
over $\hat\Si$ with conjugation $\breve{c}\!:\hat{V}\!\lra\!\hat{V}$ lifting 
$\hat{c}\!:\hat\Si\!\lra\!\hat\Si$, where $\bar{V}$ denotes the same real vector bundle
over~$\Si$ as~$V$, but with the opposite complex structure on the fibers.
We define the \textsf{Maslov index of~$(V,\ti{c})$}~by
$$\mu(V,\ti{c})=\lr{c_1(\hat{V}),[\hat\Si]}.$$
By \cite[Theorem~C.3.5 and~(C.3.4)]{MS}, this agrees with the usual definition of
the Maslov index of $(V,V^{\ti{c}})$ if $c\!=\!\id_{\prt\Si}$.
By \cite[Propositions~4.1,~4.2]{BHH}, real bundle pairs $(V,\ti{c})\!\lra\!(\Si,c)$
are in fact classified by their rank, the Maslov index, and the orientability of $V^{\ti{c}}$
over each boundary component $(\prt\Si)_i$ with $|c_i|\!=\!0$.
For the sake of completeness, we confirm this for $\Si\!=\!D^2$,
which is the only case needed for the purposes of this paper.

\begin{lem}\label{Maslovind_lmm}
Let $c$ be an orientation-preserving involution on $\prt D^2\!=\!S^1$.
The Maslov index classifies the real bundle pairs $(V,\ti{c})\!\lra\!(D^2,c)$.
In particular, a rank~$n$ real bundle pair $(V,\ti{c})\!\lra\!(D^2,c)$ is 
isomorphic to the trivial one, i.e.
$$(D^2\!\times\!\C^n,\ti{c}_n)\lra(D^2,c), \qquad 
\ti{c}_n(z,v)=\big(c(z),\bar{v}\big)\quad\forall\,(z,v)\!\in\!S^1\!\times\!\C^n,$$
if and only if $\mu(V,\ti{c})\!=\!0$.
\end{lem}

\begin{proof}
If $c\!=\!\id_{S^1}$, this statement follows immediately from 
\cite[Lemma~C.3.8, Corollary~C.3.9]{MS} and the Normalization Property of 
the Maslov index in \cite[Theorem~C.3.5]{MS}.
Thus, we can assume that $c$ is the antipodal map on $S^1\!\subset\!\C$, 
$V\!=\!D^2\!\times\!\C^n$, and 
$$\ti{c}\!:S^1\!\times\!\C^n\lra S^1\!\times\!\C^n$$
is a conjugation covering~$c$.\\

\noindent 
By \cite[Lemma~2.2]{Teh}, there exists $A\!:S^1\!\lra\!\GL_n\C$ such that 
$$\ti{c}(z,v)=\big(-z,A(-z)\ov{A(z)^{-1}v}\big) \qquad\forall~(z,v)\!\in\!S^1\!\times\!\C^n\,.$$
The loops $A_d\!:S^1\!\lra\!\GL_n\C$ sending $z\!\in\!S^1$ to the diagonal matrix
with the first entry $z^d$ and the remaining entries~1 represent the elements
of $\pi_1(\GL_n\C)\!\approx\!\Z$.
Thus, there exists $d\!\in\!\Z$ so that the~map
$$(S^1,1)\lra \big(\GL_n(\C),\bI_n\big),\qquad z\lra A_d(z)A(z)^{-1},$$
is homotopically trivial (with basepoints fixed) and therefore extends to a smooth~map
\hbox{$\Psi\!:D^2\!\lra\!\GL_n\C$}. 
The bundle isomorphism
$$D^2\!\times\!\C^n\lra D^2\!\times\!\C^n, \qquad (z,v)\lra \big(z,\Psi(z)v\big),$$
identifies the real bundle pair $(D^2\!\times\!\C^n,\ti{c})$ with the real bundle pair
$(D^2\!\times\!\C^n,\ti{c}_{n,d})$, where
$$\ti{c}_{n,d}(z,v)=\big(-z,A_d(-z)\ov{A_d(z)^{-1}v}\big) \qquad\forall~(z,v)\!\in\!S^1\!\times\!\C^n\,.$$
\\

\noindent 
The double of $(D^2\!\times\!\C^n,\ti{c}_{n,d})$,
$$\hat{V}\equiv\big(D_+^2\!\times\!\C^n\sqcup 
D_-^2\!\times\!\bar\C^n\big)\big/\!\sim,\quad
(+,z,v)\sim\big(-,-z,A_d(-z)\ov{A_d(z)^{-1}v}\big)~~~\forall\,(z,v)\!\in\!S^1\!\times\!\C^n\,,$$
has trivializations
\begin{alignat*}{2}
\hat{V}|_{D_+^2}&\lra D_+^2\!\times\!\C^n, &\qquad [+,z,v]&\lra\big([+,z],v\big),\\
\hat{V}|_{D_-^2}&\lra D_-^2\!\times\!\C^n, &\qquad [-,z,v]&\lra\big([-,z],\bar{v}\big).
\end{alignat*}
The overlap between these trivializations is given by
\begin{equation*}\begin{split}
D_+^2\!\cap\!D_-^2\times\C^n&\lra D_+^2\!\cap\!D_-^2\times\C^n, \\
\big([+,z],v\big)&\lra \big([+,z],\ov{A_d(-z)}A_d(z)^{-1}v\big)
=\big([+,z],(-1)^dA_{-2d}(z)v\big).
\end{split}\end{equation*}
Thus, $\hat{V}\!\approx\!\cO(2d)\oplus(n\!-\!1)\cO$,
where $\cO,\cO(2d)\!\!\lra\!\P^1$ are the trivial complex line bundle and 
the $2d$-th power of the hyperplane line bundle, respectively.
If follows that 
$$\mu\big(D^2\!\times\!\C^n,\ti{c}_{n,d}\big)\equiv
\blr{c_1(\hat{V}),[\P^1])}=2d.$$
This establishes both claims of the proposition.
\end{proof}


\noindent
Let $\fJ$ be a complex structure on $\Si$ so that $c$ is real-analytic with respect 
to~$\fJ$.
Every real Cauchy-Riemann operator~$D$ on $(V,\ti{c})\!\lra\!(\Si,c)$ 
as  in~\eref{CRdfn_e} compatible with~$\fJ$  \textsf{doubles} to 
a real Cauchy-Riemann operator
$$\hat{D}\!:W^{1,p}(\hat\Si;\hat{V})\lra
 W^{0,p}\big(\hat\Si;(T^*\hat\Si,\hat\fJ)\!\otimes_{\C}\!\hat{V}\big),$$
where $p\!>\!2$ and $W^{1,p}$ and $W^{0,p}$ denote Sobolev completions with respect to some metrics
on~$\Si$ and~$V$ (doubled to $\hat\Si$ and~$\hat{V}$) 
on the appropriate spaces of bundle sections, by
$$\hat{D}\xi\big|_{\Si^+} =D\xi,\quad
\hat{D}\xi\big|_{\Si^-} =
\breve{c}\circ  D(\breve{c}\!\circ\!\xi\!\circ\!\hat{c})\circ\tnd\hat{c}
\qquad\forall\,\xi\in W^{1,p}(\hat\Si;\hat{V}).$$
Since the image of $\hat{D}$ lies in $W^{0,p}$, 
there is no overlap condition for $\hat{D}$ along $\prt\Si$ to be checked.
This operator satisfies
$$\hat{D}(\breve{c}\!\circ\!\xi\!\circ\!\hat{c})=  \breve{c}\circ\{\hat{D}\xi\}\circ\tnd\hat{c}.$$
In particular, $\hat{D}$ takes the complementary subspaces
\begin{equation*}\begin{split}
W^{1,p}(\hat\Si;\hat{V})^{\ti{c}}\equiv
\big\{\xi\!\in\!W^{1,p}(\hat\Si;\hat{V})\!:~\breve{c}\!\circ\!\xi\!\circ\!\hat{c}=\xi\big\},\\
\fI\,W^{1,p}(\hat\Si;\hat{V})^{\ti{c}}=
\big\{\xi\!\in\!W^{1,p}(\hat\Si;\hat{V})\!:~\breve{c}\!\circ\!\xi\!\circ\!\hat{c}=-\xi\big\}
\end{split}\end{equation*}
of $W^{1,p}(\hat\Si;\hat{V})$ to the complementary subspaces
\begin{equation*}\begin{split}
W^{0,p}\big(\hat\Si;(T^*\hat\Si,\hat\fJ)\!\otimes_{\C}\!\hat{V}\big)^{\ti{c}}\equiv
\big\{\eta\!\in\!W^{0,p}\big(\hat\Si;(T^*\hat\Si,\hat\fJ)\!\otimes_{\C}\!\hat{V}\big)\!:~
\ti{c}\!\circ\!\eta\!\circ\!\tnd\hat{c}=\eta\big\},\\
\fI\,W^{0,p}\big(\hat\Si;(T^*\hat\Si,\hat\fJ)\!\otimes_{\C}\!\hat{V}\big)^{\ti{c}}=
\big\{\eta\!\in\!W^{0,p}\big(\hat\Si;(T^*\hat\Si,\hat\fJ)\!\otimes_{\C}\!\hat{V}\big)\!:~
\ti{c}\!\circ\!\eta\!\circ\!\tnd\hat{c}=-\eta\big\},
\end{split}\end{equation*}
respectively, of $W^{0,p}(\hat\Si;(T^*\hat\Si,\hat\fJ)\!\otimes_{\C}\!\hat{V})$.

\begin{prop}\label{etaFred_prp}
Let $(\Si,c)$ be an oriented sh-surface, $(V,\ti{c})\!\lra\!(\Si,c)$ be a real bundle pair, 
$\fJ\!\in\!\cJ_c$, and $p\!>\!2$.
A real Cauchy-Riemann operator~$D$ on~$(V,\ti{c})$ compatible with $\fJ$
induces a Fredholm operator between 
$W^{1,p}$ and $W^p$-completions of its domain and target, respectively, with
$$\ind_{\R}D=\mu(V,\ti{c})+\big(1\!-\!g(\hat\Si)\big)(\rk_{\C}V),$$
where $g(\hat\Si)$ is the genus of $\hat\Si$.
Furthermore, the kernel of the standard $\dbar$-operator on the real bundle pair
$(\Si\!\times\!\C^n,\ti{c})\!\lra\!(\Si,c)$
with $\ti{c}$ induced by the standard conjugation on~$\C^n$
consists of constant $\R^n$-valued functions on~$\Si$;
this operator is surjective if $\Si\!=\!D^2$.
\end{prop}

\begin{proof}
Since $c$ is real-analytic with respect to the complex structure~$\fJ$ corresponding to~$D$,
we have a commutative diagram
$$\xymatrix{W^{1,p}(\hat\Si;\hat{V})^{\ti{c}}\ar[r]^>>>>>{\hat{D}}\ar[d] &
W^{0,p}\big(\hat\Si;(T^*\hat\Si,\hat\fJ)\!\otimes_{\C}\!\hat{V}\big)^{\ti{c}}\ar[d]\\
W^{1,p}(\Si;V)^c\ar[r]^>>>>>{D} &
W^{0,p}\big(\Si;(T^*\Si,\fJ)\!\otimes_{\C}\!V\big)
}$$
where the vertical arrows are the restriction isomorphisms $\xi\!\lra\!\xi|_{\Si^+}$.
Since $\hat{D}$ preserves the $\pm1$-eigenspaces of~$\ti{c}$, $W^{k,p}$ and $\fI W^{k,p}$ above,
this diagram induces isomorphisms
$$\ker \hat{D}^+\approx\ker D,\qquad \Im\,\hat{D}^+\approx\Im\,D, \qquad
\cok\,\hat{D}^+\approx\cok\, D,$$
where $\hat{D}^{\pm}$ is the restriction of $\hat{D}$ to the $\pm1$-eigenspace of~$\ti{c}$.
Since $\hat{D}$ is Fredholm, it follows that so is~$D$.
The index of~$D$ is the same as the index of its $\C$-linear part~$D^{1,0}$.
Since multiplication by~$\fI$ commutes with~$\hat{D}^{1,0}$,
it induces isomorphisms 
$$\ker(\hat{D}^{1,0})^+\lra\ker(\hat{D}^{1,0})^-, \qquad 
\cok\,(\hat{D}^{1,0})^+\lra\cok\,(\hat{D}^{1,0})^-.$$
Thus,
$$\ind_{\R}D=\ind_{\R}D^{1,0}
=\ind_{\R}(\hat{D}^{1,0})^+
=\frac12\ind_{\R}\hat{D}^{1,0}
=\lr{c_1(\hat{V}),[\hat\Si]}+(1\!-\!g(\hat\Si))(\rk_{\C}V);$$
the last equality follows from Riemann-Roch for a closed complex curve;
see \cite[Theorem~C.1.10(ii)]{MS}, for example.\\

\noindent
If $D$ is the standard $\dbar$-operator on the trivial real bundle pair
$(\Si\!\times\!\C^n,\ti{c}_n)\!\lra\!(\Si,c)$,
$\hat{D}$ is the standard $\dbar$-operator in a trivial vector bundle over~$\hat\Si$
with the standard conjugation. This implies the last claim.
\end{proof}

\begin{remark}\label{notFred_rem}
We now show that the standard $\dbar$-operator $\dbar_0$ on 
$(D^2\!\times\!\C,\ti{c})\!\lra\!(D^2,c)$, where 
$\ti{c}$ is the lift of the involution~$c$ on $S^1\!=\!\prt\Si$ induced by the standard
conjugation in~$\C$, has infinite-dimensional cokernel if $c$ is not real-analytic
with respect to the standard complex structure~$\fJ_0$ on~$D^2$.
Specifically, we show that 
$$\big\{\bar{z}^{2k-1}\tnd\bar{z}:\,k\!\in\!\Z^+\big\}
\cap\dbar_0\big(W^{1,p}(D^2)^{\ti{c}}\big)=\{0\}\subset W^{0,p}(D^2;(T^*D^2,\fJ_0))$$
if $c$ is not real-analytic.
If $k\!\in\!\Z^+$,
$$\big\{f\!\in\!W^{1,p}(D^2)\!:\dbar f=k\bar{z}^{2k-1}\tnd\bar{z}\big\}
=\big\{\Re\,z^{2k}\!+\!h\!:~h\!\in\!\Hol(D^2)\big\},$$
where $\Hol(D^2)$ is the space of continuous maps on the closed disk $D^2$ that 
are holomorphic in the interior.
The condition that $\Re\,z^{2k}\!+\!h$ lies in $W^{1,p}(D^2)^{\ti{c}}$ is equivalent~to
$$\Re\big(z^{2k}\!+\!h(z)\big)=\Re\big(c(z)^{2k}\!+\!h(c(z))\big), \qquad
\Im\,h(z)=-\Im\,h(c(z)).$$
The functions $\Re(z^{2k}\!+\!h(z))$ and $\Im\,h(z)$ are real-analytic on~$S^1$.
If $c$ is not real-analytic, the above conditions imply that 
\BE{analcond_e}\Re\big(z^{2k}\!+\!h(z)\big)=C, 
\quad  \Im\,h(z)=0\qquad\forall\,z\!\in\!S^1,\EE
for some $C\!\in\!\R$, which we can take to be~0.
Indeed, if $f(z)=\Re(z^{2k}\!+\!h(z)),\Im\,h(z)$ were not constant on~$S^1$,
we could choose an analytic coordinate~$\th$ near any point~$\th_0$ on~$S^1$ 
and an analytic coordinate~$\vt$ near the point $\vt_0\!=\!c(\th_0)$ on~$S^1$ 
so~that 
$$f(\vt)-f(\vt_0)=\pm\vt^m, \qquad f(c(\th))-f(\vt_0)=\pm\th^n$$
for some $m,n\!\in\!\Z^+$. 
Since $c$ is smooth, $n|m$ and so $\vt(c(\th))=\pm\th^{n/m}$ is real-analytic at $\th_0$.  
This confirms~\eref{analcond_e}.
Finally, \eref{analcond_e} with $C\!=\!0$ implies that 
$$h^{\lr{m}}(0)=\frac{m!}{2\pi\fI}\oint_{|z|=1}\frac{-\Re\,z^{2k}\tnd z}{z^{m+1}}
=\begin{cases} -\frac{(2k)!}{2}&\hbox{if}~m\!=\!2k;\\ 0,&\hbox{otherwise}. 
\end{cases}$$
Thus, $h(z)=-\frac12z^{2k}$, which contradicts \eref{analcond_e}.
\end{remark}

\section{Proofs of main statements} 
\label{mainpf_sec}

\noindent
We begin this section by recalling some standard facts concerning determinant lines of 
real Cauchy-Riemann operators, rephrasing the first half of \cite[Section~2]{Ge}
in terms of real bundle pairs $(V,\ti{c})\!\lra\!(\Si,c)$,
instead of bundles $(V,V^{\ti{c}})\!\lra\!(\Si,\prt\Si)$ of the $|c|_1\!=\!0$ case.
We then deduce Theorem~\ref{main_thm} from \cite[Theorem~1.1]{Ge} and 
Lemma~\ref{eta_lem}.
The latter treats a very special case of Theorem~\ref{main_thm} and is the analogue
of \cite[Lemmas~3.4,~3.6]{Ge} for the non-trivial involutions~$c_i$ on~$(\prt\Si)_i$.
We conclude this section with a set of lemmas extending \cite[Lemmas~2.2-2.4]{Ge} 
to arbitrary boundary involutions~$c$.\\

\noindent
A short exact sequence of Fredholm operators
$$\begin{CD}
0
@>>>X'@>>>X@>>>X''@>>>0 \\
@. @V V D' V@VV D V@VV D'' V@.\\
0@>>> Y'@>>>Y@>>>Y''@>>>0
\end{CD}$$
determines a canonical isomorphism
\BE{sum} \det D\approx (\det D')\otimes (\det D'').\EE
For a continuous family of Fredholm operators $D_t:X_t\!\lra\!Y_t$  parametrized by
a topological space $B$, the determinant lines $\det D_t$ form a line bundle
over $B$; see \cite[Section A.2]{MS} and~\cite{detLB}. For a short exact sequence of such
families, the isomorphisms (\ref{sum}) give rise to a canonical isomorphism
between determinant line bundles.\\

\noindent
Let $(\Si,c)$ be an oriented, possibly nodal, sh-surface, with nodes away from the boundary
and $\fJ\!\in\!\cJ_c$.
Let $\pi\!:\wt\Si\!\lra\!\Si$ be the normalization of~$\Si$.
Fix an ordering of the boundary components of~$\Si$ (and thus of~$\wt\Si$)
and of the nodes of~$\Si$. 
A real Cauchy-Riemann operator~$D$ on a real bundle pair $(V,\ti{c})\!\lra\!(\Si,c)$ corresponds 
to a real Cauchy-Riemann operator $\wt{D}=\bigoplus_iD^i$ on 
$\pi^*(V,c)\!\lra\!(\wt\Si,c)$, 
where the sum is taken over the components of~$\wt\Si$.
Thus, by~\eqref{sum}, there is a canonical isomorphism
$$\det \wt{D}\approx\bigotimes_i(\det D^i)  .$$
On the other hand, gluing together punctured disks around the nodes~$z_j$
of~$\Si$, we obtain a smooth surface $\Si^{\ve}$ and 
a real Cauchy-Riemann operator $D^{\ve}$ over $(\Si^{\ve},c)$ 
for a gluing parameter~$\ve$. 
Similarly to \cite[Appendix~D.4]{Huang} and \cite[Section 3.2]{EES}, 
for every sufficiently small $\ve$ there is a canonical (up to homotopy) isomorphism
\BE{dets} 
\det D^\ve \approx(\det\wt{D})\otimes
\La_{\R}^{\top}\bigg(\bigoplus_jV_{z_j}\bigg)^*.\EE
Moreover, the gluing maps satisfy an associativity property: the isomorphism~\eqref{dets} 
is independent of the order in which we smooth the nodes.

\begin{remark}\label{con_rmk}
The space of real Cauchy-Riemann operators on $(V,\ti{c})\!\lra\!(\Si,c)$ is
contractible; thus,  a choice of orientation on one determinant line canonically
induces orientations on the rest.  Any two families of real
Cauchy-Riemann operators on a family $(V_t,\ti{c}_t)\!\lra\!(\Si_t,c_t)$ are
fiberwise homotopic. This implies that their determinant bundles have the same
Stiefel-Whitney classes.
\end{remark}

\begin{proof}[{\bf\emph{Proof of Theorem \ref{main_thm}}}]
By Lemma \ref{def_lm}, we can assume that~$\psi$ restricts to the identity in a
neighborhood of the boundary.
For each boundary component $(\prt\Si)_i$ of~$\Si$ with $|c_i|\!=\!1$, let
\BE{ui_eq}U_i= S^1\times (\prt\Si)_i \times [0,2\eps]
\approx S^1\times\Cyl\EE
be a neighborhood of $S^1\!\times\!(\prt\Si)_i$ in~$M_{\psi}$ and 
$$\wt{U}_i=\bI\times(\prt\Si)_i\times [0,2\eps]$$
be the corresponding neighborhood of $\bI\!\times\!(\prt\Si)_i\subset\bI\!\times\!\Si$.
By the proof of \cite[Proposition~3.1]{Ge}, we can assume that the identification~\eref{ui_eq}  
commutes with the complex structures on the fibers over~$S^1$.\\

\noindent
By \cite[Lemma 2.2]{Teh},
$$(V,\ti{c})|_{S^1\times(\prt\Si)_i} \approx 
\big(\bI\!\times\!(S^1\!\times\!\C^n,\ti{c}_i)\big)\big/\sim, \qquad
(0,z,v)\sim\big(1,z,g_i(z)v\big)\quad\forall\,(z,v)\!\in\!(\prt\Si)_i\!\times\!\C^n,$$
where $\ti{c}_i$ denotes the lift of $c_i$ induced by the standard conjugation on~$\C^n$
and
$$g_i\!:(\prt\Si)_i\lra\GL_n\C \qquad\hbox{s.t.}\quad g_i(c(z))=\ov{g_i(z)}~~\forall~z\!\in\!(\prt\Si)_i\,;$$
in fact, $g_i$ can be taken to be a constant function with values in the diagonal matrices,
with at most one diagonal entry~-1 and the remaining diagonal entries~1.
Since such $g_i$ is homotopically trivial, it can be extended to a~map
$$g_i\!:  (\prt\Si)_i\!\times\![0,2\eps] \lra \GL_n\C\qquad
\text{s.t.}\qquad g_i|_{(\prt\Si)_i\times[\eps/2,2\eps]}=\Id.$$\\

\noindent
For all $i$ with $|c_i|\!=\!1$ and $t\!\in\!S^1$,  pinch $t\!\times\!\Si$ along the curve
$t\!\times\!(\prt\Si)_i\!\times\!\eps$ to obtain a nodal curve $\Si^s$
with normalization consisting of a disjoint union of disks $D_i^2$ with $|c_i|\!=\!1$ and 
a Riemann surface~$\Si'$, whose boundary components are the boundary components 
$(\prt\Si)_i$ of $\Si$ with $|c_i|\!=\!0$, with special points $0\!\in\!D_i^2$ and 
$p_i\!\in\!\Si'$ with $|c_i|\!=\!1$. 
The real bundle pair $(V,\ti{c})$ descends to a real bundle pair over the family of
nodal curves as in \cite[Remark~2.1]{Ge}, inducing bundle pairs
$$(V',\ti{c}')\lra S^1\!\times\!(\Si',\prt\Si') \quad \text{and}\quad 
(V_i,\ti{c}_i)\equiv \bI\times_{g_i}
\big(D^2\!\times\!\C^n,\ti{c}_i\big)\lra  S^1\!\times\!(D^2_i,c_i),$$
with $\mu(V_i,\ti{c}_i)\!=\!0$ and with
isomorphisms $V'|_{t\times p_i}\approx\C^n\approx V_i|_{t\times 0}$ 
for every $t\!\in\!S^1$.\\

\noindent
Taking a family of real Cauchy-Riemann operators $D'$ on $(V',\ti{c}')$ and 
gluing it to a family of real Cauchy-Riemann operators $D_i$ on $(V_i,\ti{c}_i)$, 
we obtain a family of real Cauchy-Riemann operators $D^{\ve}$ on~$(V,\ti{c})$. 
By Remark \ref{con_rmk} and~\eqref{dets},
\BE{orientsplit_e}
\det D_{(V,\ti{c})}\approx \det D^{\ve}\approx 
(\det D')\otimes
\bigotimes_{|c_i|=1}\!\!\big((\det D_i)\!\otimes\!\La_{\R}^{\top}(V'|_{(t,p_i)})\big).\EE
Thus,
$$w_1(\det D)= w_1(\det D')+
\sum_{|c_i|=1}\!\!\big(w_1(\det D_i)+w_1(V'|_{S^1\times p_i})\big).$$
The complex structure on $V'|_{S_1\times p_i}$
induces a canonical orientation on this space; 
in particular, $w_1(V'|_{S_1\times p_i})\!=\!0$. 
The term $w_1(\det D')$ is given  by \cite[Theorem~1.1]{Ge}.
Therefore, the problem reduces to the families of operators $D_i$ on $(V_i,\ti{c}_i)$ over  
$S^1\!\times\!(D_i^2,c_i)$.  
Theorem~\ref{main_thm} now follows from Lemma~\ref{eta_lem}.
\end{proof}

\begin{lem}\label{eta_lem} 
Let $c\!:S^1\!\times\!\prt D^2\!\lra\!S^1\!\times\!\prt D^2$ be 
a fiberwise orientation-preserving involution different from the identity and  
$(V,\ti{c})\!\lra\!(S^1\!\times\!D^2,c)$ be a real bundle pair
with $\mu(V,\ti{c})\!=\!0$ on each fiber.
If $D$ is any family of real Cauchy-Riemann operators on $(V,\ti{c})$ over~$S^1$, then
$$\lr{w_1(\det D),S^1}= \blr{w_2^{\La_{\C}^{\top}\ti{c}}(\La_{\C}^{\top}V),
[S^1\!\times\!\prt D^2]^c}.$$
\end{lem}

\begin{proof}  
Let $n\!=\!\rk_{\C}V$. Denote by 
$$(V_{\pm},\ti{c}_{\pm})\lra (S^1\!\times\!D^2,c)$$
the real bundle pairs with 
$$V_{\pm}=\big(\bI\!\times\!D^2\!\times\!\C\big)\big/\!\!\sim, \qquad
(0,z,v)\sim(1,z,\pm v)~~~\forall~z\!\in\!D^2,\,v\!\in\!\C,$$
with the involutions induced by the standard conjugation on~$\C$.\\

\noindent
By Proposition~\ref{etaFred_prp}, the standard $\bp_0$-operator on the trivial bundle pair
$$(t\!\times\!D^2\!\times\!\C^n,\ti{c}_n)\!\lra\!(t\!\times\!D^2,c|_{t\times S^1}), 
\qquad t\in S^1,$$ 
is surjective and its kernel consists of constant real-valued functions.
Thus, the index bundle of the family of the standard $\bp_0$-operators on 
the trivial rank~$n$ real bundle pair $nV_+$ (direct sum of $n$ copies of~$V_+$)
is isomorphic to $S^1\!\times\!\R^n$ by evaluation at a boundary point and in particular 
is orientable. 
On the other hand, the index bundle of the family of the standard $\bp_0$-operators on
$$V_-\!\oplus\!(n\!-\!1)V_+\lra S^1\!\times\!D^2$$
is the direct sum of the Mobius line bundle over~$S^1$ and $n\!-\!1$ copies of the trivial
real line bundle; in particular, it is non-orientable.
By Remark~\ref{con_rmk}, the determinant bundle of any family of real Cauchy-Riemann operators 
on a trivializable real bundle pair as in the statement of the lemma is thus orientable
and on a real bundle isomorphic to $V_-\!\oplus\!(n\!-\!1)V_+$ is~not.\\


\noindent
By Lemma~\ref{Maslovind_lmm}, every bundle pair $(V,\ti{c})$ as in
the statement of the lemma is isomorphic the bundle pair
$$ \big(\bI\!\times\!D^2\!\times\!\C\big)\big/\sim, \quad
(0,z,v)\sim(1,z,A(z)v)~~~\forall~z\!\in\!D^2,\,v\!\in\!\C,$$
for some smooth map $A\!:D^2\!\lra\!\GL_n\C$ such that 
$A(c(z))\!=\!\ov{A(z)}$ for all~$z\!\in\!S^1$.
By \cite[Lemma~2.2]{Teh}, there are two homotopy classes of such maps;
they are represented by  the constant maps with values in the diagonal matrices,
with at most one diagonal entry~-1 and the remaining diagonal entries~1.
Lemmas~\ref{T2equivL_lmm} and~\ref{Latop_lmm} then imply that 
$w_2^{\La_{\C}^{\top}\ti{c}}(\La_{\C}^{\top}V)$ classifies 
the rank~$n$ real bundle pairs $(V,\ti{c})$ as in the statement of Lemma~\ref{eta_lem}.
Thus, if $w_2^{\La_{\C}^{\top}\ti{c}}(\La_{\C}^{\top}V)\!=\!0$, $(V,\ti{c})$ is trivializable;
by the previous paragraph, $\det D$ is orientable in this case.
On the other hand, if $w_2^{\La_{\C}^{\top}\ti{c}}(\La_{\C}^{\top}V)\!\neq\!0$, $(V,\ti{c})$ is isomorphic to 
the twisted pair of the previous paragraph and thus $\det D$  is not orientable.
Combining the two cases, we obtain the claim.
\end{proof}

\noindent
The next three lemmas are used in the proof of Theorem~\ref{main_thm} and 
in some of its applications.
In particular, in some situations they allow us to replace arbitrary diffeomorphisms of $(\Si,c)$
by those that restrict to the identity near~$\prt\Si$.

\begin{lem}\label{iso_lm} 
Let $(\Si,c)$ be an oriented sh-surface.
For every $h_0\!\in\cD_c$, there exists a path~$h_t$ in~$\cD_c$ starting at~$h_0$
such that~$h_1$ restricts to the identity
on a neighborhood of $\prt\Si$ in~$\Si$.
\end{lem}

\begin{proof} 
Fix a component $(\prt\Si)_i\!\approx\!S^1$ of $\prt\Si$, an
identification of a neighborhood of $(\prt\Si)_i$ in~$\Si$ with 
$S^1\!\times\![0,2\de]$, and $\eps\!\in\!(0,\de/2)$ such that
$h_0(S^1\!\times\![0,2\eps])\subset S^1\!\times\![0,\de]$. 
After composing~$h_0$ with a path of diffeomorphisms on $\Si$ which restrict to 
the identity outside $S^1\!\times\!(0,2\de)$, we can assume that 
$h_0(S^1\!\times\![0,2\eps])=S^1\!\times\![0,2\eps]$.\\

\noindent 
By \cite[Proposition 2.4]{FM} and \cite[(1.1)]{Mas}, the group of diffeomorphisms 
of the cylinder preserving the orientation and each boundary component is path-connected. 
In particular, there exists a path of diffeomorphisms
\begin{gather}\label{ftpath_e}
f_t\!: S^1\!\times\![0,2\eps]\lra  S^1\!\times\![0,2\eps]
\qquad\hbox{s.t.}\quad f_0=h_0,~f_1=\ti{h}_0,\\
\notag
\hbox{where}\qquad
\ti{h}_0(z,s)=(\pi_1(h_0(z,0)),s)~\forall\,(z,s)\!\in\!S^1\!\times\![0,2\eps].
\end{gather}
Replacing $f_t$ with $\ti{h}_0\!\circ\!\ti{f}_t^{-1}\!\circ\!f_t$,
with $\ti{f}_t$ defined analogously to~$\ti{h}_0$,
we obtain a path of diffeomorphisms~$f_t$ as in~\eref{ftpath_e} that restrict to~$h_0$ 
on $S^1\!\times\!0$.
Thus, after composing~$h_0$ with a path of diffeomorphisms on $\Si$ that restrict 
to the identity outside $S^1\!\times\!(0,2\eps)$, we can assume that
$h_0\!=\!\ti{h}_0$ on $S^1\!\times\![0,\eps]$; 
such a path is constructed from a path of diffeomorphisms on $S^1\!\times\![0,2\eps]$
using vector fields as below.\\ 

\noindent
Since $z\!\lra\!h_0(z,0)$ commutes with the involution~$c_i$, 
$h_0$ descends to a diffeomorphism~$h_0'$
on the quotient $(S^1/c_i)\!\times\![0,\eps]$.
By \cite[Proposition 2.4]{FM} and \cite[(1.1)]{Mas},
there is a path of diffeomorphisms
$$f_t'\!: (S^1/c_i)\!\times\![0,\eps]\lra (S^1/c_i)\!\times\![0,\eps]
\qquad\hbox{s.t.}\quad f_0'=\id, ~ f_1'=h_0'^{-1}.$$
This path lifts to a path of diffeomorphisms
$$f_t\!: S^1\!\times\![0,\eps]\lra S^1\!\times\![0,\eps]
\qquad\hbox{s.t.}\quad f_0=\id, ~f_1=h_0^{-1}|_{S^1\times[0,\eps]},~
c_i\!\circ\!f_t|_{S^1\times0}=f_t\!\circ\!c_i.$$
The path $f_t$ generates a time-dependent vector field $X_t$. 
Multiplying~$X_t$ by a bump function on $\Si$ vanishing outside $[0,\eps]$ 
and restricting to $1$ on $S^1\!\times\![0,\epsilon/2]$, 
we obtain a time-dependent vector field $\wt{X}_t$ on~$\Si$. 
This new vector field gives rise to diffeomorphisms $\ti{f}_t$ of $\Si$ 
which restrict to the identity outside $S^1\!\times\![0,\eps]$, 
while $\ti{f}_1$ restricts to $h_0^{-1}$ on $S^1\!\times\![0,\eps/2]$. 
Then $h_0\!\circ\!\ti{f}_t$ is a path in~$\cD_c$ connecting $h_0$ with a diffeomorphism 
which restricts to the identity in a neighborhood of~$(\prt\Si)_i$.
\end{proof}

\begin{lem}\label{def_lm}
Let $(\Si,c)$ be an oriented sh-surface and $\psi\!\in\!\cD_c$.
Every family of real Cauchy-Riemann operators on a real bundle pair $(V,\ti{c})$
over $M_{\psi}$ with $D_t$ compatible with some $\fJ_t\!\in\!\cJ_c$ for each $t\!\in\!\bI$
can be smoothly deformed through such families to a  family of real Cauchy-Riemann
operators on a bundle pair $(V',\ti{c}')$ over $M_{\psi'}$ for some $\psi'\in\cD_c$
such that $\psi'$ restricts to the identity on a neighborhood of $\prt\Si$.
\end{lem}

\begin{proof}
By Lemma \ref{iso_lm}, there exists a path $h_s$ in $\cD_c$ such that $h_0\!=\!\psi$
and $h_1$ restricts to the identity on a neighborhood of $\prt\Si$ in $\Si$. 
Set $f_s\!=\!\psi^{-1}\!\circ\!h_s$. 
Let $(\fJ_t,V_t,\ti{c}_t,D_t)$, with $t\!\in\!\bI$, be any family of tuples such that 
$\fJ_t\!\in\!\cJ_c$, 
$D_t$ is a real Cauchy-Riemann operator on $(V_t,\ti{c}_t)$ over $(\Si,c)$
compatible with~$\fJ_t$, and
$$(\fJ_1,V_1,\ti{c}_1,D_1)=\psi^*(\fJ_0,V_0,\ti{c}_0,D_0).$$
For each $s\!\in\!\bI$, let
$$(\fJ_{s;t},V_{s;t},\ti{c}_{s;t},D_{s;t})
=f_{st}^*(\fJ_t,V_t,\ti{c}_t,D_t). $$
Since $(\fJ_{s;1},V_{s;1},\ti{c}_{s;1},D_{s;1})
=h_s^*(\fJ_{s;0},V_{s;0},\ti{c}_{s;0},D_{s;0})$, 
this defines  families of real Cauchy-Riemann operators on the real bundle pairs 
$(V_s,\ti{c}_s)$ over~$M_{h_s}$. 
Since $h_0\!=\!\psi$, we have thus constructed the desired deformation of the original
family.
\end{proof}

\begin{lem}\label{lift_cor} 
Let $(X,\phi)$ be a smooth manifold with an involution,
$(\Si,c)$ be an oriented sh-surface,
and $\b$ be a tuple as in~\eref{btuple_eq}.
Every loop $\gm$ in $\cH_g(X,\b)^{\phi,c}$ lifts to a path $\wt\gm$ 
in $\fB_g(X,\b)^{\phi,c}\!\times\!\cJ_c$ such that 
$\wt\gm_1=\psi\cdot\wt\gm_0$ for some $\psi\!\in\!\cD_c$ with
$\psi|_{\prt\Si}=\id$.
\end{lem}

\begin{proof}
Under the assumption of Remark \ref{mfld_rem}, the projection 
$$\fB_g(X,\b)^{\phi,c}\times\cJ_c\lra\cH_g(X,\b)^{\phi,c}$$
admits local slices. 
Thus, there exists a path $\wt\gm_t=(u_t,\fJ_t)$ in $\fB_g(X,\b)^{\phi,c}\!\times\!\cJ_c$
lifting~$\gm$.
Let $\psi\!\in\!\cD_c$ be such that $\wt\gm_1=\psi\!\cdot\!\wt\gm_0$.
By Lemma~\ref{iso_lm}, there exists a path $h_t$ in $\cD_c$ such that $h_0\!=\!\psi$
and $h_1$ restricts to the identity on the boundary.
The lift $\wt\gm'_t=h_t\!\cdot\!\psi^{-1}\!\cdot\!\wt\gm_t$ of $\gm$ 
then satisfies $\wt\gm'_1\!=\!h_1\!\cdot\!\wt\gm'_0$.
\end{proof}

\section{Local systems of orientations}
\label{sec:ls}

\noindent
This section  extends \cite[Section~4]{Ge} to arbitrary boundary involutions and
reformulates Corollary~\ref{orient_cor} in terms of local 
systems of orientations,
making it easier to compare systems of orientations induced from different bundles
as in Remark~\ref{fam_rem}.
For the sake of completeness, we begin by recalling the basics of local systems 
following~\cite{Ste}.
We continue by constructing a local system $\cZ_{(w_1,w_2)}^{\ti\phi,c}$ on the product of
$|c|_1$ copies of the equivariant free loop space of~$X$
and $|c|_0$ copies of the $\phi$-fixed locus~$X^{\phi}$ and its free loop 
space $\cL(X^{\phi})$.
We then show that its pull-back  is  isomorphic to the local system twisted 
by the first Stiefel-Whitney class of $\det D_{(V,\ti\phi)}$.

\begin{definition} A \textsf{system of local groups} $\cG$ 
on a path-connected topological space~$L$ consists~of
\begin{enumerate}[label=$\bullet$,leftmargin=*]
\item a group $G_x$ for every $x\!\in\!L$ and
\item a group isomorphism $\al_{xy}:G_x\!\lra\!G_y$  for every homotopy class $\al_{xy}$ 
of paths from $x$ to $y$
\end{enumerate}
such that the composition $\be_{yz}\!\circ\!\al_{xy}$ is the isomorphism corresponding
to the path $\al_{xy}\be_{yz}$. 
\end{definition}

\begin{lem}[{\cite[Theorem 1]{Ste}}]\label{uniq} 
Let $L$ be a path-connected topological space, $p_0\!\in\!L$, and $G$ be a group.
For every group homomorphism $\psi\!:\!\pi_1(L,p_0)\!\lra\!\aut(G)$,
there is a unique system $\cG_{\psi}\!=\!\{G_x\}$ of local groups on $L$ such that $G_{p_0}\!=\!G$ 
and the operations of $\pi_1(L,p_0)$ on $G_{p_0}$ are those determined by~$\psi$.
\end{lem}

\noindent
Two local systems $\cG$ and $\cG'$ on $L$ are \textsf{isomorphic} if for every point $x\!\in\!L$
there is an isomorphism $h_x\!: G_x\!\approx\!G_x'$ such that $\al_{xy}h_x\!=\!h_y\al_{xy}$ 
for every path~$\al_{xy}$  between~$x$ and~$y$. 
Equivalently, two local systems are isomorphic if the groups $G$ and $G'$ are isomorphic 
and the induced actions of $\pi_1(L,x_0)$ are the same.
There are $\aut(G)$  of such isomorphisms, and one is fixed by a choice of an
isomorphism $G_{x_0}\!\approx\!G'_{x_0}$.\\

\noindent
A continuous map $f\!:(L_1,p_1)\!\lra\!(L_2,p_2)$ naturally \textsf{pulls back}
a local system $\cG$ on~$L_2$ to a local system $f^*\cG$ on~$L_1$.
If $L_1$ and $L_2$ are path-connected and $\cG$ is induced by a group homomorphism
$\psi\!: \pi_1(L_2,p_2)\!\lra\!\aut(G)$, then $f^*\cG$ is induced by the group homomorphism
$$\psi\!\circ\!f_{\#}\!: \pi_1(L_1,p_1)\lra \aut(G),$$ 
where  $f_{\#}: \pi_1(L_1,p_1)\!\lra\!\pi_1(L_2,p_2)$.
The \textsf{local system of orientations} for a vector bundle $V\!\lra\!L$,
denoted by~$\cZ_{w_1(V)}$, 
is the system induced  by the homomorphism
$$\psi:\pi_1(L,p_0)\lra\aut(\Z)=\Z_2,\qquad
\al\lra\lr{w_1(V),\al}.$$\\

\noindent
Let $(X,\phi)$ be a topological space with an involution
and  $(V,\ti\phi)\!\lra\!(X,\phi)$ be a real bundle pair.
Fix base points $p_i$, $\gm_j$, and~$\Ga_k$ for the connected components
$X^{\phi}_i$, $\cL(X^{\phi})_j$, and $\cL(\bB_{\phi}X)_k$ of
$X^{\phi}$, $\cL(X^{\phi})$, and $\cL(\bB_{\phi}X)$, respectively.
Let $\cZ_{w_1,w_2}^{\ti\phi}$ be the local system on $X^{\phi}\!\times\!\cL(X^{\phi})$
corresponding to the homomorphism
\begin{gather*}
\psi:\pi_1(X^{\phi}_i\!\times\!\cL(X^{\phi})_j,p_i\!\times\!\gm_j)
=\pi_1(X^{\phi}_i,p_i)\!\times\!\pi_1(\cL(X^{\phi})_j,\gm_j) \lra\aut(\Z)=\Z_2,\\
(\al,\be)\lra \big(\lr{w_1(V^{\ti\phi}),\gm_j}+1\big)\lr{w_1(V^{\ti\phi}),[\al]}
+\lr{w_2(V^{\ti\phi}),[\be]},
\end{gather*}
and $\cZ_{w_2^{\ti\phi}}$ be the local system on $\cL(\bB_{\phi}X)$ 
corresponding to the homomorphism
$$\psi:\pi_1(\cL(\bB_{\phi}X)_k,\Ga_k)\lra\aut(\Z),\qquad
\be\lra \lr{w_2^{\ti\phi}(V),[\be]}.$$
If $c$ is a boundary involution on an oriented surface~$\Si$, 
we define $\cZ^{\ti\phi,c}_{(w_1,w_2)}$ on 
$$X_{\phi,c}\equiv 
\big(X^{\phi}\!\times\!\cL(X^{\phi})\big)^{|c|_0}\times \cL(\bB_{\phi}X)^{|c|_1}$$
to be the pull-back of the local systems $\cZ_{w_1,w_2}^{\ti\phi}$ and
$\cZ_{w_2^{\ti\phi}}$ by the projection maps.
Thus, the restriction of this system to a component of $X_{\phi,c}$
with a basepoint 
\BE{basepoint_e}(\vec{p},\vec\gm,\vec\Ga)\equiv\big(p_1,\gm_1,\ldots,p_{|c|_0},\gm_{|c|_0},
\Ga_{|c|_0+1},\ldots,\Ga_{|c|_0+|c|_1}\big)\EE
is given by the homomorphism
\begin{gather}
\psi: \pi_1\big(X_{\phi,c},(\vec{p},\vec\gm,\vec\Ga)\big) \lra \aut(\Z)=\Z_2, \notag\\ 
\label{zf_eq}
\begin{split}
&\big(\al_1,\be_1,\ldots,\al_{|c|_0},\be_{|c|_0},\be_{|c|_0+1},\ldots,\be_{|c|_0+|c|_1}\big)\\
&\qquad\lra 
\sum_{i=1}^{|c|_0}\Big(\big(\lr{w_1(V^{\ti\phi}),\gm_i}\!+\!1\big)\lr{w_1(V^{\ti\phi}),[\al_i]}
+\lr{w_2(V^{\ti\phi}),[\be_i]}\Big)
+\sum_{i=|c|_0+1}^{|c|_0+|c|_1}\!\!\!\lr{w_2^{\ti\phi}(V),[\be_i]}.
\end{split}
\end{gather}\\

\noindent
If $(X,\phi)$ and $(\Si,c)$ are as above, $g$ is the genus of~$\Si$,
$\b$ is a tuple of homology classes as in~\eref{btuple_eq},
and $\k\!=\!(k_1,\ldots,k_{|c|_0+|c|_1})$ is a tuple of nonnegative integers, let
\begin{equation*}\begin{split}
\fB_{g,\k}(X,\b)^{\phi,c}&=\fB_g(X,\b)^{\phi,c}\times
\prod_{i=1}^{|c|_0+|c|_1}\!\!\!\!\!\big((\prt\Si)_i^{k_i}-\De_{i,k_i}\big),\\
\cH_{g,\k}(X,\b)^{\phi,c}&=
\big(\fB_{g,\k}(X,\b)^{\phi,c}\!\times\!\cJ_c\big)\big/\cD_c,
\end{split}\end{equation*}
where 
$$\De_{i,k_i}=\big\{(x_{i,1},\ldots,x_{i,k_i})\!\in\!(\prt\Si)_i^{k_i}\!:~
x_{i,j'}\!\in\!\{x_{i,j},c(x_{i,j})\}~\tn{for some}~j,j'\!=\!1\ldots,k_i,~j\!\neq\!j'\big\}$$
is the big $c$-symmetrized diagonal.
In the case~$X$ is a point and $\b$ is the zero tuple, we denote 
$\cH_{g,{\bf 0}}(X,\b)^{\phi,\id_{\prt\Si}}$ by $\cM_{\Si}$;
this is the usual Deligne-Mumford moduli space of stable bordered Riemann surfaces 
with ordered boundary components if $\Si$ is not a disk or a cylinder
(for stability reasons).
In all cases, let 
$$\ff:\cH_{g,\k}(X,\b)^{\phi,c}\lra\cH_g(X,\b)^{\phi,c}$$ 
be the map forgetting the marked points.

\begin{prop}\label{evmap_prop}
Let $(X,\phi)$, $(\Si,c)$, $g$, $\b$, $\k$, and $X_{\phi,c}$  be as above.
If $\k\!\in\!(\Z^+)^{|c|_0+|c|_1}$, there is a map 
$$\ev\!: \cH_{g,\k}(X,\b)^{\phi,c} \lra X_{\phi,c}$$
such that for every real bundle pair $(V,\ti\phi)\!\lra\!(X,\phi)$ 
the local system $\cZ_{w_1(\det D_{V,\ti\phi})}$ over $\cH_{g,\k}(X,\b)^{\phi,c}$
is isomorphic to $\ev^*\cZ^{\ti\phi,c}_{(w_1,w_2)}$.
If $\k\!=\!\1$, this local system pushes down to a local system 
$$\wt\ev^*\cZ^{\ti\phi,c}_{(w_1,w_2)}\equiv\ff_*\circ\ev^*\cZ^{\ti\phi,c}_{(w_1,w_2)}$$ 
over $\cH_g(X,\b)^{\phi,c}$ isomorphic to $\cZ_{w_1(\det D_{V,\ti\phi})}$.
\end{prop}

\begin{proof} 
The proof is similar to the proofs of \cite[Proposition~4.6]{Ge}
and \cite[Lemma~4.7]{Ge}; so we omit some of the details and refer the reader
to~\cite{Ge}.\\

\noindent
The component $\ev_i^{X^{\phi}}$ of~$\ev$ to the $i$-th $X^{\phi}$ factor is given by 
the evaluation at the first marked point~$x_{i,1}$ on the $i$-th boundary component. 
We now describe the remaining components of~$\ev$.
Let $\cD_0$ and~$\cD_1$ denote the subgroup of~$\cD_c$ restricting 
to the identity on~$\prt\Si$ and 
the subgroup of~$\cD_c$ fixing the first marked point~$x_{i,1}$
on each boundary component~$(\prt\Si)_i$, respectively.
The fibration 
$$\big(\fB_{g,\k}(X,\b)^{\phi,c}\!\times\!\cJ_c\big)\big/\cD_0
\lra
\big(\fB_{g,\k}(X,\b)^{\phi,c}\!\times\!\cJ_c\big)\big/\cD_1$$
has contractible fibers and thus admits a section~$s$. 
Since the elements of $\cD_0$ fix $\prt\Si$ pointwise, there are well-defined maps
$$e_i\!: \big(\fB_{g,\k}(X,\b)^{\phi,c}\!\times\!\cJ_c\big)\big/\cD_0
\lra \cL(X^{\phi}), \quad
[u,\x_1,\ldots,\x_{|c|_0+|c|_1}]\lra u|_{(\prt\Si)_i}\,,$$
for $i\!=\!1,\ldots,|c|_0$.
If $i\!=\!|c|_0\!+\!1,\ldots,|c|_0\!+\!|c|_1$, 
$\phi\!\circ\!u|_{(\prt\Si)_i}=u\!\circ\!c|_{(\prt\Si)_i}$. 
Thus, $u|_{(\prt\Si)_i}$ determines an element of~$\cL(\bB_{\phi}X)$,
$$\bB_{\phi,c_i}(u|_{(\prt\Si)_i})\!: S^1\approx (\prt\Si)_i/\Z_2 \lra \bB_{c_i}(\prt\Si)_i
\stackrel{\id\times_{\Z_2}u}\lra\bB_{\phi}X\,,$$
where the middle map is a section of the fiber bundle~\eref{Z2fib_e2}; 
see also the beginning of Section~\ref{equivappl_subs}.
So, there are well-defined maps
$$e_i\!: \big(\fB_{g,\k}(X,\b)^{\phi,c}\!\times\!\cJ_c\big)\big/\cD_0
\lra \cL(\bB_{\phi}X)$$
for $i\!=\!|c|_0\!+\!1,\ldots,|c|_0\!+\!|c|_1$.
Let $\1\!=\!(1,\ldots,1)\in\!\Z^{|c|_0+|c|_1}$.
The component~$\ev_i^{\cL}$ of~$\ev$ to the $i$-th factor 
$\cL(X^{\phi})$ or $\cL(\bB_{\phi}X)$ is the composition 
$$\cH_{g,\k}(X,\b)^{\phi,c} \lhra
\big(\fB_{g,\k-\1}(X,\b)^{\phi,c}\!\times\!\cJ_c\big)\big/\cD_1
\stackrel{e_i\circ s}{\lra}
\begin{cases}
\cL(X^{\phi}),&\hbox{if}~i=1,\ldots,|c|_0;\\
\cL(\bB_{\phi}X),&\hbox{if}~i=|c|_0\!+\!1,\ldots,|c|_0\!+\!|c|_1.
\end{cases}$$
It is well-defined up to homotopy.\\

\noindent
We now show that the local systems $\cZ_{w_1(\det D_{V,\ti\phi})}$ and
$\ev^*\cZ^{\ti\phi,c}_{(w_1,w_2)}$ over $\cH_{g,\k}(X,\b)^{\phi,c}$ are isomorphic.
Let $u_0\!\in\!\cH_{g,\k}(X,\b)^{\phi,c}$ be a preimage under $\ev$ of a basepoint
$(\vec{p},\vec\gm,\vec\Ga)$ in~$X_{\phi,c}$ as in~\eref{basepoint_e}.
In particular,
$$b_i=[\gm_i]\in H_1(X^{\phi};\Z)~~\forall\,i=1,\ldots,|c|_0,\quad
b_i=[\Ga_i]\in H_1^{\phi}(X;\Z)~~\forall\,i=|c|_0\!+\!1,\ldots,|c|_0\!+\!|c|_1.$$
It is enough to show that the action of every element~$\gm$ of 
$\pi_1(\cH_{g,\k}(X,\b)^{\phi,c},u_0)$ on the two local systems is the same.
The action on $\cZ_{w_1(\det D_{V,\ti\phi})}$ is given~by
$\lr{w_1(\det D_{(V,\ti\phi)}),\gm}$.
By Corollary~\ref{orient_cor} and Lemma~\ref{Latop_lmm},
\BE{act_e}\begin{split}
\lr{w_1(\det D_{(V,\ti\phi)}),\gm}&= 
\sum_{i=1}^{|c|_0}\Big(\big(\blr{w_1(V^{\ti\phi}),b_i}+1\big)
\lr{w_1(V^{\ti\phi}),[\al_i]}+\blr{w_2(V^{\ti\phi}),[\be_i]}\Big)\\
&\quad+\sum_{i=|c|_0+1}^{|c|_0+|c|_1} \!\!
\blr{w_2^{\ti\phi}(V),[\be_i]^{\id\times c_i}},
\end{split}\EE
where $\al_i\!:S^1\!\lra\!X$ and $\be_i\!:S^1\!\times\!(\prt\Si)_i\!\lra\!X$ 
are the paths traced by~$x_{i,1}$ and by the entire boundary component~$(\prt\Si)_i$ 
along $\wt\gm\!=\!s\!\circ\!\gm$. 
In particular,
\begin{gather*}
\al_i=\ev_i^{X^{\phi}}\!\circ\!\gm,\quad \be_i=\ev_i^{\cL}\!\circ\!\gm 
\qquad \forall~i=1,\ldots,|c|_0,\\
\bB_{\phi,\id_{S^1}\times c_i}\be_i=\ev_i^{\cL}\!\circ\!\gm \qquad\forall~i=|c|_0\!+\!1,\ldots,|c|_0\!+\!|c|_1.
\end{gather*}
The action of $\gm$ on $\ev^*\cZ^{\ti\phi,c}_{(w_1,w_2)}$ is given~by the action~of
$$\ev\!\circ\!\gm
=\big(\al_1,\be_1,\ldots,\al_{|c|_0},\be_{|c|_0},
\bB_{\phi,\id_{S^1}\times c_i}\be_{|c|_0+1},\ldots,\bB_{\phi,\id_{S^1}\times c_i}\be_{|c|_0+|c|_1}\big)$$
on $\cZ^{\ti\phi,c}_{(w_1,w_2)}$.
By~\eref{zf_eq}, this action is also given by the right-hand side of~\eref{act_e}.\\

\noindent
For the last claim of the proposition, it is sufficient to show that 
the system $\ev^*\cZ^{\ti\phi,c}_{(w_1,w_2)}$ over $\cH_{g,\1}(X,\b)^{\phi,c}$
pushes down under the map forgetting the boundary points on 
a particular component $(\prt\Si)_i$ of~$\prt\Si$, 
i.e.~that the system is trivial along each fiber of this map.
For the components with $|c_i|\!=\!0$, this is done in the proof of 
\cite[Lemma~4.7]{Ge}; so, we assume that $|c_i|\!=\!1$.
Let $\gm$ be a loop in the fiber (which is homotopic to~$S^1$). 
The~map
$$\bB_{\phi,\id_{S^1}\times c_i}(\ev_i^{\cL}\!\circ\!\gm)\!: 
S^1\!\times\!((\prt\Si)_i/\Z_2)\lra\bB_{\phi}X$$
factors through 
$(\prt\Si)_i/\Z_2\!\approx\!S^1$.
Thus, the map $\ev_i^{\cL}\!\circ\!\gm$ represents the zero class in 
$H_2^{\phi}(X)$, and so the reasoning in the proof of 
\cite[Lemma~4.7]{Ge} still applies.
\end{proof}

\begin{prop}\label{orient_prop}  
Let $(X,\phi)$, $(\Si,c)$, $g$, $\b$, $\k$, and $X_{\phi,c}$  be as before
and $(V,\ti\phi)\!\lra\!(X,\phi)$ be a real bundle pair.
An isomorphism between the local systems of Proposition~\ref{evmap_prop}
is determined by trivializations~of 
\begin{enumerate}[label=(\arabic*),leftmargin=*]
\item $\La^{\top}_{\R}(V^{\ti\phi})$ over a basepoint of each component of $X^{\phi}$,
\item $V^{\ti\phi}\oplus 3\La^{\top}_{\R}V^{\ti\phi}$ over
representatives for free homotopy classes of loops in~$X^{\phi}$
(one representative for each homotopy class), and
\item $(V,\ti\phi)$ over representatives for free homotopy classes of maps $S^1\!\lra\!X$
intertwining $\phi$ and the antipodal involution~$\fa$ on~$S^1$.
\end{enumerate}
The effect on this isomorphism over $\cH_{g,\k}(X,\b)^{\phi,c}$ under the changes 
$$o_{\R}^{\ti\phi}\!:\pi_0(X^{\phi})\lra\{0,1\}, \quad
s_{\R}^{\ti\phi}\!: \pi_1(X)\lra\pi_1(\SO(\rk_{\C}V)), \quad
o_{\C}^{\ti\phi}\!:\pi_1^{\phi}(X)\lra\{0,1\}$$  
in these trivializations is the multiplication by $(-1)^\eps$, where
\BE{orientprop_e}\begin{split}
\eps&=\sum_{|c_i|=0}\!\!
\big((\lr{w_1(V^{\ti\phi}),b_i}\!+\!1)\,o^{\ti\phi}_{\R}(\lr{b_i}_0)
+s^{\ti\phi}_{\R}(b_i)\big)
+\sum_{|c_i|=1}\!\!o^{\ti\phi}_{\C}(b_i)
\end{split}\EE
and $\lr{b_i}_0\in\pi_0(X^{\phi})$ is the component determined by 
$b_i\!\in\!H_1(X^{\phi};\Z)$.\footnote{$H_1(X^{\phi};\Z)$ is the direct sum
over $\pi_0(X^{\phi})$; $\cH_{g,\k}(X,\b)^{\phi,c}=\eset$ unless 
each $b_i$ lies in a summand of this decomposition.}
\end{prop}

\begin{proof}
An isomorphism between the two local systems is determined
by a trivialization of $\det D_{(V,\ti\phi)}$ over a basepoint~$u_0$
for each component of $\cH_{g,\k}(X,\b)^{\phi,c}$ that lies in the preimage
of a basepoint $(\vec{p},\vec\gm,\vec\Ga)$ of~$X_{\phi,c}$.
This fixes the group $\Z$ at $u_0$ and thus an isomorphism between the two systems.
By the proof of \cite[Theorem~1.7]{Ge}, this isomorphism is independent of 
the choice of~$u_0$.\\

\noindent
Fix $u_0\!\in\!\cH_{g,\k}(X,\b)^{\phi,c}$ as above.
By \eref{orientsplit_e},
$$\det D_{(V,\ti\phi)}|_{u_0}\approx 
(\det D_{u_0}')\otimes
\bigotimes_{|c_i|=1}\!\!\big((\det D_{u_0;i})\!\otimes\!\La_{\R}^{\top}(V_{p_i}')\big),$$
where $D_{u_0}'$ is a real Cauchy-Riemann operator on a bundle pair
$(V',\ti{c}')\!\lra\!(\Si',c)$ with
$$(\prt\Si',c)=\bigsqcup_{|c_i|=0}\!\!((\prt\Si)_i,c_i), \quad
\big(V',\ti{c}'\big)\big|_{(\prt\Si)_i}=u_0^*(V,\ti\phi)|_{(\prt\Si)_i},$$
$D_{u_0;i}$ is a real Cauchy-Riemann operator on a bundle pair
$(V_i,\ti{c}_i)\!\lra\!(D^2,c_i)$ with 
$$\big(V_i|_{S^1},\ti{c}_i\big)=u_0^*(V,\ti\phi)|_{(\prt\Si)_i}\,,$$
and $p_i\!\in\!\Si'$.
The vector spaces $V_{p_i}'$ are canonically oriented by their complex structures.
By \cite[Proposition~4.8]{Ge},
an orientation on $D_{u_0}'$ is determined by the trivializations~(1)
and~(2) in the statement of the proposition and 
the effect of the changes in these choices on the orientation of~$D_{u_0}'$
is described by the first sum in~\eref{orientprop_e}.
By the proof of Lemma~\ref{eta_lem}, an orientation of $D_{u_0;i}$ is determined 
by the homotopy class of a trivialization of 
$u_0^*(V,\ti\phi)$ over~$(\prt\Si)_i$ and changing the homotopy class changes
the induced orientation.
This implies the claim.
\end{proof}


\section{Moduli spaces of maps with crosscaps}
\label{sec:ls2}

\noindent
We begin this section by constructing moduli spaces of $J$-holomorphic maps from
oriented sh-surfaces.
We then discuss implications of Proposition~\ref{orient_prop} for the local systems 
of orientations on these spaces, giving an explicit formula for their first Stiefel-Whitney
classes;
see Corollary~\ref{ls_cor}.\\

\noindent
As in~\cite{Ge}, let $\cJ_{\Si}$ and $\cD_{\Si}$ denote the space of 
all almost complex structures on~$\Si$ and the group of diffeomorphisms
of~$\Si$ preserving the orientation and the boundary components, respectively.
The~map 
$$\cJ_c/\cD_c \lra \cJ_{\Si}/ \cD_{\Si}$$ 
is surjective, but has infinite-dimensional fibers unless $c\!=\!\id_{\prt\Si}$
(in which case $\cJ_c\!=\!\cJ_{\Si}$ and $\cD_c\!=\!\cD_{\Si}$).
We define subspaces $\cJ_c^*\!\subset\!\cJ_c$ and  $\cD_c^*\!\subset\!\cD_c$ so that 
the~map 
\BE{DMisom_e} \cJ_c^*/\cD_c^* \lra \cJ_{\Si}/ \cD_{\Si} \EE
induced by the inclusions $\cJ_c^*\!\lra\!\cJ_{\Si}$ and $\cD_c^*\!\lra\!\cD_{\Si}$
is an isomorphism, 
whenever $(\Si,c)$ is not a disk with an involution different from the identity.\\

\noindent
If $\Si$ is a disk, we identify $\Si$ with the unit disk in~$\C$ so that $c$ corresponds
to either the identity map or the antipodal involution~$\fa$  on $S^1\!\subset\!\C$.
Let $\cJ_c^*\!=\!\{\fJ_0\}$, 
where $\fJ_0$ is the standard complex structure on the disk.
If $c\!=\!\id_{S^1}$, we take
$$\cD_c^*=\PGL_2^0\R\equiv\Big\{z\!\lra\!v\frac{z\!+\!\bar{a}}{1\!+\!az}\!:\,
v\!\in\!S^1,\,a\!\in\!\C,\,|a|\!<\!1\Big\};$$
this is the group of holomorphic automorphisms of the disk.
If $c\!=\!\fa$, we take $\cD_c^*$ to be the subgroup of $\PGL_2^0\R$
consisting of the standard rotations of~$S^1$.
The map~\eref{DMisom_e} is then surjective, since
for any other complex structure~$\fJ$ on the disk compatible with the orientation, 
there exists an orientation-preserving diffeomorphism $h$ of the disk such 
that $\fJ\!=\!h^*\fJ_0$; see \cite[Corollary~1.9.5]{AG}.
In particular, the map~\eref{DMisom_e} is an isomorphism if $c\!=\!\id_{S^1}$;
if $c\!=\!\fa$, this map takes a point with the trivial $S^1$-action 
to a point with the trivial $\PGL_2^0\R$-action.\\

\noindent
Suppose next that $\Si$ is a cylinder with ordered boundary components $(\prt\Si)_1$
and~$(\prt\Si)_2$.
Let $\oI\!=\!(0,1)$. For each $r\!\in\!\oI$, we define 
$$A_r=\big\{z\!\in\!\C\!:\,(|z|\!-\!r)(r|z|\!-\!1)\le0\big\},\quad
(\prt A_r)_1=\big\{z\!\in\!\C\!:\,|z|\!=\!r\big\}, \quad
(\prt A_r)_2=\big\{z\!\in\!\C\!:\,r|z|\!=\!1\big\}.$$
Choose a smooth map
$$\Psi\!: \oI\!\times\!\Si\lra\C^*, \qquad
\Psi(r,z)\lra \Psi_r(z),$$
such that each map
$$\Psi_r\!: \big(\Si,(\prt\Si)_1,(\prt\Si)_2\big)
\lra  \big(A_r,(\prt A_r)_1,(\prt A_r)_2\big), \qquad r\in\oI,$$
is a diffeomorphism so that $\fa\!\circ\!\Psi_r\!=\!\Psi_r\!\circ\!c_i$
on $(\prt\Si)_i$
if $|c_i|\!=\!1$ and $i\!=\!1,2$ and the diffeomorphisms
$$\Psi_r\!\circ\!\Psi_{r'}^{-1}\!: A_{r'}\lra A_r, \qquad r,r'\in\oI,$$
commute with the standard action of $S^1\!\subset\!\C^*$ on~$\C$.
The last condition implies~that the $S^1$-action on~$\Si$ given~by
\BE{S1act_e}S^1\times\Si\lra\Si, \qquad u\cdot z=\Psi_r^{-1}\big(u\Psi_r(z)\big)
\quad\forall\,z\!\in\!\Si,\,u\!\in\!S^1\!\subset\!\C,\EE
is independent of $r\!\in\!\oI$.
In this case, we take 
$$\cJ_c^*=\big\{\Psi_r^*\fJ_0\!:\,r\!\in\!\oI\big\}$$
and $\cD_c^*\!\subset\!\cD_c$ to be the subgroup corresponding to the action~\eref{S1act_e}.
The latter is the group of automorphisms of each complex structure in $\cJ_c^*$
that preserve each boundary component of~$\Si$.
By the classification of complex structures on the cylinder \cite[Section~9]{AG}, 
for every $\fJ\!\in\!\cJ_{\Si}$, there exist a unique 
$r\!\in\!\oI$ and a diffeomorphism~$h$ of~$\Si$ preserving the orientation
and the boundary components such that $\fJ\!=\!h^*\Psi_r^*\fJ_0$.
It follows that  the map~\eref{DMisom_e} is an isomorphism.\\

\noindent
If $\Si$ is not a disk or a cylinder, i.e.~the genus of its double is at least~2,
we identify each boundary component $(\prt\Si)_i$ of $\prt\Si$ with $S^1$ in such 
a way that $c_i\!\equiv\!c|_{(\prt\Si)_i}$ corresponds to either the identity or
the antipodal map on~$S^1$
and denote by $\cD_i$ the subgroup of diffeomorphisms of $(\prt\Si)_i$ corresponding
to the rotations of~$S^1$ under this identification.
For each $\fJ\!\in\!\cJ_{\Si}$, there exists a unique metric $\hat{g}_{\fJ}$
on the double $(\hat\Si',\hat\fJ')$ of $(\Si,\fJ)$ with respect to the involution
$\id_{\prt\Si}$ so that $\hat{g}_{\fJ}$ has constant scalar curvature~-1
and is compatible with~$\hat\fJ'$.
Each boundary component $(\prt\Si)_i$ is a geodesic with respect to~$\hat{g}_{\fJ}$,
and each isometry of $(\prt\Si)_i$ with respect to $\hat{g}_{\fJ}$ is real-analytic
with respect to~$\fJ$.
We denote by $\cJ_c^*\!\subset\!\cJ_{\Si}$ the subspace of complex structures~$\fJ$
so that each $\cD_i$ is the group of isometries of~$(\prt\Si)_i$ with respect to~$\hat{g}_{\fJ}$ 
and by $\cD_c^*$ the subgroup of diffeomorphisms of $\Si$ that preserve the orientation
and the boundary components and restrict to elements of~$\cD_i$ on each boundary component
$(\prt\Si)_i$ of~$\Si$. 
Since $c_i\!\in\!\cD_i$ for each~$i$, 
$\cJ_c^*\!\subset\!\cJ_c$ and $\cD_c^*\!\subset\!\cD_c$.
We verify in the proof of Lemma~\ref{DM_lmm} that the map~\eref{DMisom_e} is an isomorphism
in this case as well.\\

\noindent 
If $(X,\phi)$ and $(\Si,c)$ are as above, $g$ is the genus of~$\Si$, and
$\b$ is a tuple of homology classes as in~\eref{btuple_eq},
and $\k\!=\!(k_1,\ldots,k_{|c|_0+|c|_1})$ is a tuple of nonnegative integers, 
we define
$$\cH_{g,\k}^*(X,\b)^{\phi,c}=
\big(\fB_{g,\k}(X,\b)^{\phi,c}\times\cJ_c^*\big)/\cD_c^*\,.$$
If in addition $J$ is an almost complex structure on~$X$ such that $\phi^*J\!=\!-J$, 
let
$$\mf_{g,\k}(X,J,\b)^{\phi,c}=
\big\{[u,\x_1,\ldots,\x_{|c|_0+|c|_1},\fJ]\!\in\!\cH_{g,\k}^*(X,\b)^{\phi,c}\!:~
\dbar_{J,\fJ}u\!=\!0\big\},$$
where $\dbar_{J,\fJ}$ is as in~\eref{dbar_e},
be the moduli space of marked $J$-holomorphic maps from Riemann sh-surfaces 
that intertwine the involutions on the boundary.
In the case~$X$ is a point and $\b$ is the zero tuple, 
we denote $\cH_{g,{\mathbf 0}}^*(X,\b)^{\phi,c}$ by $\cM_{\Si}^c$.

\begin{lem}\label{DM_lmm}
If $(\Si,c)$ is an sh-surface and $\Si$ is not a disk or a cylinder,
the~map
\BE{DMlmm_e}\cM_{\Si}^c\lra\cM_{\Si} ,\EE
induced by the inclusions $\cJ_c^*\!\lra\!\cJ_{\Si}$ and $\cD_c^*\!\lra\!\cD_{\Si}$,
is an isomorphism.
\end{lem}

\begin{proof}
If $\fJ\!\in\!\cJ_c^*$ and $h\!\in\!\cD_{\Si}$ are such that $h^*\fJ\!\in\!\cJ_c^*$, 
then each parametrized boundary component of~$\Si$ is a geodesic with respect to 
the metrics $\hat{g}_{\fJ}$ and $h^*\hat{g}_{\fJ}$ on  the doubles 
$(\hat\Si',\hat\fJ')$ and $(\hat\Si',\widehat{h^*\fJ}')$.
Thus, the restriction of $h$ to each boundary component of~$\Si$ is an isometry
with respect to the metric~$\hat{g}_{\fJ}$,
and so $h|_{(\prt\Si)_i}\!\in\!\cD_i$ and $h\!\in\!\cD_c^*$.
Thus, the map~\eref{DMisom_e} is injective and continuous (since it is induced
by inclusions before taking the quotients).\\

\noindent
Suppose $\fJ\!\in\!\cJ_{\Si}$. For each boundary component $(\prt\Si)_i$ of $\Si$,
let $f\!:(\prt\Si)_i\!\lra\!(\prt\Si)_i$ be an orientation-preserving geodesic parametrization 
of the target with respect to the metric $\hat{g}_{\fJ}$ and 
the chosen parametrization of the domain.
Similarly to the proof of Lemma~\ref{iso_lm}, $f$ extends to a diffeomorphism~$h$
of $\Si$ that preserves the orientation and the boundary components.
By the assumption on~$f$, each parametrized boundary component $(\prt\Si)_i$
is a geodesic with respect to $h^*\hat{g}_{\fJ}\!=\!\hat{g}_{h^*\fJ}$
and so $h^*\fJ\!\in\!\cJ_c^*$.
Thus, the map~\eref{DMisom_e} is surjective and open
(since $h$ can be chosen to depend continuously on~$\fJ$). 
\end{proof}

\begin{cor}[of Proposition~\ref{orient_prop}]\label{ls_cor}  
Let $(X,\phi)$, $(\Si,c)$, $g$, $\b$, $\k$, and $X_{\phi,c}$ be as before.
There is a local system $\cZ^{\tnd\phi,c}_{(w_1,w_2)}$ on $X_{\phi,c}$ 
such that the local system of orientations on the moduli space 
$\mf_{g,\k}(X,J,\b)^{\phi,c}$ is isomorphic to 
$\wt\ev^*\cZ^{\tnd\phi,c}_{(w_1,w_2)}$.
An isomorphism between the two systems is determined by trivializations~of 
\begin{enumerate}[label=(\arabic*),leftmargin=*]
\item $\La^{\top}_{\R}(TX^{\phi})$ over a basepoint of each component of $X^{\phi}$,
\item $TX^{\phi}\oplus 3\La^{\top}_{\R}TX^{\phi}$ over
representatives for free homotopy classes of loops in~$X^{\phi}$
(one representative for each homotopy class), and
\item $(TX,\tnd\phi)$ over representatives for free homotopy classes of maps $S^1\!\lra\!X$
intertwining $\phi$ and the antipodal involution~$\fa$ on~$S^1$.
\end{enumerate}
The effect on this isomorphism of the changes in the above trivializations is described
as in~\eref{orientprop_e}.
\end{cor}

\begin{proof}
Suppose first that $\Si$ is not a disk or a cylinder. 
The proof of Lemma~\ref{iso_lm} then applies with $\cD_c$ 
replaced by~$\cD_c^*$, once $f_t$ in the first part of the third paragraph 
is chosen to be a path in~$\cD_i$.
It follows that the statements and proofs of Lemmas~\ref{def_lm} and~\ref{lift_cor} 
and Theorem~\ref{main_thm} with $(\cJ_c,\cD_c)$ replaced by $(\cJ_c^*,\cD_c^*)$ hold as well.
Therefore, Corollary~\ref{orient_cor} and Propositions~\ref{evmap_prop} and~\ref{orient_prop} 
apply with $\cH_{g,\k}(X,\b)^{\phi,c}$ replaced by $\cH_{g,\k}^*(X,\b)^{\phi,c}$.
In light of Lemma~\ref{DM_lmm}, Corollary~\ref{ls_cor} follows from 
Proposition~\ref{orient_prop} with $\cH_{g,\k}(X,\b)^{\phi,c}$ replaced by 
$\cH_{g,\k}^*(X,\b)^{\phi,c}$
by the same argument as \cite[Corollary~1.8]{Ge} follows from \cite[Theorem~1.7]{Ge}.\\

\noindent
If $\Si$ is a disk or a cylinder, the general principles of the proof of 
\cite[Corollary~1.8]{Ge} still apply.
The orientation of $T\mf_{g,\k}(X,J,\b)^{\phi,c}$ at each point of $\mf_{g,\k}(X,J,\b)^{\phi,c}$
is determined by orientations for the index of a Cauchy-Riemann operator on a real bundle pair
and either
the appropriate Deligne-Mumford space or the automorphism group of the complex structures.
With our choices of $\cJ_c^*$ and $\cD_c^*$, the last two objects are canonically oriented.
\end{proof}

\noindent
By Corollary~\ref{ls_cor}, \cite[Corollary~1.6]{Ge}, and Lemma~\ref{w2equiv_lmm},
$\mf_{g,\k}(X,J,\b)^{\phi,c}$ is orientable~if
\begin{enumerate}[label=(\arabic*),leftmargin=*]
\item $X^{\phi}\!\subset\!X$ is orientable and
$w_2(X^{\phi})\!=\!\ka^2\!+\!\vp|_{X^{\phi}}$ for some \hbox{$\ka\!\in\!H^1(X^{\phi};\Z_2)$}
and \hbox{$\vp\!\in\!H^2(X;\Z_2)$}, 
\item and $w_2^{\La_{\C}^{\top}\tnd\phi}(\La_{\C}^{\top}TX)$ is a square class,
e.g.~if either $\pi_1(X)\!=\!0$ and $w_2(TX)\!=\!0$ or 
$\La_{\C}^{\top}(TX,\tnd\phi)$ admits a real square root;
\end{enumerate}
the first condition is not needed if $|c|_0\!=\!0$, while 
the second condition is not needed if $|c|_1\!=\!0$.
For example, these moduli spaces are orientable for
a smooth quintic hypersurface~$X$ in~$\P^4$ cut out by an equation with
real coefficients and with the involution~$\phi$ being the restriction 
of the standard involution~$\tau_4$ on~$\P^4$.
If 
$$w_2(X^{\phi})=\ka^2+\vp|_{X^{\phi}}
\qquad\hbox{for some}\quad 
\ka\!\in\!H^1(X^{\phi};\Z_2),~\vp\!\in\!H^2(X;\Z_2),$$
but $X^{\phi}$ is not orientable, and (2) above still holds,
the orientation system of $\mf_{g,\k}(X,J,\b)^{\phi,c}$ is a pull-back/push-down of 
several copies of the orientation system of the Lagrangian~$X^{\phi}$.
The  $|c|_1\!=\!0$ case of these results contains
\cite[Theorem~8.1.1]{FOOO} and \cite[Theorem 1.1]{Sol}.
However, the presence of $w_2$ in~\eref{zf_eq} means that in  general the local system of 
orientations on $\mf_{g,\k}(X,J,\b)^{\phi,c}$ is not the pull-back of a system on~$X$
or~$X^{\phi}$.\\
  
\noindent
As described in \cite[Section~1]{Ge}, the index bundles $D_{V,\ti\phi}$
over $\mf_{g,\k}(X,J,\b)^{\phi,c}$ constructed as in Remark~\ref{fam_rem}
are expected to play a prominent role in the computation of real Gromov-Witten
invariants of submanifolds, such as complete intersections in projective spaces.
Specifically, given $n\!\in\!\Z^+$ and an $m$-tuple $\a=(a_1,\ldots,a_m)$
of positive integers,  let
$$V_{n;\a}=\cO_{\P^n}(a_1)\oplus \ldots \oplus \cO_{\P^n}(a_m)\lra \P^n.$$
This bundle admits a natural lift $\ti\tau_{n;\a}$ of the standard involution~$\tau_n$ on~$\P^n$
given by the conjugation of each homogeneous component.
If $|\b|$ is sufficiently large, the operators $D_{V_{n;\a},\ti\tau_{n;\a}}$ over 
$\mf_{g,\k}(\P^n,\b)^{\phi,c}$ are surjective and their kernels form a vector bundle
$$\cV_{n;\a}\lra \mf_{g,\k}(\P^n,\b)^{\tau_n,c}$$
whose orientation bundle is $\det D_{V_{n;\a},\ti\tau_{n;\a}}$.
The euler class of this bundle, with suitable boundary conditions,
is expected to relate real Gromov-Witten invariants of a complete intersection $X_{n;\a}$
to real Gromov-Witten invariants of~$\P^n$.
The following corollary, which extends \cite[Corollary~1.10]{Ge}, suggests
that it may indeed be possible to integrate $e(\cV_{n;\a})$ over 
$\mf_{g,\k}(\P^n,\b)^{\tau_n,c}$ when $n\!-\!|\a|$ is odd.
In these cases, the moduli space $\mf_{g,\k}(X_{n;\a},\b)^{\tau_n,c}$ is oriented 
and in fact has a canonical orientation, 
constructed using the Euler sequence for $\P^n$ and the normal bundle sequence for~$X_{n;\a}$,
similarly to the proof of \cite[Corollary~1.10]{Ge};
see also the proof of \cite[Proposition~7.5]{Ge2} for similar results for compactified moduli spaces
with $\Si\!=\!D^2$.

\begin{cor}\label{some_cor}
Let $n\!\in\!\Z^+$, $m\!\in\!\Z^{\ge0}$, $\a\!\in\!(\Z^+)^m$ be such that 
$n\!-\!|\a|$ is odd and $(\Si,c)$, $\k$, and $\b$ for $(X,\phi)\!=\!(\P^n,\tau_n)$ be as before. 
If $|\b|$ is sufficiently large, the line bundles
$$\La^{\top}_{\R}\cV_{n;\a},~\La^{\top}_\R T\mf_{g,\k}(\P^n,\b)^{\tau_n,c}\lra 
\mf_{g,\k}(\P^n,\b)^{\tau_n,c}$$
are canonically isomorphic up to multiplication by $\R^+$ in each fiber.  
\end{cor}
  
\begin{proof}
By Proposition~\ref{orient_prop} and Corollary~\ref{ls_cor}, 
the local systems for the two line bundles are isomorphic to the
push-down/pull-back of the local systems corresponding to $T\P^n$ and~$V_{n;\a}$.
The action of a loop $\gm$ in $\mf_{g,\k}(\P^n,\b)^{\tau_n,c}$
on the last two local systems is described by~\eref{zf_eq}
with $V=T\P^n,V_{n;\a}$.
As shown in the proof of \cite[Corollary~1.10]{Ge},
the first sum in~\eref{zf_eq} is the same for $V\!=\!T\P^n,V_{n;\a}$.
Since $\pi_1(\P^n)\!=\!0$, by the last statement of \eref{w2equiv_lmm} this is 
also the case for the second sum if the usual second Stiefel-Whitney classes
of $T\P^n$ and $V_{n;\a}$ are the same; this is indeed so under our assumptions.
Thus, the push-down/pull-back of the local systems are the same, and so the two line
bundles are isomorphic.\\

\noindent
An isomorphism between the local systems for the two line bundles is induced by 
identifications of choices (1)-(3) in Proposition~\ref{orient_prop} for the two bundles.
The proof of \cite[Corollary~1.10]{Ge} describes such identifications for~(1) and~(2).
Identifications for~(3) are described similarly.
They are specified by the canonical, $\Z_2$-equivariant, trivialization over representatives 
for homotopy classes of loops $(S^1,c)\!\lra\!(X,\phi)$,
where $c\!:S^1\!\lra\!S^1$ is the antipodal map, 
of the vector bundle
$$(n\!+\!1)\cO_{\P^n}(1)\oplus V_{n;\a}
=(n\!+\!1)\cO_{\P^n}(1)\oplus\cO_{\P^n}(a_1)\oplus\ldots\oplus\cO_{\P^n}(a_m)$$
with the complex conjugation induced from the natural complex conjugation in $\cO_{\P^n}(1)$.
This canonical trivialization is obtained by taking either of the $\Z_2$-equivariant
trivializations of $\cO_{\P^n}(1)$ and using it to trivialize all of the bundle components.
The effect of changing the  trivialization of $\cO_{\P^n}(1)$ on the trivialization
of the entire bundle~is
$$(-1)^{n+1}\cdot(-1)^{a_1}\cdot\ldots(-1)^{a_m}=1,$$
by our assumption on~$\a$.
\end{proof}

\appendix

\section{Almost complex structures on bordered surfaces}
\label{realanal_app}

\noindent
In this appendix we show that every bordered Riemann surface $(\Si,\fJ)$ can be covered by 
$(\fJ,\fJ_0)$-holomorphic charts
$$\psi\!:(U,U\!\cap\!\prt\Si)\lra (W,W\!\cap\!\R),$$
where $U$ is an open subset of~$\Si$, $W$ is an open subset of $\bH$,
and $\fJ_0$ is the standard complex structure on~$\C$; see Corollary~\ref{sh_cor}.
We also show that every symmetric Riemann surface $(\hat\Si,\fJ,\si)$ can be covered by 
holomorphic charts that intertwine~$\si$ with the standard conjugation~$\si_0$ on~$\C$;
see Corollary~\ref{symsurf_cor}.
These statements are likely known, but we could not find them in the literature
and thus include them with proofs for the sake of completeness.

\begin{lem}\label{realanal_lem}
If $U$ is an open neighborhood of the origin in $\bH$ and $\fJ$ is an almost complex 
structure on~$U$, there exists a diffeomorphism
$$h\!: (U',U'\!\cap\!\R)\lra (W,W\!\cap\!\R)$$
between an open neighborhood of $0$ in $U$ and an open subset $W$ of~$\bH$
such that $\fJ\!=\!h^*\fJ_0$.
\end{lem}

\begin{proof}
There exist $a,b\!\in\!C^{\i}(U,\R)$  such that 
$$\fJ(x,y)\frac{\prt}{\prt x}=a(x,y)\frac{\prt}{\prt x}+b(x,y)\frac{\prt}{\prt y}\,.$$
By shrinking $U$ if necessary, it can be assumed that $b(x,y)\!\neq\!0$.
With
$$(s,t)=\big(b(0,0)x\!-\!a(0,0)y,y\big),$$ 
we find that 
\begin{gather*}
\fJ(s,t)\frac{\prt}{\prt s}=
\frac{b(0,0)a(x,y)-a(0,0)b(x,y)}{b(0,0)}\frac{\prt}{\prt s}
+\frac{b(x,y)}{b(0,0)}\frac{\prt}{\prt t}\,.
\end{gather*}
Thus, we can assume that
\BE{Jform_e}\fJ(s,t)\frac{\prt}{\prt s}=\al(s,t)\frac{\prt}{\prt s}
+\big(1+\be(s,t)\big)\frac{\prt}{\prt t}\EE
for some $\al,\be\!\in\!C^{\i}(U,\R)$ with $\al(0,0),\be(0,0)\!=\!0$.\\

\noindent
The condition $\fJ\!=\!h^*\fJ_0$ is equivalent to 
$$\tnd h\bigg(\fJ\frac{\prt}{\prt s}\bigg)=\fJ_0\tnd h\bigg(\frac{\prt}{\prt s}\bigg)\,,$$
if $h$ is a diffeomorphism.
With 
$$h(s,t)=\big(s+x(s,t),t+y(s,t)\big)$$ 
and $\fJ$ as in~\eref{Jform_e}, the latter condition is 
equivalent~to 
\BE{hcond_e}P\left(\!\!\!\begin{array}{c}x\\ y\end{array}\!\!\!\right)
+Q\left(\!\!\!\begin{array}{c}x\\ y\end{array}\!\!\!\right)=\ze,\EE
where
$$P\left(\!\!\!\begin{array}{c}x\\ y\end{array}\!\!\!\right)=
\left(\!\!\!\begin{array}{c}x_t\!+\!y_s\\ y_t\!-\!x_s\end{array}\!\!\!\right),\qquad
Q\left(\!\!\!\begin{array}{c}x\\ y\end{array}\!\!\!\right)=
\left(\!\!\!\begin{array}{c}\al x_s\!+\!\be x_t\\ \al y_s\!+\!\be y_t\end{array}\!\!\!\right),
\qquad \ze=-\left(\!\!\!\begin{array}{c}\al\\ \be\end{array}\!\!\!\right).$$
Let $\eta\!:\R\!\lra\![0,1]$ be a smooth function such that
$$\eta(r)=\begin{cases}1,&\hbox{if}~r\le1;\\
0,&\hbox{if}~r\ge2.
\end{cases}$$
We will view $D^2$ as the closure of $\bH$ in $S^2$, i.e.~as the upper hemisphere.
For each $\de\!\ge\!0$, define
$$\eta_{\de}\!: D^2\lra [0,1] \qquad\hbox{by}\qquad 
\eta_{\de}(z)=\eta(|z|/\de).$$
In particular, $\|\eta_{\de}\|_{C^1(D^2)}\le C/\de$ and 
\BE{etaest_e} \begin{split}
\|\eta_{\de}\gm f\|_{L^p_1(D^2)}
&\le \|\eta_{\de}\gm\|_{C^0(D^2)}\|f\|_{L^p_1(D^2)}+\|\eta_{\de}\gm\|_{L^p_1(D^2)}\|f\|_{C^0(D^2)}\\
&\le C_{\gm}(\de+\de^{2/p})\|f\|_{L^p_1(D^2)} \,;
\end{split}\qquad 
 \begin{split}
\forall~&p\!>\!2,~\gm\!=\!\al,\be,\\ &f\!\in\!L^p_1(D^2);
\end{split}
\EE
this estimate uses the vanishing of $\gm$ at the origin.
Since the standard operator
$$\dbar\!: \big\{\xi\!\in\!L^p_2(D^2;\C)\!:\,\xi|_{\R}\!\in\!C^0(\R;\R),~\xi(0)\!=\!0\big\}
\lra L^p_1\big(D^2;(T^*D^2)^{0,1}\big)$$
is an isomorphism, for all $\de\!>\!0$ sufficiently small there exists a unique 
$\xi_{\de}\!\in\!L^p_2(D^2;\C)$ such that 
$$(\Im\,\xi_{\de})\big|_{\R}\!=\!0,\qquad \xi_{\de}(0)=0, \qquad
-\fJ_0 \dbar\xi_{\de}+\eta_{\de}(Q\xi_{\de})\tnd\bar{z}=\eta_{\de}\ze\tnd\bar{z}.$$
Furthermore, 
\BE{xide_e}\|\xi_{\de}\|_{L^2_p(D^2)}\le C\|\eta_{\de}\ze\tnd\bar{z}\|_{L^2_p(D^2)}
\le C'\de^{2/p}.\EE
On the disk of radius $\de$ around the origin, $\xi_{\de}$ restricts to a solution
of~\eref{hcond_e}.
If $p\!>\!2$, \eref{xide_e} implies that $\xi_{\de}'$ is a continuous function and
$|\xi_{\de}'(0)|\le C''\de^{2/p}$.\\

\noindent
Thus, if $\de$ is sufficiently small, the restriction $(x,y)$ of $\xi_{\de}$ to 
a neighborhood~$U'$ of the origin induces a diffeomorphism $h$ satisfying $\fJ\!=\!h^*\fJ_0$.
The condition $(\Im\,\xi_{\de})|_{\R}\!=\!0$ corresponds to $y|_{t=0}\!=\!0$ 
and $h(U'\!\cap\!\R)\!\subset\!\R$.
\end{proof}

\begin{cor}\label{sh_cor}
Let $\Si$ be a bordered surface with an almost complex structure~$\fJ$.
For every $z\!\in\!\Si$, there exists a coordinate chart 
$$\psi\!:(U,U\!\cap\!\prt\Si)\lra (\bH,\R)$$
around $z$ so that $\psi^*\fJ_0\!=\!\fJ$.
The overlap map between any two such charts is a restriction of a bi-holomorphic map
between open subsets of~$\C$.
\end{cor}

\begin{proof}
If $z\!\not\in\!\prt\Si$, such a chart exists because $\fJ$ is integrable on 
$\Si\!-\!\prt\Si$ by Newlander-Nirenberg Theorem; 
in fact, the proof in the $z\!\in\!\prt\Si$ case can be easily adapted to this case.
If $z\!\in\!\prt\Si$, a smooth chart 
$$\psi\!:(U,U\!\cap\!\prt\Si,z)\lra (\bH,\R,0)$$
induces an almost complex structure $\fJ'$ on a neighborhood of the origin in~$\bH$.
Composing~$\psi$ with a diffeomorphism~$h$ provided by Lemma~\ref{realanal_lem}, 
we obtain a desired chart around~$z$.\\

\noindent
If $\psi\!:U\!\lra\!\bH$ and $\psi'\!:U'\!\lra\!\bH$ are two charts as 
in the statement of the corollary, 
$$\psi'\circ\psi^{-1}\!: \big(\psi(U\!\cap\!U'),\psi(U\!\cap\!U')\!\cap\!\R\big)
\lra \big(\psi'(U\!\cap\!U'),\psi'(U\!\cap\!U')\!\cap\!\R\big)$$
is a bi-holomorphic map.
By the Schwartz Reflection Principle,
\begin{gather*}
\big\{z\!\in\!\C:\,\{z,\bar{z}\}\!\cap\!\psi(U\!\cap\!U')\neq\eset\big\}
\lra \big\{z\!\in\!\C:\,\{z,\bar{z}\}\!\cap\!\psi'(U\!\cap\!U')\neq\eset\big\},\\
z\lra\begin{cases}\psi'(\psi^{-1}(z)),&\hbox{if}~z\!\in\!\psi(U\!\cap\!U');\\
\ov{\psi'(\psi^{-1}(\bar{z}))},&\hbox{if}~\bar{z}\!\in\!\psi(U\!\cap\!U');
\end{cases}\end{gather*}
is a bi-holomorphic map between open subsets of~$\C$.
\end{proof}

\begin{cor}\label{symsurf_cor}
Let $(\hat\Si,\fJ,\si)$ be a Riemann surface with an involution (and without boundary).
For every $z\!\in\!\hat\Si$, there exists a holomorphic coordinate chart $\psi\!:U\!\lra\!W\subset\!\C$
($W$ not necessarily connected) 
such that $\psi\!\circ\!\si\!=\!\si_0\!\circ\!\psi$.
\end{cor}

\begin{proof}
If $z\!\not\in\!\hat\Si^{\si}$ and $\psi\!:U\!\lra\!\C$ is any holomorphic chart around $z$ such that
$U\!\cap\!\si(U)\!=\!\eset$, the holomorphic chart
$$U\!\cup\!\si(U)\lra \C, \qquad
z\lra\begin{cases}\psi(z),&\hbox{if}~z\!\in\!U;\\
\ov{\psi(\si(z))},&\hbox{if}~\si(z)\!\in\!U;
\end{cases}$$
has the desired property.\\

\noindent
Suppose $z\!\in\!\hat\Si^{\si}$. Since $\hat\Si^{\si}\!\subset\!\hat\Si$ is a smooth one-dimensional
submanifold, there exists a smooth chart 
$$\psi\!:(U,U\!\cap\!\hat\Si^{\si},z)\lra(\C,\R,0);$$
we can assume that $\si(U)\!=\!U$.
This chart induces an almost complex structure $\fJ'$ on a neighborhood of the origin in~$\bH$.
Composing~$\psi$ with a diffeomorphism provided by Lemma~\ref{realanal_lem}, 
we obtain a diffeomorphism
$$\psi'\!: \big(\psi^{-1}(\bH),\psi^{-1}(\bH)\!\cap\!\si(\psi^{-1}(\bH))\big)
\lra (\bH,\R)$$
such that $\psi'^*\fJ_0\!=\!\fJ$.
The holomorphic chart
\begin{gather*}
U\lra\C, \qquad
z\lra\begin{cases}\psi'(z),&\hbox{if}~z\!\in\!\psi^{-1}(\bH);\\
\ov{\psi'(\si(z))},&\hbox{if}~z\!\in\!\si\big(\psi^{-1}(\bH)\big);
\end{cases}\end{gather*}
intertwines $\si$ and $\si_0$.
\end{proof}


\begin{thebibliography}{99}

\bibitem{AAHV} B.~Acharya, M.~Aganagic, K.~Hori, and C.~Vafa, 
\emph{Orientifolds, mirror symmetry and superpotentials}, 
hep-th/0202208

\bibitem{AG} N.~L.~Alling and N.~Greenleaf, 
{\it Foundations of the Theory of Klein Surfaces}, 
Lecture Notes in Mathematics 219, Springer-Verlag, 1971

\bibitem{BHH} I.~Biswas, J.~Huisman, and J.~Hurtubise, 
\emph{The moduli space of stable vector bundles over a real algebraic curve}, 
Math.~Ann.~347 (2010), no.~1, 201–-233
 
\bibitem{EES} T.~Ekholm, J.~Etnyre, and M.~Sullivan, 
\emph{Orientations in Legendrian contact homology and exact Lagrangian immersions}, 
Internat.~J.~Math.~16 (2005), no.~5, 453--532

\bibitem{Teh} M.~Farajzadeh Tehrani,
\emph{Counting genus zero real curves in symplectic manifolds},
math/1205.1809v2

\bibitem{FM} B.~Farb and D.~Margalit,  
\emph{A Primer on Mapping Class Groups}, 
Princeton Mathematical Series 49, 2012

\bibitem{FO} K.~Fukaya and K.~Ono,
\emph{Arnold Conjecture and Gromov-Witten Invariant},
Topology 38 (1999), no.~5, 933--1048

\bibitem{FOOO} K.~Fukaya, Y.-G.~Oh, H.~Ohta, and K.~Ono,
\emph{Lagrangian Floer Theory: Anomaly and Obstruction}, AMS~2009

\bibitem{Ge} P.~Georgieva,
\emph{The orientability problem in open Gromov-Witten theory},
Geom.~Top.~17 (2013), no.~4, 2485--2512
 
\bibitem{Ge2} P.~Georgieva,
\emph{Open Gromov-Witten disk invariants in the presence of anti-symplectic 
involution}, math/1306.5019
 
\bibitem{GZ}  P.~Georgieva and A.~Zinger,
\emph{The moduli space of maps with crosscaps:
the relative signs of the natural automorphisms}, preprint
 
 
\bibitem{Gr} M.~Gromov, \emph{Pseudoholomorphic curves in symplectic manifolds},  
Invent.~Math.~82 (1985), no~ 2, 307--347 

\bibitem{Hopf} H.~Hopf, 
{\it Fundamentalgruppe und zweite Bettische Gruppe}, German,
Comment.~Math.~Helv.~14 (1942), 257–-309

\bibitem{Huang} Y.-Z.~Huang, {\it Two-Dimensional Conformal Geometry and Vertex 
Operator Algebras}, Progress in Math.~148, Birkh\"auser 1997

\bibitem{KatzLiu} S.~Katz and M.~Liu, 
{\it Enumerative geometry of stable maps with Lagrangian boundary conditions and 
multiple covers of the disc},
Geom.~Topol.~Monographs 8 (2006), 1--47


\bibitem{LT}  J.~Li and G.~Tian, 
\emph{Virtual moduli cycles and Gromov-Witten invariants of general symplectic manifolds}, 
Topics in Symplectic \hbox{$4$-Manifolds},
47-83, First Int.~Press Lect.~Ser., I, Internat.~Press, 1998

\bibitem{Melissa} C.-C.~Liu,
\emph{Moduli of $J$-holomorphic curves with Lagrangian boundary condition and open
Gromov-Witten invariants for an $S^1$-pair},
math/0210257v2

\bibitem{Loo} E.~Looijenga, 
\emph{Smooth Deligne-Mumford compactifications by means of Prym level structures}, 
J.~Algebraic Geom.~3 (1994), 283--293 

\bibitem{Mas} G.~Massuyea, \emph{A short introduction to mapping class groups}, available at http://www-irma.u-strasbg.fr/~massuyea/talks/MCG.pdf

\bibitem{McSa94} D.~McDuff and D.~Salamon, 
\emph{$J$-Holomorphic Curves and Quantum Cohomology},
University Lecture Series~6, AMS,~1994


\bibitem{MS} D. McDuff and D.~Salamon, 
\emph{J-holomorphic Curves and Symplectic Topology},
Colloquium Publications~52, AMS, 2004

\bibitem{MiSa} J.~Milnor and J.~Stasheff, {\it Characteristic Classes}, 
Annals of Mathematics Studies, no.~76, 
Princeton University Press, Princeton, 1974

\bibitem{Mu2}  J.~Munkres, {\it Elements of Algebraic Topology},
Addison-Wesley, 1984

\bibitem{Mu}  J.~Munkres, {\it Topology: a first course},
2nd Ed., Pearson, 2000

\bibitem{psw} R.~Pandharipande, J.~Solomon, and J.~Walcher,
\emph{Disk enumeration on the quintic 3-fold}, 
J.~Amer.~Math.~Soc. 21 (2008), no.~4, 1169--1209


\bibitem{RT} Y.~Ruan and G.~Tian, \emph{A mathematical theory of quantum cohomology}, 
J.~Differential Geom.~42 (1995), no.~2, 259–-367

\bibitem{SV} S.~Sinha and C.~Vafa, 
\emph{SO and Sp Chern-Simons at large~N}, hep-th/0012136


\bibitem{Sol} J.~Solomon,  
\emph{Intersection theory on the moduli space of holomorphic curves with 
Lagrangian boundary conditions}, math/0606429

\bibitem{Ste} N.~Steenrod, \emph{Homology with local coefficients}, 
Ann.~Math.~44 (1943), no.~4, 610--627

\bibitem{Wal} J.~Walcher, \emph{Evidence for tadpole cancellation in the
topological string}, Comm.~Number Theory Phys.~3 (2009), no.~1, 111--172



\bibitem{detLB} A.~Zinger, 
\emph{The determinant line bundle for Fredholm operators: construction, properties, and classification}, math:1304.6368

\end{thebibliography}
\end{document}